\documentclass[12pt]{amsart}

\usepackage[matrix,arrow,curve,frame]{xy}
\usepackage{amsmath,amsthm,amssymb,enumerate}
\usepackage{latexsym}
\usepackage{amscd}
\usepackage[colorlinks=false]{hyperref}
\usepackage{euscript}
\usepackage{appendix}

\newtheorem{thm}{Theorem}[section]

\newtheorem{cor}{Corollary}[section]
\newtheorem{lemma}{Lemma}[section]

\theoremstyle{definition}

\newtheorem{conj}{Conjecture}[section]

\theoremstyle{remark}
\newtheorem{remark}{Remark}[section]

\numberwithin{equation}{section}

\def\C{{\bf C}}

\def\cT{{\cal T}}

\def\cG{{\cal G}}
\def\cP{{\cal P}}
\def\cA{{\cal A}}

\def\ra{\rightarrow}

\def\A{{\bf A}}

\def\C{{\bf C}}

\def\W{{\bf W}}

\def\cA{{\mathcal A}}
\def\cB{{\mathcal B}}
\def\cC{{\mathcal C}}
\def\cD{{\mathcal D}}
\def\cE{{\mathcal E}}

\def\cG{{\mathcal G}}
\def\cH{{\mathcal H}}
\def\cI{{\mathcal I}}
\def\cJ{{\mathcal J}}
\def\cK{{\mathcal K}}
\def\cL{{\mathcal L}}
\def\cM{{\mathcal M}}

\def\cP{{\mathcal P}}

\def\cR{{\mathcal R}}
\def\cS{{\mathcal S}}
\def\cT{{\mathcal T}}

\def\cV{{\mathcal V}}
\def\cW{{\mathcal W}}

\def\gg{{\mathfrak g}}
\def\gh{{\mathfrak h}}

\def\gl{{\mathfrak l}}

\def\gs{{\mathfrak s}}

\newfont{\german}{eufm10}

\begin{document}
\pagestyle{plain}

\title
{Universal two-parameter $\cW_{\infty}$-algebra and vertex algebras of type $\cW(2,3,\dots, N)$}

\author{Andrew R. Linshaw}

\address{Department of Mathematics, University of Denver}
\email{andrew.linshaw@du.edu}
\thanks{This work was partially supported by grants \#318755 and \#635650 from the Simons Foundation. I thank T. Arakawa, T. Creutzig, and F. Malikov for helpful discussions and comments on an earlier draft of this paper.}


{\abstract \noindent  We prove the longstanding physics conjecture that there exists a unique two-parameter $\cW_{\infty}$-algebra which is freely generated of type $\cW(2,3,\dots)$, and generated by the weights $2$ and $3$ fields. Subject to some mild constraints, all vertex algebras of type $\cW(2,3,\dots, N)$ for some $N$ can be obtained as quotients of this universal algebra. As an application, we show that for $n\geq 3$, the structure constants for the principal $\cW$-algebras $\cW^k(\gs\gl_n, f_{\text{prin}})$ are rational functions of $k$ and $n$, and we classify all coincidences among the simple quotients $\cW_k(\gs\gl_n, f_{\text{prin}})$ for $n\geq 2$. We also obtain many new coincidences between $\cW_k(\gs\gl_n, f_{\text{prin}})$ and other vertex algebras of type $\cW(2,3,\dots, N)$ which arise as cosets of affine vertex algebras or nonprincipal $\cW$-algebras.}

\maketitle
\section{Introduction}

Associated to a simple, finite-dimensional Lie algebra $\gg$, a nilpotent element $f \in \gg$, and a complex parameter $k$, is a vertex algebra $\cW^k(\gg,f)$ known as an {\it affine $\cW$-algebra}. These are among the most important and best-studied examples of vertex algebras in both the physics and mathematics literature. The first algebra of this kind other than the Virasoro algebra is the Zamolodchikov $\cW_3$-algebra \cite{Zam}, and is associated to $\gs\gl_3$ with its principal nilpotent element $f_{\text{prin}}$. It is of type $\cW(2,3)$, meaning that it has a minimal strong generating set consisting of a field in weights $2$ and $3$. Its structure is more complicated than that of affine vertex algebras since the operator product expansion (OPE) of the weight $3$ field with itself contains nonlinear terms. Similarly, $\cW^k(\gs\gl_n, f_{\text{prin}})$ was defined in \cite{FL} and is of type $\cW(2,3,\dots, n)$. For a general $\gg$, the definition of $\cW^k(\gg,f_{\text{prin}})$ via quantum Drinfeld-Sokolov reduction was given by Feigin and Frenkel in \cite{FFI}. This algebra is of type $\cW(d_1,\dots, d_m)$, where $d_1,\dots, d_m$ are the degrees of the fundamental invariants of $\gg$. The definition of $\cW^k(\gg,f)$ for an arbitrary nilpotent element $f$ is due to Kac, Roan, and Wakimoto \cite{KRW}, and is a generalization of the quantum Drinfeld-Sokolov reduction. Although the generating fields of $\cW^k(\gg,f)$ close nonlinearly under OPE, $\cW^k(\gg,f)$ is {\it freely generated}, meaning that it has a Poincar\'e-Birkhoff-Witt basis consisting of normally ordered monomials in the generators and their derivatives. Equivalently, it has the graded character of a differential polynomial ring, and its associated variety is an affine space \cite{ArII}.

We denote by $\cW_k(\gg,f)$ the simple quotient of $\cW^k(\gg,f)$ by its maximal proper graded ideal. In the case $f = f_{\text{prin}}$, it was conjectured by Frenkel, Kac and Wakimoto \cite{FKW} and proven by Arakawa \cite{ArII,ArIII} that for a nondegenerate admissible level $k$, $\cW_k(\gg,f_{\text{prin}})$ is $C_2$-cofinite and rational. These are known as {\it minimal models} and are a generalization of the Virasoro minimal models \cite{GKO}. There are other known $C_2$-cofinite, rational $\cW$-algebras, not all of which are at admissible levels; see for example \cite{ArI,AMI,Kaw,KWIV,CLIII}.

The principal $\cW$-algebras $\cW^k(\gg,f_{\text{prin}})$ have appeared prominently in several important problems in mathematics and physics including the Alday-Gaiotto-Tachikawa correspondence \cite{AGT,BFN,MO,SV}, and the quantum geometric Langlands program \cite{AF,CG,Fr,FG, GI,GII}. They are also closely related to the classical $\cW$-algebras which arose in the context of integrable hierarchies of soliton equations in the work of Adler, Gelfand, Dickey, Drinfeld, and Sokolov \cite{Ad,GD,Di,DS}. The Korteweg-de Vries hierarchy, which corresponds to the Virasoro algebra, was generalized by Drinfeld and Sokolov to an integrable hierarchy associated to any simple Lie algebra. The corresponding classical $\cW$-algebras are Poisson vertex algebras and can be realized as the quasi-classical limits of affine $\cW$-algebras \cite{FBZ}. For a general nilpotent element $f\in \gg$, $\cW^k(\gg,f)$ can also be regarded as a chiralization of the finite $\cW$-algebra $\cW^{\text{fin}}(\gg,f)$ \cite{DSKII}. These were defined by Premet \cite{Pre}, generalizing some examples that were originally studied by Kostant \cite{Kos}. We can view $\cW^{\text{fin}}(\gg,f)$ as a quantization of the ring of functions on the Slodowy slice $S_f\subseteq \gg \cong \gg^*$ associated to $f$. Similarly, $\cW^k(\gg,f)$ is a quantization of the ring of functions on the arc space of $S_f$, which is the inverse limit of the finite jet schemes \cite{AMII}.

\subsection{Universal two-parameter $\cW_{\infty}$-algebra}
It is a longstanding conjecture in the physics literature that there exists a unique two-parameter $\cW_{\infty}$-algebra of type $\cW(2,3,\dots)$ denoted by $\cW_{\infty}[\mu]$ in \cite{GGII}, which interpolates between all the type $A$ principal $\cW$-algebras $\cW^k(\gs\gl_n, f_{\text{prin}})$, in the following sense.

\begin{enumerate}
\item All structure constants appearing in the OPE algebra among the generators of $\cW_{\infty}[\mu]$ are continuous functions of the central charge $c$ and the parameter $\mu$.
\item  If we set $\mu = n$, there is a truncation at weight $n+1$ that allows all fields in weights $d \geq n+1$ to be eliminated in the simple quotient of $\cW_{\infty}[\mu]$, and this quotient is isomorphic to $\cW^k(\gs\gl_n, f_{\text{prin}})$ as a one-parameter vertex algebra.
\end{enumerate}

This conjecture appears in a number of papers including \cite{YW,BaKi,B-H,BS,GGII,ProI,ProII,PR}. Considerable evidence for the existence and uniqueness of $\cW_{\infty}[\mu]$ was given in \cite{YW}, and later in \cite{GGII}, by solving Jacobi identities in low weights. It was observed that the structure constants in the OPEs of the first few generators depend on two free parameters, and it was conjectured that the full OPE algebra is determined recursively and consistently from this data. In the quasi-classical limit, the existence of a Poisson vertex algebra of type $\cW(2,3,\dots)$ which interpolates between the classical $\cW$-algebras of $\gs\gl_n$ for all $n$, has been known for over twenty years; see \cite{KZ,KM}, and more recently \cite{DSKV}. It can be defined using an affinization of Feigin's $\gg\gl_{\lambda}$-algebra of matrices of complex size $\lambda$, which is a certain quotient of $U(\gs\gl_2)$ that interpolates between the Lie algebras $\gg\gl_n$ for all $n$ in an appropriate sense \cite{F}. Recently, $\cW_{\infty}[\mu]$ has become important in the conjectured duality between families of two-dimensional conformal field theories and higher spin gravity on three-dimensional anti-de Sitter space \cite{GGI,GGII,GH}. In \cite{HII}, an interpretation of the $\cW$-algebra of $\gs\gl_{-n}$ was given by formally replacing $n$ with $-n$ in the structure constants of $\cW_{\infty}[\mu]$, and a coset realization of this algebra was proposed. It was observed in \cite{B-H} that other algebras of type $\cW(2,3,\dots, N)$ such as the parafermion algebra $N^k(\gs\gl_2)$, should also arise as quotients of this universal algebra. We mention that $\cW_{\infty}[\mu]$ is closely related to a number of other algebraic structures that arise in very different contexts, including the deformed $\cW_{1+\infty}$-algebra, Cherednik's double affine Hecke algebra, the spherical elliptic Hall algebra, and the Yangian of $\widehat{\gg\gl}_1$; see for example \cite{AS,MO,SV,Ts}.

\subsection{Main result} In this paper we prove the existence and uniqueness of a vertex algebra $\cW(c,\lambda)$ which is freely generated of type $\cW(2,3,\dots)$, depending on two parameters $c$ and $\lambda$. Its quotients include {\it all} one-parameter vertex algebras of type $\cW(2,3,\dots, N)$ satisfying some mild hypotheses, including $\cW^k(\gs\gl_n, f_{\text{prin}})$ for $n\geq 3$. We use a different parameter $\lambda$ which is related to $\mu$ by 
\begin{equation}\label{intro:lambdaofmu} \lambda =  \frac{(\mu-1) (\mu+1)}{(\mu-2) (3 \mu^2  - \mu -2+ c (\mu + 2))}.\end{equation} This choice is not canonical but is natural because $\cW(c,\lambda)$ is then defined over the polynomial ring $\mathbb{C}[c,\lambda]$. In other words, all structure constants appearing in the OPE algebra of the generators lie in $\mathbb{C}[c,\lambda]$. The algebra is generated by a Virasoro field $L$ of central charge $c$ and a primary weight $3$ field $W^3$ which is normalized so that $\displaystyle W^3_{(5)} W^3 = \frac{c}{3} 1$. The remaining strong generators $W^i$ of weight $i \geq 4$ are defined inductively by $$W^i = W^3_{(1)} W^{i-1},\qquad i \geq 4.$$ 
We show that by imposing all Jacobi identities among the generators, the structure constants in the OPEs of $L(z) W^i(w)$ and $W^i(z) W^j(w)$ are uniquely determined as polynomials in $c$ and $\lambda$, for all $i$ and $j$. The idea of determining the OPEs among the generators of a vertex algebra by imposing Jacobi identities has appeared in a number of papers in the physics literature including \cite{KauWa,Bow,B-V,HI}. If all Jacobi identities can be solved, this procedure is enough to establish uniqueness, but to rigorously construct a vertex algebra in this way from generators and relations, it is necessary to invoke a deep result of \cite{DSKI} which we call the {\it De Sole-Kac correspondence}. This is an equivalence between the categories of freely generated vertex algebras and nonlinear Lie conformal algebras. Roughly speaking, the OPE algebra of a set of free generators for a vertex algebra determines a nonlinear Lie conformal algebra. Conversely, associated to a nonlinear Lie conformal algebra $A$ is its universal enveloping vertex algebra, which is a certain quotient of the tensor algebra of $A$, and is always freely generated. It will be important for us to relax the notion of nonlinear Lie conformal algebra by omitting a subset of Jacobi identities. The resulting structure is called a {\it degenerate} nonlinear conformal algebra in \cite{DSKI}. Its universal enveloping vertex algebra is still defined but need not be freely generated.

\subsection{Quotients of $\cW(c,\lambda)$ and the classification of vertex algebras of type $\cW(2,3,\dots, N)$}
The vertex algebra $\cW(c,\lambda)$ has a conformal weight grading $$\cW(c,\lambda) = \bigoplus_{n\geq 0} \cW(c,\lambda)[n],$$ where each $\cW(c,\lambda)[n]$ is a free $\mathbb{C}[c,\lambda]$-module and $\cW(c,\lambda)[0] \cong \mathbb{C}[c,\lambda]$. There is a symmetric bilinear form on $\cW(c,\lambda)[n]$ given by
$$\langle ,  \rangle_n : \cW(c,\lambda)[n] \otimes_{\mathbb{C}[c,\lambda]} \cW(c,\lambda)[n] \ra \mathbb{C}[c,\lambda],\qquad \langle \omega, \nu \rangle_n = \omega_{(2n-1)} \nu.$$ The level $n$ Shapovalov determinant $\text{det}_n \in \mathbb{C}[c,\lambda]$ is just the determinant of this form. It turns out that $\text{det}_n$ is nonzero for all $n$, and in Section \ref{section:voaring} it is shown that this is equivalent to the simplicity of $\cW(c,\lambda)$ as a vertex algebra over $\mathbb{C}[c,\lambda]$.

Let $p$ be an irreducible factor of $\text{det}_{N+1}$ and let $I = (p) \subseteq \mathbb{C}[c,\lambda] \cong \cW(c,\lambda)[0]$ be the corresponding ideal. Consider the quotient
$$\cW^I(c,\lambda) = \cW(c,\lambda) / I \cdot \cW(c,\lambda),$$ 
where $I \cdot \cW(c,\lambda)$ is the vertex algebra ideal generated by $I$. This is a vertex algebra over the ring $\mathbb{C}[c,\lambda]/I$, which is no longer simple. It contains a singular vector $\omega$ in weight $N+1$, which lies in the maximal proper ideal $\cI\subseteq \cW^I(c,\lambda)$ graded by conformal weight. If $p$ does not divide $\text{det}_{m}$ for any $m<N+1$, $\omega$ will have minimal weight among elements of $\cI$. Often, there exists a localization $R$ of $\mathbb{C}[c,\lambda]/I$ such that $\omega$ has the form \begin{equation} \label{sing:intro} W^{N+1} - P(L, W^3,\dots, W^{N-1}),\end{equation} in the localization $\cW^I_R(c,\lambda) = R \otimes_{\mathbb{C}[c,\lambda]/I} \cW^I(c,\lambda)$. Here $P$ is a normally ordered polynomial in the fields $L,W^3,\dots, W^{N-1}$, and their derivatives, with coefficients in $R$. If this is the case, there will exist relations $$W^m = P_m(L, W^3, \dots, W^N)$$ for all $m \geq N+1$ expressing $W^m$ in terms of $L, W^3,\dots, W^N$ and their derivatives. The simple quotient $\cW^I_R(c,\lambda) / \cI$ will then be of type $\cW(2,3,\dots, N)$. Conversely, we will show that any simple one-parameter vertex algebra of type $\cW(2,3,\dots, N)$ satisfying some mild hypotheses, can be obtained as the simple quotient of $\cW^I_R(c,\lambda)$ for some $I$ and $R$. This reduces the classification of such vertex algebras to the classification of prime ideals $I = (p) \subseteq\mathbb{C}[c,\lambda]$ such that $p$ divides $\text{det}_{N+1}$ but does not divide $\text{det}_m$ for $m<N+1$, and $\cW^I(c,\lambda)$ contains a singular vector of the form \eqref{sing:intro}, possibly after localizing. 

In addition to $\cW^k(\gs\gl_n, f_{\text{prin}})$, there are many other one-parameter vertex algebras of type $\cW(2,3,\dots, N)$ for some $N$. Here is a short list of examples, which is by no means exhaustive. 

\begin{enumerate}

\item The parafermion algebra $N^k(\gs\gl_2) = \text{Com}(\cH, V^k(\gs\gl_2))$. Here $V^k(\gs\gl_2)$ denotes the universal affine vertex algebra of $\gs\gl_2$, $\cH$ is the Heisenberg algebra corresponding to the Cartan subalgebra $\gh$, and the commutant means the subalgebra of $V^k(\gs\gl_2)$ which commutes with $\cH$. This of type $\cW(2,3,4,5)$ \cite{DLY}.

\item The coset of $V^k(\gg\gl_n)$ inside $V^k(\gs\gl_{n+1})$. We call this the algebra of {\it generalized parafermions} since in the case $n=1$ it is just $N^k(\gs\gl_2)$. We will show that it is of type $\cW(2,3,\dots, n^2+3n+1)$, which was conjectured in \cite{B-H}.

\item The coset of the Heisenberg algebra $\cH$ inside the Bershadsky-Polyakov algebra, which is the $\cW$-algebra associated to $\gs\gl_3$ with its nonprincipal nilpotent element. This coset is of type $\cW(2,3,4,5,6,7)$ \cite{ACLI}.

\item The coset of the Heisenberg algebra $\cH$ inside the $\cW$-algebra $\cW^k(\gs\gl_4, f_{\text{subreg}})$ associated to $\gs\gl_4$ with its subregular nilpotent element. This is of type $\cW(2,3,4,5,6,7,8,9)$ \cite{CLIII}.

\item The coset of $V^{k+1}(\gg\gl_{n-2})$ inside the $\cW$-algebra $\cW^k(\gs\gl_n, f_{\text{min}})$ associated to $\gs\gl_n$ with its minimal nilpotent element $f_{\text{min}}$, for $n\geq 3$. This is of type $\cW(2,3,\dots, n^2 - 2)$ \cite{ACKL}. 

\item The coset of $V^k(\gs\gl_n)$ inside $V^{k+1}(\gs\gl_n) \otimes L_{-1}(\gs\gl_n)$. We will show that it is of type $\cW(2,3,\dots, n^2+2n)$, which was conjectured in \cite{B-H}.
\end{enumerate}

Unlike $\cW^k(\gs\gl_n, f_{\text{prin}})$, the above algebras are not freely generated. In fact, it is a folklore conjecture that the $\cW^k(\gs\gl_n, f_{\text{prin}})$ are the {\it only} freely generated vertex algebras of type $\cW(2,3,\dots,N)$ for some $N$. Note that Example (5) is a generalization of Example (3), which is just the case $n=3$.  Also, Examples (3) and (4) are part of the family of cosets of $\cH$ inside the algebra $\cW^k(\gs\gl_n, f_{\text{subreg}})$ associated to $\gs\gl_n$ with its subregular nilpotent element $f_{\text{subreg}}$. Conjecturally, these cosets are of type $\cW(2,3,\dots, 2n+1)$, and are not freely generated. We also mention a class of vertex algebras known as $Y$-algebras which were recently introduced by Gaiotto and Rap\v{c}\'ak \cite{GR}. In a very recent paper \cite{PR}, Proch\'azka and Rap\v{c}\'ak have conjectured that these algebras are of type $\cW(2,3,\dots, N)$ for some $N$, and can be obtained as quotients of the universal $\cW_{\infty}$-algebra.

All the vertex algebras in Examples (1)-(6), as well as $\cW^k(\gs\gl_n, f_{\text{prin}})$ for $n\geq 3$, arise as quotients of $\cW^I_R(c,\lambda)$ for some prime ideal $I = (p) \subseteq\mathbb{C}[c,\lambda]$ and some localization $R$ of $\mathbb{C}[c,\lambda]/I$. We shall explicitly describe $I$ for Examples (1)-(4), as well as $\cW^k(\gs\gl_n, f_{\text{prin}})$. For Examples (5) and (6), as well as $\text{Com}(\cH, \cW^k(\gs\gl_n, f_{\text{subreg}}))$, we shall give a conjectural description of $I$. A remarkable feature of these ideals is that they are organized into infinite families that admit a uniform description, and the corresponding varieties $V(I) \subseteq\mathbb{C}^2$ are rational curves, possibly singular. We call $V(I)$ the {\it truncation curve} associated to the one-parameter vertex algebra arising as the simple quotient of $\cW^I_R(c,\lambda)$. We speculate that {\it all} truncation curves are rational curves, and are organized into similar infinite families.

The vertex algebras $\cW^I_R(c,\lambda)$ for $I = (p)$ are one-parameter families in the sense that the ring $R$ has Krull dimension $1$. It is also important to consider $\cW^I(c,\lambda)$ when $I\subseteq \mathbb{C}[c,\lambda]$ is a {\it maximal} ideal, which has the form $I = (c- c_0, \lambda- \lambda_0)$ for some $c_0, \lambda_0\in \mathbb{C}$. Then $\cW^I(c,\lambda)$ and its quotients are vertex algebras over $\mathbb{C}$. Given two maximal ideals $I_0 = (c- c_0, \lambda- \lambda_0)$ and $I_1 = (c - c_1, \lambda - \lambda_1)$, let $\cW_0$ and $\cW_1$ be the simple quotients of $\cW^{I_0}(c,\lambda)$ and $\cW^{I_1}(c,\lambda)$. There is a very simple criterion for $\cW_0$ and $\cW_1$ to be isomorphic. We must have $c_0 = c_1$, and if this central charge is $0$ or $-2$, there is no restriction on $\lambda_0, \lambda_1$. For all other values of the central charge, we must have $\lambda_0 = \lambda_1$. By contrast, if the parameter $\mu$, which is related to $\lambda$ by \eqref{intro:lambdaofmu}, is used instead of $\lambda$, there are three distinct values of $\mu$ which give rise to the same algebra for generic values of $c$ \cite{GGII}. This phenomenon is known as {\it triality}.

Suppose that $\cW^I_R(c,\lambda)$ and $\cW^J_S(c,\lambda)$ are distinct one-parameter quotients of $\cW(c,\lambda)$. A nontrivial isomorphism between the simple quotients of these algebras at a special value of the central charge will be called a {\it coincidence} throughout this paper. Our criterion for $\cW_0$ and $\cW_1$ to be isomorphic implies that aside from the coincidences at $c=0,-2$, all other coincidences between $\cW^I_R(c,\lambda)$ and $\cW^J_S(c,\lambda)$, correspond to intersection points of their truncation curves. This explains a number of recently discovered coincidences between simple vertex algebras of type $\cW(2,3,\dots, N)$ at special parameter values, and simple principal $\cW$-algebras of type $A$. For example, for all integers $k\geq 3$, the simple parafermion algebra $N_k(\gs\gl_2)$ is isomorphic to a principal $\cW$-algebra of $\gs\gl_k$ which is $C_2$-cofinite and rational \cite{ALY}. There are similar coincidences between the Heisenberg cosets in Examples (3) and (4) at values where they are $C_2$-cofinite and rational, and principal $\cW$-algebras of type $A$ \cite{ACLI,CLIII}. The method of proving these coincidences relies on rationality, although other coincidences of this kind among more general $\cW$-algebras and their cosets which are not necessarily rational, have been conjectured in \cite{ACKL,CS,C,CG}. Our main result provides a new and powerful way to establish coincidences among vertex algebras of type $\cW(2,3,\dots, N)$, which need not be rational or $C_2$-cofinite. For example, we shall classify all coincidences among the simple algebras $\cW_k(\gs\gl_n, f_{\text{prin}})$ for $n\geq 2$, which settles a conjecture of Gaberdiel and Gopakumar \cite{GGII}. We also classify all coincidences between $\cW_k(\gs\gl_n, f_{\text{prin}})$ and the simple quotients of the vertex algebras in Examples (1)-(4) above. Many of these coincidences are new. Finally, note that $\cW(c,\lambda)$ has many rational quotients, since $\cW_k(\gs\gl_n, f_{\text{prin}})$ is such a quotient for any nondegenerate admissible level $k$ \cite{ArIII}. Due to the above coincidences, all known rational quotients of $\cW(c,\lambda)$ are of the form $\cW_k(\gs\gl_n, f_{\text{prin}})$. An interesting question that we do not address is whether $\cW(c,\lambda)$ admits rational quotients that are {\it not} of this kind.

\subsection{Organization} This paper is organized as follows. In Section \ref{section:VOAs}, we review the basic definitions and examples of vertex algebras that we need. In Section \ref{section:kds} we review the De Sole-Kac correspondence between the categories of freely generated vertex algebras and nonlinear Lie conformal algebras. In Section \ref{section:voaring}, we discuss the notion of vertex algebras and nonlinear Lie conformal algebras over commutative rings, which is a straightforward generalization of the usual notions. In Section \ref{section:main}, we prove our main result, which is the existence and uniqueness of the vertex algebra $\cW(c,\lambda)$. In Section \ref{section:classification}, we discuss localizations and quotients of $\cW(c,\lambda)$. In Sections \ref{section:prinw} and  \ref{section:genpara}, we give the explicit generators for the ideals in $\mathbb{C}[c,\lambda]$ corresponding to $\cW^k(\gs\gl_n, f_{\text{prin}})$ and $\text{Com}(V^k(\gg\gl_n), V^k(\gs\gl_{n+1}))$. In Section \ref{section:shap}, we give conjectural generators for the ideals in $\mathbb{C}[c,\lambda]$ that correspond to two additional families, namely $\text{Com}(V^{k+1}(\gg\gl_{n-2}), \cW^k(\gs\gl_n, f_{\text{min}}))$ and $\text{Com}(\cH, \cW^k(\gs\gl_n, f_{\text{subreg}}))$. In Section \ref{section:coincidences}, we classify all coincidences among the simple algebras $\cW_k(\gs\gl_n, f_{\text{prin}})$, as well as the coincidences between $\cW_k(\gs\gl_n, f_{\text{prin}})$ and several other families of vertex algebras of type $\cW(2,3,\dots, N)$. Finally, in Section \ref{section:oneplus} we discuss a one-parameter deformation of the $\cW_{1+\infty}$-algebra with central charge $c$.

\section{Vertex algebras} \label{section:VOAs}
In this section, we define vertex algebras, which have been discussed from various different points of view in the literature  (see for example \cite{Bor,FLM,FHL,K,FBZ}). We will follow the formalism developed in \cite{LZ} and partly in \cite{LiII}. Let $V=V_0\oplus V_1$ be a super vector space over $\mathbb{C}$, $z,w$ be formal variables, and $\text{QO}(V)$ be the space of linear maps $$V\ra V((z))=\{\sum_{n\in\mathbb{Z}} v(n) z^{-n-1}|
v(n)\in V,\ v(n)=0\ \text{for} \ n>\!\!>0 \}.$$ Each element $a\in \text{QO}(V)$ can be represented as a power series
$$a=a(z)=\sum_{n\in\mathbb{Z}}a(n)z^{-n-1}\in \text{End}(V)[[z,z^{-1}]].$$ We assume that $a=a_0+a_1$ where $a_i:V_j\ra V_{i+j}((z))$ for $i,j\in\mathbb{Z}/2\mathbb{Z}$, and we write $|a_i| = i$.

For each $n \in \mathbb{Z}$, we have a nonassociative bilinear operation on $\text{QO}(V)$, defined on homogeneous elements $a$ and $b$ by
$$ a(w)_{(n)}b(w)=\text{Res}_z a(z)b(w)\ \iota_{|z|>|w|}(z-w)^n- (-1)^{|a||b|}\text{Res}_z b(w)a(z)\ \iota_{|w|>|z|}(z-w)^n.$$
Here $\iota_{|z|>|w|}f(z,w)\in\mathbb{C}[[z,z^{-1},w,w^{-1}]]$ denotes the power series expansion of a rational function $f$ in the region $|z|>|w|$. For $a,b\in \text{QO}(V)$, we have the following identity of power series known as the {\it operator product expansion} (OPE) formula.
 \begin{equation}\label{opeform} a(z)b(w)=\sum_{n\geq 0}a(w)_{(n)} b(w)\ (z-w)^{-n-1}+:a(z)b(w):. \end{equation}
Here $:a(z)b(w):\ =a(z)_-b(w)\ +\ (-1)^{|a||b|} b(w)a(z)_+$, where $a(z)_-=\sum_{n<0}a(n)z^{-n-1}$ and $a(z)_+=\sum_{n\geq 0}a(n)z^{-n-1}$. Often, \eqref{opeform} is written as
$$a(z)b(w)\sim\sum_{n\geq 0}a(w)_{(n)} b(w)\ (z-w)^{-n-1},$$ where $\sim$ means equal modulo the term $:a(z)b(w):$, which is regular at $z=w$. 

Note that $:a(w)b(w):$ is a well-defined element of $\text{QO}(V)$. It is called the {\it Wick product} or {\it normally ordered product} of $a$ and $b$, and it
coincides with $a(w)_{(-1)}b(w)$. For $n\geq 1$ we have
$$ n!\ a(z)_{(-n-1)} b(z)=\ :(\partial^n a(z))b(z):,\qquad \partial = \frac{d}{dz}.$$
For $a_1(z),\dots ,a_k(z)\in \text{QO}(V)$, the $k$-fold iterated Wick product is defined inductively by
\begin{equation}\label{iteratedwick} :a_1(z)a_2(z)\cdots a_k(z):\ =\ :a_1(z)b(z):,\qquad b(z)=\ :a_2(z)\cdots a_k(z):.\end{equation}
We often omit the formal variable $z$ when no confusion can arise.

A subspace $\cA\subseteq \text{QO}(V)$ containing $1$ which is closed under all the above products will be called a {\it quantum operator algebra} (QOA). We say that $a,b\in \text{QO}(V)$ are {\it local} if $$(z-w)^N [a(z),b(w)]=0$$ for some $N\geq 0$. Here $[,]$ denotes the super bracket. This condition implies that $a_{(n)}b = 0$ for $n\geq N$, so (\ref{opeform}) becomes a finite sum. Finally, a {\it vertex algebra} will be a QOA whose elements are pairwise local. This notion is well known to be equivalent to the notion of a vertex algebra in the sense of \cite{FLM}. 
In particular, there is always an injective QOA homomorphism 
$$\rho:\cA\hookrightarrow \text{QO}(\cA),\qquad\ a\mapsto\hat a,\qquad \hat
a(\zeta)b=\sum_{n\in\mathbb{Z}} (a\circ_n b)~\zeta^{-n-1},$$
and the quadruple of structures $(\cA,1,\partial, \rho)$ is a vertex
algebra as in \cite{FLM}. Conversely, if $(V,{\bf 1},D, Y)$ is
a vertex algebra, then $Y(V)\subseteq \text{QO}(V)$ is a
QOA whose elements are pairwise local. The linear isomorphism $V \cong Y(V)$ is called the {\it state-field correspondence}.

A vertex algebra $\cA$ is said to be {\it generated} by a subset $S=\{\alpha^i|\ i\in I\}$ if $\cA$ is spanned by words in the letters $\alpha^i$, and all products, for $i\in I$ and $n\in\mathbb{Z}$. We say that $S$ {\it strongly generates} $\cA$ if $\cA$ is spanned by words in the letters $\alpha^i$, and all products for $n<0$. Equivalently, $\cA$ is spanned by $$\{ :\partial^{k_1} \alpha^{i_1}\cdots \partial^{k_m} \alpha^{i_m}:| \ i_1,\dots,i_m \in I,\ k_1,\dots,k_m \geq 0\}.$$ 

Suppose that $S$ is an ordered strong generating set $\{\alpha^1, \alpha^2,\dots\}$ for $\cA$ which is at most countable. We say that $S$ {\it freely generates} $\cA$, if $\cA$ has a Poincar\'e-Birkhoff-Witt basis consisting of all normally ordered monomials 
\begin{equation} \label{freegen} \begin{split} & :\partial^{k^1_1} \alpha^{i_1} \cdots \partial^{k^1_{r_1}}\alpha^{i_1} \partial^{k^2_1} \alpha^{i_2} \cdots \partial^{k^2_{r_2}}\alpha^{i_2}
 \cdots \partial^{k^n_1} \alpha^{i_n} \cdots \partial^{k^n_{r_n}} \alpha^{i_n}:,\qquad 
 1\leq i_1 < \dots < i_n,
 \\ & k^1_1\geq k^1_2\geq \cdots \geq k^1_{r_1},\quad k^2_1\geq k^2_2\geq \cdots \geq k^2_{r_2},  \ \ \cdots,\ \  k^n_1\geq k^n_2\geq \cdots \geq k^n_{r_n},
 \\ &  k^{t}_1 > k^t_2 > \dots > k^t_{r_t} \ \ \text{if} \ \ \alpha^{i_t}\ \ \text{is odd}. 
 \end{split} \end{equation}

\subsection{Conformal structure} A conformal structure with central charge $c$ on a vertex algebra $\cA$ is a Virasoro vector $L(z) = \sum_{n\in \mathbb{Z}} L_n z^{-n-2} \in \cA$ satisfying
\begin{equation} \label{virope} L(z) L(w) \sim \frac{c}{2}(z-w)^{-4} + 2 L(w)(z-w)^{-2} + \partial L(w)(z-w)^{-1},\end{equation} such that in addition, $L_{-1} \alpha = \partial \alpha$ for all $\alpha \in \cA$, and $L_0$ acts diagonalizably on $\cA$. We say that $\alpha$ has conformal weight $d$ if $L_0(\alpha) = d \alpha$, and we denote the conformal weight $d$ subspace by $\cA[d]$. In all our examples, the conformal weight grading will be by $\mathbb{Z}_{\geq 0}$,
$$\cA = \bigoplus_{d\geq 0} \cA[d].$$

All vertex algebras in this paper will be purely even, that is, $|\alpha| = 0$ for all $\alpha \in \cA$. As a matter of notation, we say that an even vertex algebra $\cA$ is of type $$\cW(d_1,d_2,\dots)$$ if it has a minimal strong generating set consisting of one field in each conformal weight $d_1, d_2, \dots $. If $\cA$ is freely generated of type $\cW(d_1,d_2,\dots)$, it has graded character
\begin{equation} \label{gradedchar:ua} \chi(\cA, q) = \sum_{n\geq 0} \text{dim}(\cA[n]) q^n = \prod_{i\geq 1} \prod_{k\geq 0} \frac{1}{1-q^{d_i +k}}.\end{equation}

\subsection{Important identities} We recall some important identities that hold in any vertex algebra $\cA$.  For any fields $a,b,c \in \cA$, we have
\begin{equation} \label{deriv} (\partial a)_{(n)} b = -na_{(n-1)} b\qquad \forall n\in \mathbb{Z},\end{equation}
\begin{equation} \label{commutator} a_{(n)} b  = (-1)^{|a||b|} \sum_{p \in \mathbb{Z}} (-1)^{p+1} (b_{(p)} a)_{(n-p-1)} 1,\qquad \forall n\in \mathbb{Z},\end{equation}
\begin{equation} \label{nonasswick} :(:ab:)c:\  - \ :abc:\ 
=  \sum_{n\geq 0}\frac{1}{(n+1)!}\big( :(\partial^{n+1} a)(b_{(n)} c):\ +
(-1)^{|a||b|} (\partial^{n+1} b)(a_{(n)} c):\big)\ .\end{equation}
\begin{equation} \label{ncw} a_{(n)}
(:bc:) -\ :(a_{(n)} b)c:\ - (-1)^{|a||b|}\ :b(a_{(n)} c): \ = \sum_{i=1}^n
\binom{n}{i} (a_{(n-i)}b)_{(i-1)}c, \qquad \forall n \geq 0.
\end{equation}

Given fields $a,b,c$ and integers $m,n \geq 0$, the following identities are known as Jacobi relations of type $(a,b,c)$. 
\begin{equation} \label{jacobi} a_{(r)}(b_{(s)} c) = (-1)^{|a||b|} b_{(s)} (a_{(r)}c) + \sum_{i =0}^r \binom{r}{i} (a_{(i)}b)_{(r+s - i)} c.\end{equation}
For each triple of fields $a,b,c$, these identities are nontrivial only for finitely many choices of $r,s$.

\subsection{$C_2$-cofiniteness and rationality} Given a vertex algebra $\cV$, let $C(\cV)$ denote the vector space quotient of $\cV$ by the span of  elements of the form $a_{(-2)} b$ for all $a,b \in \cV$. By \eqref{commutator} and \eqref{nonasswick}, the normally ordered product on $\cV$ induces a commutative, associative product on $C(\cV)$, which is known as {\it Zhu's commutative algebra} \cite{Z}. If $C(\cV)$ is finite-dimensional as a vector space, $\cV$ is called {\it $C_2$-cofinite}. This condition was introduced by Zhu in \cite{Z}, and has the important consequence that $\cV$ has finitely many simple $\mathbb{Z}_{\geq 0}$-graded modules. A vertex algebra $\cV$ is called {\it rational} if every $\mathbb{Z}_{\geq 0}$-graded $\cV$-module is completely reducible.

\subsection{Affine vertex algebras}Given a simple, finite-dimensional Lie algebra $\gg$, the {\it universal affine vertex algebra} $V^k(\gg)$ is freely generated by fields $X^{\xi}$ which are linear in $\xi \in \gg$ and satisfy 
\begin{equation}
X^{\xi}(z) X^{\eta}(w) \sim k ( \xi, \eta) (z-w)^{-2} + X^{[\xi,\eta]}(w)(z-w)^{-1},
\end{equation}
where $(\cdot ,\cdot )$ denotes the normalized Killing form $\frac{1}{2h^{\vee}} \langle \cdot,\cdot \rangle$. For all $k\neq -h^{\vee}$, $V^k(\gg)$ has a conformal vector
\begin{equation} \label{sugawara} L^{\gg}  = \frac{1}{2(k+h^{\vee})} \sum_{i=1}^n :X^{\xi_i} X^{\xi'_i}: \end{equation} of central charge $$c = \frac{k\ \text{dim}(\gg)}{k+h^{\vee}}.$$ Here $\xi_i$ runs over a basis of $\gg$, and $\xi'_i$ is the dual basis with respect to $(\cdot,\cdot)$.

As a module over the affine Lie algebra $\widehat{\gg} = \gg[t,t^{-1}] \oplus \mathbb{C}$, $V^k(\gg)$ is isomorphic to the vacuum $\widehat{\gg}$-module. For generic values of $k$, the vacuum module is irreducible and $V^k(\gg)$ is a simple vertex algebra. At certain special values of $k$, $V^k(\gg)$ is not simple. We denote by $L_k(\gg)$ the  quotient of $V^k(\gg)$ by its maximal proper ideal graded by conformal weight.

\subsection{Affine $\cW$-algebras} Given a simple Lie algebra $\gg$ and a nilpotent element $f \in \gg$, there is a vertex algebra $\cW^k(\gg, f)$ known as an affine $\cW$-algebra. The standard construction is via the quantum Drinfeld-Sokolov reduction and its generalizations, and the reader is referred to \cite{FFI,KRW} for details. We briefly recall some basic features of the $\cW$-algebras of $\gs\gl_n$ associated to the principal, subregular, and minimal nilpotent elements.

\begin{enumerate} 
\item For $n\geq 3$, the principal $\cW$-algebra $\cW^k(\gs\gl_n, f_{\text{prin}})$ is freely generated of type $$\cW(2,3,\dots, n).$$ It is generated by the weights $2$ and $3$ fields \cite{ALY}, and the Virasoro element $L$ has central charge $$c = -\frac{(n-1) (n^2 + n k - n-1) (n^2 + k + n k)}{n + k}.$$

\item For $n\geq 3$, the subregular $\cW$-algebra $\cW^k(\gs\gl_n, f_{\text{subreg}})$ is freely generated of type $$\cW(1,2,3,\dots, n-1, n/2, n/2).$$ The weight $1$ field generates a Heisenberg subalgebra $\cH$, and the Virasoro field has central charge
$$c = -\frac{((k +n)(n-1)-n)((k +n)(n-2)n - n^2+1)}{k+n}.$$

\item For $n\geq 4$, the minimal $\cW$-algebra $\cW^k(\gs\gl_n, f_{\text{min}})$ is freely generated of type $$\cW \big( 1^{(n-2)^2}, 2, (3/2)^{2(n-2)}\big).$$ The $(n-2)^2$ fields in weight $1$ generate an affine vertex algebra $$V^{k+1}(\gg\gl_{n-2}) = \cH \otimes V^{k+1}(\gs\gl_{n-2}),$$ and the fields in weight $3/2$ are primary and transform under $\gg\gl_{n-2}$ as the sum of the standard representation $\mathbb{C}^{n-2}$ and its dual. The Virasoro element has central charge $$c = -\frac{5 k + 6 k^2 + 4 n + 5 k n - n^2 - k n^2}{k + n}.$$
\end{enumerate} 

The complete OPE algebra of $\cW^k(\gs\gl_n, f_{\text{prin}})$ is not known in general, although for $n\leq 5$ it has been worked out explicitly in the physics literature \cite{Zh}. On the other hand, the OPE algebra of all minimal $\cW$-algebras, including $\cW^k(\gs\gl_n, f_{\text{min}})$, is simpler and appears in \cite{KWIII}. By a theorem of Genra \cite{Gen}, $\cW^k(\gs\gl_n, f_{\text{subreg}})$ coincides with the Feigin-Semikhatov algebra $\cW^{(2)}_n$ at level $k$ \cite{FS}, and part of the OPE algebra appears in \cite{FS}.

\subsection{Coset construction} 
Given a vertex algebra $\cV$ and a subalgebra $\cA \subseteq \cV$, the {\it coset} or {\it commutant} of $\cA$ in $\cV$, denoted by $\text{Com}(\cA,\cV)$, is the subalgebra of elements $v\in\cV$ such that $$[a(z),v(w)] = 0,\qquad\forall a\in\cA.$$ This was introduced by Frenkel and Zhu in \cite{FZ}, generalizing earlier constructions in \cite{GKO,KP}. Equivalently, $v\in \text{Com}(\cA,\cV)$ if and only if $a_{(n)} v = 0$ for all $a\in\cA$ and $n\geq 0$. 
Note that if $\cV$ and $\cA$ have Virasoro elements $L^{\cV}$ and $L^{\cA}$, $\text{Com}(\cA,\cV)$ has Virasoro element $L = L^{\cV} - L^{\cA}$ as long as $ L^{\cV} \neq L^{\cA}$. 

\subsection{Principal $\cW$-algebras as cosets} 

Let $\gg$ be simply laced. There exists a diagonal homomorphism $$V^{k+1}(\gg) \rightarrow V^k(\gg)\otimes L_1(\gg),\qquad X^{\xi}(z) \mapsto X^{\xi}(z) \otimes 1 + 1 \otimes X^{\xi}(z),\qquad \xi \in \gg,$$ which descends to a homomorphism $L_{k+1}(\gg) \rightarrow L_k(\gg)\otimes L_1(\gg)$ for all admissible values of $k$. The following result was recently proven in \cite{ACLII}.

\begin{thm}  \label{ACLIImain} (\cite{ACLII}, Main Theorems 1 and 2)
Let $\gg$ be simply laced and let $k, k'$ be complex numbers related by 
\begin{equation*} k' +h^{\vee}=\frac{k+h^{\vee}}{k+h^{\vee}+1}.\end{equation*} \begin{enumerate}
\item For generic values of $k$, we have a vertex algebra isomorphism 
$$\cW^{k'}(\gg)\cong  \text{Com}(V^{k+1}(\gg),V^k(\gg)\otimes L_1(\gg)).$$ In particular, this isomorphism holds for all real numbers $k > -h^{\vee} -1$.
\item Suppose that $k$ is an admissible level for $\widehat{\gg}$. Then $k'$ is a nondegenerate admissible level for $\widehat{\gg}$ so that $\cW_{k'}(\gg)$ is a minimal series $\cW$-algebra, and we have a vertex algebra isomorphism $$\cW_{k'}(\gg)\cong  \text{Com}(L_{k+1}(\gg),L_k(\gg)\otimes L_1(\gg)).$$
\end{enumerate}
\end{thm}

This was conjectured in \cite{BBSSII} in the case of discrete series $\cW$-algebras, which correspond to the case $k \in \mathbb{N}$, in \cite{KWI,KWII} for arbitrary minimal series $\cW$-algebras, and in the case of generic level in \cite{FJMM}. In the case $\gg = \gs\gl_n$, the coset realization of $\cW^{k'}(\gs\gl_n, f_{\text{prin}})$ is very useful for our purposes because the weight $3$ field appears explicitly in \cite{BBSSII}.

\subsection{Coincidences} For later use, we mention some recently established coincidences between certain cosets of affine or $\cW$-algebras, and principal, rational $\cW$-algebras of type $A$. First, recall the universal parafermion algebra $N^k(\gs\gl_2) = \text{Com}(\cH, V^k(\gs\gl_2))$, as well as its simple quotient $N_k(\gs\gl_2) = \text{Com}(\cH, L_k(\gs\gl_2))$ \cite{DLY}.

\begin{thm} \label{thm:aly} (\cite{ALY}, Theorem 6.1) For all positive integers $k \geq 3$, $$N_k(\gs\gl_2) \cong \cW_{k'}(\gs\gl_k, f_{\rm{prin}}),\qquad k' = -k + \frac{2 + k}{1 + k}, \  -k + \frac{1 + k}{2 + k},$$
which has central charge $\displaystyle c = \frac{2 (k-1)}{k+2}$, and is rational and $C_2$-cofinite.
\end{thm}

Next, recall the {\it Bershadsky-Polyakov algebra}, which is just the $\cW$-algebra $\cW^k(\gs\gl_3, f_{\text{min}})$. As in \cite{ACLI}, it is convenient to introduce a shift of level; we define
$$\cW^{\ell} = \cW^{\ell - 3/2}(\gs\gl_3, f_{\text{min}}),$$ and we denote its simple quotient by $\cW_{\ell}$. The weight $1$ field generates a Heisenberg subalgebra $\cH$. We define $\cC^{\ell} = \text{Com}(\cH, \cW^{\ell})$ and $\cC_{\ell} = \text{Com}(\cH, \cW_{\ell})$, which is the simple quotient of $\cC^{\ell}$. 

\begin{thm} \label{thm:acl} (\cite{ACLI}, Theorem 8.3) For all positive integers $\ell$, $$\cC_{\ell} \cong \cW_{\ell'}(\gs\gl_{2\ell}, f_{\rm{prin}}),\qquad \ell' = -2 l + \frac{3 + 2 l}{1 + 2 l},\  -2 l + \frac{1 + 2 l}{3 + 2 l},$$
which has central charge $\displaystyle c = -\frac{3(2\ell-1)^2}{2\ell +3}$ and is $C_2$-cofinite and rational.
\end{thm}

Finally, recall the $\cW$-algebra $\cW^k(\gs\gl_4, f_{\text{subreg}})$, and its simple quotient $\cW_k(\gs\gl_4, f_{\text{subreg}})$. Let $\cC^k = \text{Com}(\cH, \cW^k(\gs\gl_4, f_{\text{subreg}}))$ and $\cC_k = \text{Com}(\cH, \cW_k(\gs\gl_4, f_{\text{subreg}}))$ which is the simple quotient of $\cC^k$.

\begin{thm} \label{thm:cl} (\cite{CLIII}, Theorem 6.2) Let $n\geq 3$ and assume that $n-1$ is coprime to at least one of $n+1$ and $n+4$. Then $$\cC_k \cong \cW_{k'}(\gs\gl_n, f_{\rm{prin}}), \qquad k = -4 + \frac{4 + n}{3}, \qquad k' = -n + \frac{4 + n}{1 + n},\  -n + \frac{1+ n}{4 + n},$$ which has central charge $\displaystyle c = -\frac{4 (n-1) (2n -1)}{4 + n}$ and is $C_2$-cofinite and rational.
\end{thm}

These theorems were proven by exhibiting $V_k(\gs\gl_2)$, $\cW_{\ell}$, and $\cW_k(\gs\gl_4, f_{\text{subreg}})$, respectively, as simple current extensions of $V_L \otimes \cW$. Here $V_L$ is a certain rank one lattice vertex algebra that extends $\cH$, and $\cW$ is the appropriate principal $\cW$-algebra of type $A$. This approach does not easily generalize if the vertex algebras involved are not rational, or if the larger vertex algebra is not a simple current extension of a tensor product of two nice subalgebras. One of the goals of this paper is to establish many more coincidences of this kind. For example, we shall classify all cases where the above cosets $N_k(\gs\gl_2)$, $\cC_{\ell}$, and $\cC_k$, are isomorphic to a principal $\cW$-algebra of type $A$. There are several families of coincidences which do not involve rational or $C_2$-cofinite vertex algebras.

\section{Nonlinear Lie conformal algebras and the De Sole-Kac correspondence} \label{section:kds}
We recall the main results of an important paper of De Sole and Kac \cite{DSKI}, which gives an equivalence between the categories of nonlinear Lie conformal algebras and freely generated vertex algebras. De Sole and Kac use the language of lambda-brackets rather than the equivalent language of OPEs, which we briefly explain. First, a {\it Lie conformal algebra} is a $\mathbb{Z}_2$-graded $\mathbb{C}[T]$-module $A$, where $|a|$ denotes the parity of an element $a\in A$. It is equipped with a $\mathbb{C}$-bilinear map $[-,-]: A\otimes A \mapsto \mathbb{C}[\lambda] \otimes A$ called the lambda-bracket, satisfying

\begin{equation} \label{lca:deriv} [Ta_{\lambda} b] = -\lambda [a_{\lambda} b],\qquad [a_{\lambda} Tb] = (\lambda + T)[a_{\lambda} b],\qquad \forall a,b \in A,\end{equation}

\begin{equation} \label{lca:commutator} [a_{\lambda} b] = - (-1)^{|a||b|} [b_{-\lambda - T} a], \qquad \forall a,b \in A.\end{equation}

\begin{equation} \label{lca:jacobi} [a_{\lambda} [b_{\mu} c]] - (-1)^{|a||b|} [b_{\mu}[a_{\lambda} c]] - [[a_{\lambda} b] _{\lambda + \mu} c] =0, \qquad \forall a,b,c \in A.\end{equation}

The operator $T$ is called the {\it translation operator}. Vertex algebras can be regarded as a special class of Lie conformal algebras with an additional bilinear operation $A\otimes A \ra A$ sending $(a,b) \mapsto :ab:$ called the {\it normally ordered product}, satisfying

\begin{equation} \label{lca:nonasswick}  :(:ab:)c: - :a(:bc:): \ = \ : \bigg( \int_0^T d\lambda\  a\bigg) [b_{\lambda} c]: + (-1)^{|a||b|}  : \bigg(\int_0^T d\lambda \ b\bigg) [a_{\lambda} c] :.\end{equation}
 
\begin{equation}\label{lca:commwick} :ab: - (-1)^{|a||b|} :ba:\ = \int_{-T}^0 d\lambda \ [a_{\lambda} b].\end{equation}
 
\begin{equation} \label{lca:ncw} [a_{\lambda} :bc:] = \ :[a_{\lambda} b] c: + (-1)^{|a||b|}  :b[a_{\lambda} c]: + \int_0^{\lambda} d\mu\ [[a_{\lambda} b]_{\mu} c].
\end{equation}
Note that \eqref{lca:commutator} and \eqref{lca:ncw} together imply
\begin{equation} \label{lca:ncwright} 
[:ab:_{\lambda} c] = \ :(e^{T \partial_{\lambda}} a) [b_{\lambda} c]: + (-1)^{|a||b|} :(e^{T\partial_{\lambda}} b) [a_{\lambda} c]: + (-1)^{|a||b|}  \int_0^{\lambda} d\mu\ [b_{\mu} [a_{\lambda-\mu} c]].
\end{equation}

The equivalence between vertex algebras and such Lie conformal algebras was proven by Bakalov and Kac \cite{BK}. In particular, $A$ can be given a vertex algebra structure as follows. For $a,b \in A$, we define products $a_{(n)} b$ for $n\geq 0$ by expanding the lambda-bracket
$$[a_{\lambda} b] = \sum_{n\geq 0} \frac{\lambda^n}{n!} a_{(n)} b.$$ This corresponds to the OPE relation $a(z) b(w) \sim \sum_{n\geq 0} (a_{(n)} b)(w)(z-w)^{-n-1}$, that is, the coefficient of $\lambda^n$ corresponds to $1/n!$ times the pole of order $n+1$ appearing in the OPE formula. The translation operator $T$ corresponds to $\partial = \frac{d}{dz}$, and for $n\leq -1$ we define $a_{(n)} b = \ :(T^{-n-1} a)b:$. The identities \eqref{deriv}-\eqref{jacobi} are then equivalent to \eqref{lca:deriv}-\eqref{lca:ncwright}. In particular, the Jacobi identity \eqref{jacobi} is equivalent to \eqref{lca:jacobi}.

The above formalism is not well adapted to studying the notion of strong generation in a vertex algebra, in which nonlinear terms in the generators and their derivatives can appear in the OPEs. The notion of a nonlinear Lie conformal algebra in \cite{DSKI} is designed for this purpose. First, a {\it nonlinear conformal algebra} is a $\mathbb{Z}_2$-graded $\mathbb{C}[T]$-module $A$ possessing a grading $$A = \bigoplus_{\Delta \in \Gamma \setminus \{0\}} A[\Delta]$$ for some semigroup $\Gamma$, and a $\mathbb{C}$-bilinear operation $[-_{\lambda} -]: A\otimes A \ra \mathbb{C}[\lambda] \otimes \cT(A)$, where $\cT(A)$ is the tensor algebra of $A$. It must satisfy 
\begin{equation} \label{kds:deriv} [Ta_{\lambda} b] = -\lambda [a_{\lambda} b],\qquad [a_{\lambda} Tb] = (\lambda + T)[a_{\lambda} b],\qquad \forall a,b \in A,\end{equation}
together with the following grading condition:
\begin{equation} \label{gradingcond} \Delta([a_{\lambda} b]) < \Delta(a) + \Delta(b).\end{equation} Here the $\Gamma$-grading on $A$ is extended to a grading $$\cT(A) = \bigoplus_{\Delta \in \Gamma \setminus \{0\}} \cT(A)[\Delta]$$ by additivity, so that $\Delta(1) = 0$ and $\Delta(a\otimes b) = \Delta(a) + \Delta(b)$. The $\mathbb{Z}_2$-grading also extends naturally to $\cT(A)$. We have the induced filtration $$\cT_{\Delta}(A) = \bigoplus_{\Delta' \leq \Delta} \cT(A)[\Delta']$$ on the tensor algebra $\cT(A)$. 

\begin{thm} \label{thm:dsklemma3.2} (\cite{DSKI}, Lemma 3.2)
Given a nonlinear conformal algebra $A$ as above, the lambda-bracket on $A$ can be extended uniquely to a map
$$L_{\lambda}: \cT(A) \otimes \cT(A) \rightarrow \mathbb{C}[\lambda] \otimes \cT(A),$$ that is, $L_{\lambda}(a,b) = [a_{\lambda} b]$ for $a,b \in A = \cT^1(A)$, and simultaneously, the normally ordered product
$$N: \cT(A) \otimes \cT(A) \rightarrow \cT(A)$$ can be defined uniquely so that $N(a,B) = a\otimes B$ for all $a\in A$ and $B \in \cT(A)$, and appropriate analogues of \eqref{lca:nonasswick}, \eqref{lca:ncw}-\eqref{lca:ncwright} hold. In particular, given $a,b \in A$, and $B,C \in \cT(A)$, we have
\begin{equation} \label{kds:additional} \begin{split} 
 N(a\otimes B, C) &  = N(a, N(B,C)) + N\bigg( \bigg( \int_0^T d\lambda \ a \bigg), L_{\lambda}(B,C)  \bigg) 
\\ & + (-1)^{|a| |b|} N \bigg( \bigg( \int_0^T d\lambda \ B \bigg), L_{\lambda}(a,C) \bigg),
\\ L_{\lambda}(a, b\otimes C) & = N(L_{\lambda}(a,b), C) + (-1)^{|a| |b|} N(b, L_{\lambda} (a,C) + \int_0^{\lambda} d\mu\ L_{\mu} (L_{\lambda}(a,b), C)),
\\  L_{\lambda}(a\otimes B, C) & = N (( e^{T\partial_{\lambda}} a), L_{\lambda}(B,C)) + (-1)^{|a||B|} N((e^{T\partial_{\lambda}} B, L_{\lambda}(a,C)) \\ & + (-1)^{|a||B|} \int_0^{\lambda} d\mu \ L_{\mu} (B, L_{\lambda - \mu}(a,C)).\end{split}
\end{equation}
We also have the grading conditions
\begin{equation} \label{kds:gradingcond} \Delta(N(B,C)) \leq \Delta(B) + \Delta(C),\qquad \Delta(L_{\lambda}(B,C)) < \Delta(B) + \Delta(C). \end{equation}
\end{thm}

Next, let $\cM_{\Delta}(A) \subseteq \cT_{\Delta}(A)$ denote the span of elements of the form 
\begin{equation} \label{kds-mdelta} B \otimes \bigg( (a\otimes b - (-1)^{|a||b|} b\otimes a) \otimes C - N\bigg(\int_{-T}^0 d \lambda \ L_{\lambda}(a,b),C\bigg) \bigg),\end{equation} for all $a,b\in A$ and $B,C \in \cT(A)$ such that $B \otimes a\otimes b\otimes C \in \cT_{\Delta}(A)$. Let $\cM(A) = \bigcup_{\Delta \in \Gamma} \cM_{\Delta}(A)$. This should be regarded as the $\Gamma$-graded ideal in $\cT(A)$ with respect to the product $N$, which is generated by the elements \begin{equation*} \begin{split} & a \otimes b -  (-1)^{|a||b|} b\otimes a - \int_{-T}^0 d \lambda\ L_{\lambda}(a,b) \\ & = N(a,b) -  (-1)^{|a||b|} N(b,a) - \int_{-T}^0 d \lambda\ L_{\lambda}(a,b),\qquad \forall \ a,b \in A.\end{split} \end{equation*}

We call $A$ a {\it nonlinear Lie conformal algebra} if it satisfies two additional axioms:
\begin{equation} \label{kds-skew} [a_{\lambda} b] = - (-1)^{|a||b|}  [b_{-\lambda - T} a],\qquad \forall \ a,b \in A,\end{equation} 
\begin{equation} \label{kds-jacobi} L_{\lambda}(a, L_{\mu}(b,c)) - (-1)^{|a||b|}  L_{\mu}(b, L_{\lambda}(a,c)) - L_{\lambda + \mu}(L_{\lambda}(a,b),c) \in \mathbb{C}[\lambda,\mu] \otimes M_{\Delta'}(A),\end{equation} for all $a,b,c \in A$ and some $\Delta' < \Delta(a) + \Delta(b) + \Delta(c)$. The identity \eqref{kds-jacobi} is called the Jacobi identity since it is analogous to \eqref{lca:jacobi}.

By Corollary 4.3 and Corollary 4.5 of \cite{DSKI}, $\cM(A)$ is invariant under the operator $T$, and $\mathbb{C}[\lambda] \otimes \cM(A)$ is a two-sided ideal of $\mathbb{C}[\lambda] \otimes \cT(A)$ under the bilinear operations $N$ and $L_{\lambda}$. Then $T$, as well as the operations $L_{\lambda}$ and $N$, descend to the quotient 
\begin{equation} \label{defofua} U(A) = \cT(A) / \cM(A).\end{equation} Let $\pi: \cT(A) \ra U(A)$ denote the quotient map, and write $:BC:\ = \pi(N(B,C))$, for $B, C \in \cT(A)$. Note that the identity \eqref{kds-jacobi} in $\cT(A)$ implies that the honest Jacobi identity 
\begin{equation} \label{kds-honestjacobi} L_{\lambda}(a, L_{\mu}(b,c)) - (-1)^{|a||b|}  L_{\mu}(b, L_{\lambda}(a,c)) - L_{\lambda + \mu}(L_{\lambda}(a,b),c)=0\end{equation} holds in $\mathbb{C}[\lambda, \mu]\otimes U(A)$, for all $a,b,c\in A$. The main result of \cite{DSKI} is the following

\begin{thm} \label{thm:dskmain} (\cite{DSKI}, Theorem 3.9) Let $A$ be a nonlinear Lie conformal algebra, and let $\{a_i|\ \in I\}$ be an ordered basis of $A$ compatible with the $\Gamma$-grading and $\mathbb{Z}_2$-grading on $A$. Then the set of monomials $$\{:a_{i_1} \cdots a_{i_n}: |\ i_1 \leq i_2 \leq \cdots \leq i_n,\qquad i_k < i_{k+1} \ \ \text{if} \ \ a_{i_k} \ \ \text{is odd}\},$$ is a basis of $U(A)$. There is a canonical vertex algebra structure on $U(A)$ called the universal enveloping vertex algebra of $A$, such that the vacuum vector is $\pi(1)$, and 
\begin{equation} \begin{split} & T(\pi(B)) = \pi(T(B)), \quad \forall \ B \in \cT(A),
\\ & :\pi(B) \pi(C):\  = \pi(N(B,C)), \quad \forall \ B,C \in \cT(A),
\\ & [\pi(B)_{\lambda} \pi(C)] = \pi (L_{\lambda}(B,C)), \quad \forall \ B,C \in \cT(A).
\end{split} \end{equation} \end{thm}

\begin{remark} In the terminology of \cite{DSKI}, $U(A)$ is freely generated by $\{a_i|\ i\in I\}$, but this has a slightly different meaning from the notion of free generation in Section \ref{section:VOAs}, which is now standard in the literature. The relationship between the notions is as follows. Since $A$ is a free $\mathbb{C}[T]$-module we may choose an ordered basis $\{\alpha^j|\ j\in J\}$ for $A$ as a $\mathbb{C}[T]$-module. Then the set $\{T^i \alpha^j|\ i\geq 0, \ j\in J\}$ freely generates $U(A)$ in the sense of \cite{DSKI}, and $\{\alpha^j|\ j\in J\}$ freely generates $U(A)$ in the sense of Section \ref{section:VOAs}, that is, as a differential algebra. From now on, we will always use the notion of free generation in Section \ref{section:VOAs}.

\end{remark}

The correspondence $A \mapsto U(A)$ is functorial,  and the vertex algebra $U(A)$ satisfies the following universal property.

\begin{thm} \label{thm:dskuniversal} (\cite{DSKI}, Proposition 3.11) Let $\phi$ be a homomorphism from a nonlinear Lie conformal algebra $A$ to a vertex algebra $\cV$, namely, a parity-preserving $\mathbb{C}[T]$-module homomorphism $\phi: A \ra \cV$ such that, if we extend it to a linear map $\phi: \cT(A) \ra \cV$ by setting $\phi(a_1\otimes \cdots \otimes a_s) = \ : \phi(a_1) \cdots \phi(a_s):$, we then have $\phi(L_{\lambda}(a,b)) = [\phi(a)_{\lambda} \phi(b)]$ for all $a,b\in A$. Then we have that:
\begin{enumerate}
\item $\phi(N(B,C)) = \ :\phi(B) \phi(C):$ and $\phi(L_{\lambda}(B,C)) = [\phi(B)_{\lambda} \phi(C)]$, for all $B,C \in \cT(A)$.
\item $\phi$ extends to a unique vertex algebra homomorphism $\phi: U(A) \ra \cV$.
\end{enumerate}
\end{thm}

In this paper, all generators will be even and we always take $\Gamma = \mathbb{Z}_{\geq 0}$. In this case, $U(A)$ will be graded by conformal weight if we declare $T^k(a)$ to have weight $\Delta(a) + k$. Suppose we choose an ordered basis $\{\alpha^1, \alpha^2,\dots\}$ for $A$ as a free $\mathbb{C}[T]$-module, where $\Delta(\alpha^i) = d_i >0$. Then $U(A)$ is freely generated by $\{\alpha^1, \alpha^2,\dots \}$, and hence is of type $\cW(d_1,d_2, \dots)$ with graded character \eqref{gradedchar:ua}. In this paper, we only consider examples where $U(A)$ has a Virasoro vector $L$ whose $L_0$-eigenspace decomposition coincides with the above conformal weight grading.
 
Conversely, let $\cV$ be a vertex algebra which is freely generated by a set $\{\alpha^1, \alpha^2,\dots \}$, with conformal weight grading $\cV = \bigoplus_{n\geq 0} \cV_n$ where $\alpha^i$ has conformal weight $d_i >0$. Then the free $\mathbb{C}[T]$-module generated by  $\{\alpha^1, \alpha^2,\dots \}$ is a nonlinear Lie conformal algebra $A$ with $\Delta(\alpha^i) = d_i$. By Theorem \ref{thm:dskuniversal}, the induced map $U(A) \ra \cV$ is a vertex algebra isomorphism. The functor $A \mapsto U(A)$ is therefore an equivalence between the categories of nonlinear Lie conformal algebras and freely generated vertex algebras, and we call this equivalence the {\it De Sole-Kac correspondence}. 

Theorem \ref{thm:dskuniversal} also implies that given a nonlinear Lie conformal algebra $A$ with ordered basis $\{\alpha^1, \alpha^2,\dots\}$ as a $\mathbb{C}[T]$-module, there exists a category $C_A$ whose objects are vertex algebras with strong generators $\{\alpha^1, \alpha^2,\dots \}$, and OPE relations determined by the lambda-bracket in $A$.  Theorem \ref{thm:dskuniversal} says that $U(A)$ is the initial object in $C_A$. If $U(A)$ is graded by conformal weight as above, $C_A$ has another universal object $S(A)$ which is the quotient of $U(A)$ by its maximal proper ideal graded by conformal weight. It is the unique simple graded object in $C_A$ and is a homomorphic image of any object in $C_A$.

It will be useful to slightly reformulate the De Sole-Kac correspondence. Suppose that $A$ is a nonlinear conformal algebra with generators 
$\{\alpha^1, \alpha^2,\dots\}$ as a free $\mathbb{C}[T]$-module with grading $\Delta(\alpha^i) = d_i >0$, such that in addition to \eqref{kds:deriv} and \eqref{gradingcond}, the identities \eqref{kds-skew} all hold for $a,b \in \{\alpha^1, \alpha^2,\dots\}$. In particular, the products $N$ and $L_{\lambda}$ on $\cT(A)$ are defined and they satisfy the identities \eqref{kds:additional} appearing in Theorem \ref{thm:dsklemma3.2}. We can define the graded ideal $\cM(A) \subseteq \cT(A)$ which is spanned in each $\Delta$-graded component by the elements \eqref{kds-mdelta}, and the quotient $U(A) = \cT(A) / \cM(A)$ will have a Poincar\'e-Birkhoff-Witt basis consisting of monomials in $\{T^j \alpha^i\}$. Then $U(A)$ is a vertex algebra (and in particular, \eqref{kds-honestjacobi} holds for all $a,b,c\in A$) if and only if \eqref{kds-jacobi} holds in $\cT(A)$ for all $a,b,c\in A$, or equivalently, for all $a,b,c \in \{\alpha^1, \alpha^2,\dots\}$. In other words, the Jacobi identities \eqref{kds-honestjacobi} in $U(A)$ must be consequences of \eqref{kds:deriv}, \eqref{gradingcond} and \eqref{kds-skew} alone, together with the identities \eqref{kds:additional} satisfied by the products $N$ and $L_{\lambda}$.

In this paper, we shall work with the OPE rather than lambda-bracket formalism, so we briefly explain how to translate the result of De Sole and Kac into the language of fields and OPEs. In view of the above reformulation, specifying a nonlinear Lie conformal algebra in this language means specifying fields $\{\alpha^1, \alpha^2,\dots\}$ where $\alpha^i$ has weight $d_i >0$, and pairwise OPEs $\displaystyle \alpha^i(z) \alpha^j(w) = \sum_{n\geq 0} (\alpha^i_{(n)} \alpha^j)(w)(z-w)^{-n-1}$ where each term $\alpha^i_{(n)} \alpha^j$ is a normally ordered polynomial in the generators and their derivatives of weight $d_i + d_j -n-1$, such that for all $a,b,c \in \{\alpha^1, \alpha^2,\dots\}$, \eqref{deriv}-\eqref{ncw} hold, and moreover, all Jacobi identities \eqref{jacobi} hold as consequences of \eqref{deriv}-\eqref{ncw} alone. This data uniquely determines the universal enveloping vertex algebra which is freely generated by $\{\alpha^1, \alpha^2,\dots\}$.

\subsection{Degenerate nonlinear conformal algebras}

As pointed out on the last page of \cite{DSKI}, it is interesting to relax the Jacobi identity \eqref{kds-jacobi}. A nonlinear conformal algebra $A$ satisfying \eqref{kds-skew} but not all of the identities \eqref{kds-jacobi} is called a {\it degenerate nonlinear conformal algebra}. The universal enveloping vertex algebra $U(A)$ can still be defined as a quotient of the tensor algebra $\cT(A)$. For simplicity, we restrict to case $\Gamma = \mathbb{Z}_{\geq 0}$. For $a,b,c\in \cT(A)$ with $\Delta(a) + \Delta(b) + \Delta(c) -2  = \Delta$, note that $$L_{\lambda}(a, L_{\mu}(b,c)) - (-1)^{|a||b|}  L_{\mu}(b, L_{\lambda}(a,c)) - L_{\lambda + \mu}(L_{\lambda}(a,b),c) \in \mathbb{C}[\lambda,\mu] \otimes \cT(A)_{\Delta}.$$ Let $\cJ_{\Delta}(A)$ be the span of the coefficients of $\lambda^i \mu^j$ in all such expressions, for all $i,j\geq 0$. Then we define $\cM'_{\Delta}(A) \subseteq \cT_{\Delta}(A)$ to be the span of $\cM_{\Delta}(A)$ together with the following elements:

\begin{enumerate}

\item $N(a,b)$ and $N(b,a)$ for all $a\in \cT_{\Delta_1}(A)$ and $b \in \cJ_{\Delta_2}(A)$ with $\Delta_1 + \Delta_2 = \Delta$.

\item The coefficient of $\lambda^i$ for all $i\geq 0$ in $L_{\lambda}(a,b)$ and $L_{\lambda}(b,a)$ for all $a\in \cT_{\Delta_1}(A)$ and $b \in \cJ_{\Delta_2}(A)$, with $\Delta_1 + \Delta_2-1 = \Delta$.
\end{enumerate}

Setting \begin{equation} \label{defofmprime} \cM'(A) =  \bigcup_{\Delta \in \Gamma} \cM'_{\Delta}(A),\end{equation}  it follows that $\mathbb{C}[\lambda]  \otimes \cM'(A)$ is a two-sided ideal of $\mathbb{C}[\lambda]  \otimes \cT(A)$ under the operations $N$ and $L_{\lambda}$. Setting \begin{equation} \label{defofua2} U(A) = \cT(A) / \cM'(A),\end{equation}
by a similar argument as in the case of nonlinear Lie conformal algebras, $U(A)$ has a vertex algebra structure. Note that if $A$ is a nonlinear Lie conformal algebra, $\cM(A) = \cM'(A)$, so this definition of $U(A)$ agrees with \eqref{defofua}. The assignment $A \mapsto U(A)$ is functorial, and $U(A)$ satisfies a universal property analogous to Theorem \ref{thm:dskuniversal}. This functor defines an equivalence between the category of degenerate nonlinear conformal algebras and an appropriate category of vertex algebras, which we continue to call the De Sole-Kac correspondence.

Suppose that $A$ has basis $\{\alpha^1, \alpha^2,\dots\}$ as a $\mathbb{C}[T]$-module, where $\Delta(\alpha^i) = d_i >0$. 
 Then $U(A)$ is strongly generated by the corresponding fields $\{\alpha^1, \alpha^2,\dots\}$, but is not freely generated unless $A$ is a nonlinear Lie conformal algebra. If we declare that $T^k(\alpha^i)$ has conformal weight $d_i+k$, then $U(A)$ will be graded by conformal weight. We can still consider the category $C_A$ consisting of vertex algebras graded by conformal weight with strong generators $\{\alpha^1, \alpha^2,\dots\}$, and OPE relations determined by $A$. Then $U(A)$ is a universal object in $C_A$ which admits a homomorphism to any other object in the category, and $C_A$ contains another universal object $S(A)$ which is the simple quotient of $U(A)$ by its maximal proper graded ideal.

Unlike the category of commutative rings, not every vertex algebra is the quotient of a freely generated vertex algebra with the same strong generating set. This is because Jacobi relations that fail to hold will give rise to {\it null fields}, which are elements of $\cM'(A)$ that correspond to nontrivial normally ordered polynomial relation among the generators and their derivatives. A typical example is the parafermion algebra $N^k(\gs\gl_2) = \text{Com}(\cH, V^k(\gs\gl_2))$, which is of type $\cW(2,3,4,5)$. These fields close nonlinearly and the OPE algebra can be found in the paper \cite{DLY}. Starting in weight $8$, there are Jacobi relations that do not hold as consequences of \eqref{deriv}-\eqref{ncw} alone, which give rise to nontrivial relations among the generators and their derivatives.

\section{Vertex algebras over commutative rings} \label{section:voaring}
Let $R$ be a finitely generated, commutative $\mathbb{C}$-algebra. A {\it vertex algebra over $R$} will be an $R$-module $\cA$ with a vertex algebra structure which we define as in Section \ref{section:VOAs}. A more comprehensive theory of vertex algebras over commutative rings has recently been developed by Mason \cite{Ma}, but the main difficulties are not present when $R$ is a $\mathbb{C}$-algebra. First, given an $R$-module $M$, we define $\text{QO}_R(M)$ to be the set of $R$-module homomorphisms $a: M \ra M((z))$. Such a homomorphism can be represented by a power series $$a(z) = \sum_{n\in \mathbb{Z}} a(n) z^{-n-1} \in \text{End}_R(M)[[z,z^{-1}]].$$ Here $a(n) \in \text{End}_R(M)$ is an $R$-module endomorphism, and for each $v\in M$, $a(n) v = 0$ for $n>\!\!>0$. Clearly $\text{QO}_R(M)$ is an $R$-module, and we define the products $a_{(n)} b$ as before, which are $R$-module homomorphisms from $\text{QO}_R(M) \otimes_R \text{QO}_R(M) \ra \text{QO}_R(M)$. A QOA will be an $R$-module $\cA \subseteq \text{QO}_R(M)$ containing $1$ and closed under all products. Locality is defined as before, and a vertex algebra over $R$ is a QOA $\cA\subseteq \text{QO}_R(M)$ whose elements are pairwise local. The OPE formula \eqref{opeform} still holds, as well as the fact that there is a faithful representation $\rho: \cA \hookrightarrow \text{QO}_R(\cA)$, which gives us the state-field correspondence.

In particular, we say that a set $S = \{\alpha^i|\ \ i\in I\}$ generates $\cA$ if $\cA$ is spanned as an $R$-module by all words in $\alpha^i$ and the above products. We say that $S$ strongly generates $\cA$ if $\cA$ is spanned as an $R$-module by all iterated Wick products of these generators and their derivatives. Note that we do {\it not} assume that $\cA$ is a free $R$-module. If $S = \{\alpha^1, \alpha^2,\dots\}$ is an ordered strong generating set for $\cA$ which is at most countable, we say that $S$ {\it freely generates} $\cA$, if $\cA$ has an $R$-basis consisting of all normally ordered monomials of the form \eqref{freegen}. In particular, this implies that $\cA$ is a free $R$-module.

\subsection{Conformal structure and graded character} Let $\cV$ be a vertex algebra over $R$ and let $c\in R$. Suppose that $\cV$ contains a field $L$ satisfying \eqref{virope}, such that $L_0$ acts on $\cV$ by $\partial$ and $L_1$ acts diagonalizably, and we have an $R$-module decomposition $$\cV = \bigoplus_{d\in R} \cV[d],$$ where $\cV[d]$ is the $L_0$-eigenspace with eigenvalue $d$. If each $\cV[d]$ is in addition a free $R$-module of finite rank, we have a well-defined graded character
$$\chi(\cV ,q) = \sum_{d \in R} \text{rank}_R(\cV[d]) q^d.$$
In all our examples, the grading will be by $\mathbb{Z}_{\geq 0}$ regarded as a subsemigroup of $R$, and $\cV[0] \cong R$. A typical example is $V^k(\gg)$ where $k$ is regarded as a formal variable, so $V^k(\gg)$ is a vertex algebra over the polynomial ring $\mathbb{C}[k]$. As such, it has no conformal vector, but if we define $R$ to be the localization $D^{-1} \mathbb{C}[k]$ where $D$ is the multiplicatively closed set generated by $(k + h^{\vee})$, then $V^k(\gg)$ has Virasoro vector $L^{\gg}$ given by \eqref{sugawara}.

In \cite{CLI,CLII}, a class of examples called {\it deformable families} were considered. These are vertex algebras over certain rings of rational functions of degree at most zero in a parameter $\kappa$. This notion was used to study orbifolds and cosets of affine vertex algebras and $\cW$-algebras by passing to the limit as $\kappa\ra \infty$, which has a simpler structure.

Here we allow $R$ to be the ring of functions on a variety $X \subseteq \mathbb{C}^n$, so that $X = \text{Specm}(R)$. We can think of the vertex algebra as being defined on $X$ in the sense that for each point $p\in X$, the evaluation at $p$ yields a vertex algebra over $\mathbb{C}$. There is also the notion of specialization along a subvariety. Let $I\subseteq R$ be an ideal corresponding to a subvariety $Y \subseteq X$, and let $I \cdot \cV$ denote the set of finite sums of the form $\sum_i f_i v_i$ where $f_i \in I$ and $v_i \in \cV$. Clearly $I \cdot \cV$ is the vertex algebra ideal generated by $I$, and the quotient
$$\cV^I = \cV / (I \cdot \cV)$$ is a vertex algebra over the ring $R/I$.

\subsection{Simplicity} Let $\cV$ be a vertex algebra over a ring $R$ with weight grading 
\begin{equation} \label{eq:gradedrvoa} \cV = \bigoplus_{n\geq 0} \cV[n],\qquad \cV[0] \cong R.\end{equation}
A vertex algebra ideal $\cI \subseteq \cV$ is called graded if $$\cI = \bigoplus_{n\geq 0} \cI[n],\qquad \cI[n] = \cI \cap \cV[n].$$ 
We say that $\cV$ is {\it simple} if for every proper graded ideal $\mathcal{I}\subseteq \mathcal{V}$, $\mathcal{I}[0] \neq \{0\}$. This is different from the usual notion of simplicity, namely, that $\cV$ has no nontrivial proper graded ideals, but it coincides with the usual notion if $R$ is a field. If $I \subseteq R$ is a nontrivial proper ideal, it will generate a nontrivial proper graded vertex algebra ideal $I\cdot \cV$. Then $$\mathcal{V}^I = \mathcal{V} / (I\cdot \mathcal{V})$$ is a vertex algebra over the ring $R/I$. Even if $\mathcal{V}$ is simple over $R$, $\mathcal{V}^I$ need not be simple over $R/I$.

Let $R$ be the ring of functions on a variety $X$, and let $\mathcal{V}$ be a simple vertex algebra over $R$ with weight grading \eqref{eq:gradedrvoa}. Let $I\subseteq R$ be an ideal such that $\cV^I$ is {\it not} simple. This means that $\cV^I$ possesses a maximal proper graded ideal $\mathcal{I}$ such that $\mathcal{I}[0] = \{0\}$. Then the quotient $$\cV_{I} = \cV^I / \cI$$ is a simple vertex algebra over $R/I$. Letting $Y\subseteq X$ be the closed subvariety corresponding to $I$, we can regard $\cV_I$ as a simple vertex algebra defined over $Y$. Continuing this process, we define the {\it degeneration poset} of all ideals $I\subseteq R$ for which $\cV^I$ is not simple over $R/I$. It corresponds to the poset of closed subvarieties of $X$ where degeneration occurs.

If $I_1, I_2$ are ideals in the degeneration poset, let $\mathcal{V}_{I_1} = \mathcal{V}^{I_1} / \mathcal{I}_1$ and $\mathcal{V}_{I_2} = \mathcal{V}^{I_2} / \mathcal{I}_2$ be the corresponding simple vertex algebras over the rings $R/I_1$ and $R / I_2$, respectively. Let $Y_1, Y_2 \subseteq X$ be the closed subvarieties corresponding to $I_1, I_2$, and let $p \in Y_1 \cap Y_2$ be a point in the intersection. Let $\mathcal{V}^p_{I_1}$ and $ \mathcal{V}^p_{I_2}$ be the vertex algebras over $\mathbb{C}$ obtained by evaluating at $p$. Then $\mathcal{V}^p_{I_1}$ and $ \mathcal{V}^p_{I_2}$ need not be simple, and we let $\mathcal{V}_{I_1,p}$ and $\mathcal{V}_{I_2,p}$ denote their simple quotients. Clearly $p$ corresponds to a maximal ideal $M \subseteq R$ containing both $I_1$ and $I_2$, and we have isomorphisms
$$\mathcal{V}_{I_1,p} \cong \mathcal{V}_M \cong \mathcal{V}_{I_2,p}.$$ 
In general, $\cV_{I_1}$ and $\cV_{I_2}$ can be very different, and the above isomorphism which arises from the intersection of the varieties $Y_1$ and $Y_2$, is nontrivial.

\subsection{Shapovalov determinant and spectrum} Given a conformal vertex algebra $\cV = \bigoplus_{n\geq 0} \cV[n]$ over $R$ with $\cV[0] \cong R$, $\cV[n]$ has a symmetric bilinear form \begin{equation} \label{bilinearform} \langle,\rangle_n: \cV[n]\otimes_{R} \cV[n] \ra R,\qquad \langle \omega,\nu \rangle_n = \omega_{(2n-1)}\nu.\end{equation}
If each $\cV[n]$ is a free $R$-module of finite rank, we define the level $n$ Shapovalov determinant $\text{det}_n \in R$ to be the determinant of the matrix of $\langle,\rangle_n$. Under some mild hypotheses, namely, $\cV[0] \cong R$ and $\cV[1]$ is annihilated by $L_1$, an element $\omega \in \cV[n]$ lies in the radical of the form $\langle,\rangle_n$ if and only of $\omega$ lies in the maximal proper graded ideal of $\cV$ \cite{LiI}. Then $\cV$ is simple if and only if $\text{det}_n \neq 0$ for all $n$. If $\cV$ is simple and $R$ is a unique factorization ring, each irreducible factor $a$ of $\text{det}_n$ give rise to a prime ideal $(a) \subseteq R$. Note that if $a|\text{det}_n$, then $a| \text{det}_m$ for all $m>n$. We define the {\it Shapovalov spectrum} $\text{Shap}(\cV)$  to be the set of distinct prime ideals of the form $(a) \subseteq R$ such that $a$ is a divisor of $\text{det}_n$ for some $n$. We say that $(a) \in \text{Shap}(\cV)$ has level $n$ if $a$ divides $\text{det}_n$ but does not divide $\text{det}_m$ for $m<n$. 

Note that the varieties $V(a)$ for $a\in \text{Shap}(\cV)$ are precisely the irreducible subvarieties of $\text{Specm}(R)$ where degeneration of $\cV$ occurs. For example, if $\cV = V^k(\gg)$ and $R = D^{-1} \mathbb{C}[k]$ where $D$ is the multiplicative set generated by $k + h^{\vee}$ as above, $\text{Shap}(V^k(\gg))$ is the set of all ideals $(k -k_0)$ where the vacuum module $V^{k_0}(\gg)$ is reducible. This was described explicitly by Gorelik and Kac in \cite{GK}.

\subsection{De Sole-Kac correspondence over rings}
In a similar way to \cite{DSKI}, one can define the notions of nonlinear conformal algebra and nonlinear Lie conformal algebra over a ring $R$. A nonlinear conformal algebra over $R$ will be a free $R[T]$-module $A$ with a grading $$A = \bigoplus_{\Delta \in \Gamma \setminus \{0\}} A[\Delta]$$ for some semigroup $\Gamma$, and an $R$-module homomorphism $$[-_{\lambda} -]: A\otimes_R A \ra R[\lambda] \otimes_R \cT(A)$$ satisfying \eqref{lca:deriv} and \eqref{gradingcond}. Here $\cT(A)$ is the tensor algebra of $A$ in the category of $R$-modules, which inherits the $\Gamma$-grading as before. Then the normally ordered product $$N: \cT(A) \otimes_R \cT(A) \ra \cT(A),$$ and the lambda-bracket
$$L_{\lambda}: \cT(A) \otimes_R \cT(A) \ra R[\lambda]\otimes_R \cT(A),$$ can be defined on $\cT(A)$ as in Section \ref{section:kds}, and the identities \eqref{kds:additional}-\eqref{kds:gradingcond} hold for all $a,b,c \in A$ and $B,C,D\in \cT(A)$.

Here we only need the case $\Gamma = \mathbb{Z}_{\geq 0}$. The subspaces $\cM_{\Delta}(A)$ and $\cM(A)$ are defined as in Section \ref{section:kds}. We call $A$ a {\it nonlinear Lie conformal algebra} over $R$ if it also satisfies \eqref{kds-skew} and \eqref{kds-jacobi}. As in Section \ref{section:kds}, we can define the universal enveloping vertex algebra $U(A) = \cT(A) / \cM(A)$, which is a vertex algebra over $R$. It is graded by conformal weight if we declare $T^k(a)$ to have weight $\Delta(a) + k$. We only consider examples where $U(A)$ contains a Virasoro element $L$ whose $L_0$-eigenspace grading coincides with the weight grading. If $\{\alpha^1,\alpha^2,\dots\}$ is an ordered basis of $A$ as a free $R[T]$-module, $U(A)$ is freely generated by the corresponding fields as a vertex algebra over $R$. In particular, each weighted component $U(A)[n]$ is a free $R$-module. If $\Delta(\alpha^i) = d_i > 0$, the graded character of $U(A)$ is given by \eqref{gradedchar:ua} with $\text{dim}(U(A)[n])$ replaced by $\text{rank}_R(U(A)[n])$.

Finally, suppose that $A$ is a degenerate nonlinear conformal algebra over $R$, meaning that it satisfies \eqref{kds-skew} but not \eqref{kds-jacobi}. Then the universal enveloping vertex algebra $U(A) = \cT(A) / \cM'(A)$ is still defined, where $\cM'(A)$ is defined as in Section \ref{section:kds}. If $\{\alpha^1,\alpha^2,\dots\}$ is an ordered basis for $A$ as a $R[T]$-module, $U(A)$ is strongly generated by the corresponding fields, but is not freely generated. It is still the direct sum of its weight graded pieces $U(A)[n]$, but since nontrivial normally ordered polynomial relations among the monomials hold, $U(A)[n]$ need not be a free $R$-module.

\section{Main construction}\label{section:main} 
Throughout this section we shall work over the polynomial ring $\mathbb{C}[c,\lambda]$ where $c,\lambda$ are formal variables. We wish to construct a nonlinear Lie conformal algebra $\cL(c,\lambda)$ over $\mathbb{C}[c,\lambda]$ with generators $\{L, W^i|\ i \geq 3\}$ and grading $\Delta(L) = 2$, $\Delta(W^i) = i$. Equivalently, the universal enveloping vertex algebra $$\cW(c,\lambda) = U(\cL(c,\lambda))$$ is freely generated by the corresponding fields, which we also denote by $\{L,W^i|\ i\geq 3\}$. 

We will first construct a nonlinear conformal algebra $\cL(c,\lambda)$ which satisfies \eqref{kds-skew} but only a subset of the Jacobi identities \eqref{kds-jacobi}. A priori, it is not clear that it satisfies all such identities and is a nonlinear Lie conformal algebra. The vertex algebra $\cW(c,\lambda)$ is therefore well defined but might not be freely generated. However, by considering a family of quotients of $\cW(c,\lambda)$ whose graded characters are known, we will show that $\cW(c,\lambda)$ must in fact be freely generated.

As mentioned in Section \ref{section:kds}, we prefer to work with the OPE rather than lambda-bracket formalism, so our goal will be to construct the OPE algebra of the fields $\{L, W^i|\ i\geq 3\}$ such that the identities \eqref{deriv}-\eqref{ncw}, together with a subset of Jacobi identities \eqref{jacobi}, hold. We then show that the remaining identities \eqref{jacobi} are forced to hold as well.

We postulate that $\cW(c,\lambda)$ has the following features. 

\begin{enumerate} 
\item It possesses a Virasoro field $L$ of central charge $c$ and an even weight $3$ primary field $W^3$, so that
\begin{equation} \begin{split} & L(z) L(w)  \sim \frac{c}{2} (z-w)^{-4} + 2 L(w)(z-w)^{-2} + \partial L(w)(z-w)^{-1}, \\  & L(z) W^3(w) \sim 3 W^3(w)(z-w)^{-2}+\partial W^3(w)(z-w)^{-1}.\end{split}\end{equation}
\item The fields $W^i$ are defined inductively by $$W^i = W^3_{(1)} W^{i-1},\qquad i\geq 4.$$ We have $$L(z) W^i(w) \sim \cdots + i W^i(w)(z-w)^{-2} + \partial W^i(w)(z-w)^{-1},$$ and the fields $\{L, W^i|\ i\geq 3\}$ close nonlinearly under OPE. In particular, $\cW(c,\lambda)$ is generated as a vertex algebra by $\{L,W^3\}$, and is strongly generated by $\{L, W^i|\ i\geq 3\}$.
\item The field $W^3$ is nondegenerate in the sense that $W^3_{(5)} W^3 \neq 0$. Without loss of generality we may normalize $W^3$ so that $$W^3_{(5)} W^3 = \frac{c}{3} 1.$$ 
\item There is a $\mathbb{Z}_2$-action on $\cW(c,\lambda)$ defined by $$L\mapsto L,\qquad W^3 \mapsto -W^3.$$ This forces $W^i \mapsto (-1)^i W^i$ for all $i\geq 3$, and each term appearing in $W^i_{(k)} W^j$ has weight $i+j-k-1$ and eigenvalue $(-1)^{i+j}$ under this symmetry.
\end{enumerate}

For notational convenience, we sometimes denote $L$ by $W^2$. The grading and symmetry assumptions have the following consequence. For $k = 0,1$, and $2\leq i\leq j$, $W^i_{(k)} W^j$ depends only on $L, W^3, \dots, W^{ i + j -2}$ and their derivatives. This is because all terms must have eigenvalue $(-1)^{i+j}$ under $\mathbb{Z}_2$. Then $\partial W^{i+j-3}$ cannot appear in $W^i_{(1)} W^j$, and neither $W^{i+j-1}$, $\partial^2 W^{i+j-3}$, nor $:L W^{i+j-3}:$, can appear in $W^i_{(0)} W^j$. Similarly,  for $k \geq 2$, $W^i_{(k)} W^j$ only depends on $L, W^3,\dots, W^{i + j -4}$ and their derivatives. 

In particular, for all $i,j$ such that $2 \leq i \leq j$, we can write
\begin{equation} \label{eq:aij} W^i_{(1)} W^{j} = a_{i,j}  W^{i+j-2} + C_{i,j},\end{equation} where $a_{i,j}$ denotes the coefficient of $W^{i+j-2}$ and $C_{i,j}$ is a normally ordered polynomial in $L, W^3, \dots, W^{i+j-4}$ and their derivatives. By assumption, $a_{2,j} = j$, $a_{3,j} = 1$, $C_{2,j} =0$, and $C_{3,j} = 0$ for all $j\geq 3$. Similarly,
\begin{equation}  \label{eq:bij} W^i_{(0)} W^{j} = b_{i,j} \partial W^{i+j-2} + D_{i,j},\end{equation} where $b_{i,j}$ denotes the coefficient of $\partial W^{i+j-2}$ and $D_{i,j}$ is a normally ordered polynomial in $L, W^3, \dots, W^{i+j-4}$ and their derivatives. By assumption, $b_{2,j} = 1$ and $D_{2,j} = 0$ for all $j\geq 2$.

We let $S_{n,k}$ denote the set of products $$\{W^i_{(k)} W^j |\ i+j = n,\ k \geq 0\},$$ which vanish for $k > n-1$. Our strategy is as follows. First, by imposing the identities \eqref{deriv}-\eqref{ncw}, and all Jacobi relations of type $(W^i, W^j, W^{\ell})$ for $i+j+\ell \leq 11$, this determines $S_{n,k}$ for all $n\leq 9$ and $k\geq 0$ uniquely. If we then assume $S_{m,k}$ to be known for $m\leq n$ and $k\geq 0$, we show that by imposing a subset of Jacobi relations of type $(W^i, W^j, W^{\ell})$ with $i+j+\ell = n+3$, this uniquely determines $S_{n+1,k}$. 

Note that we are not checking {\it all} Jacobi relations of type $(W^i, W^j, W^{\ell})$ for $i+j+\ell = n+3$ at this stage; we leave open the possibility that some of them might not vanish but instead give rise to nontrivial null fields. This has the effect that as we proceed through the induction, the OPEs $W^i(z) W^j(w)$ are uniquely determined modulo null fields that only depend on $L, W^3,\dots, W^{i+j-2}$ and their derivatives. For convenience, we suppress from our notation the fact that OPEs are determined up to null fields. At the end of the induction, we obtain the existence of a (possibly degenerate) nonlinear conformal algebra $\cL(c,\lambda)$ over $\mathbb{C}[c,\lambda]$ with generators $\{L, W^i|\ i\geq 3\}$, satisfying \eqref{kds-skew}. We then invoke the De Sole-Kac correspondence to conclude that the universal enveloping vertex algebra $\cW(c,\lambda)$ indeed exists. Since the lambda-brackets in $\cL(c,\lambda)$ are unique up to null fields, the OPE algebra of $\cW(c,\lambda)$ is unique.

\subsection{Step 1: $S_{n,k}$ for $n\leq 9$}
Since $W^4 = W^3_{(1)} W^3$ by definition, the most general OPE of $W^3$ with itself that is compatible with the $\mathbb{Z}_2$-symmetry is
\begin{equation*} \label{W3 generalform} \begin{split} W^3(z) W^3(w) & \sim \frac{c}{3}(z-w)^{-6} + a_0 L(w)(z-w)^{-4} + a_1 \partial L(w)(z-w)^{-3}\\ & + W^4(w)(z-w)^{-2}  + \big( a_2 \partial W^4  + a_3 \partial^3 L+a_4 :(\partial L)L: \big)(w) (z-w)^{-1}.\end{split} \end{equation*}
It is not difficult to check that imposing the Jacobi relations of type $(L,W^3,W^3)$ forces
\begin{equation}\label{ope:first} \begin{split} W^3(z) W^3(w) & \sim  \frac{c}{3}(z-w)^{-6} + 2 L(w)(z-w)^{-4} + \partial L(w)(z-w)^{-3} \\ & + W^4(w)(z-w)^{-2} + \bigg(\frac{1}{2} \partial W^4 -\frac{1}{12} \partial^3 L\bigg)(w)(z-w)^{-1}. \end{split} \end{equation}
The most general OPE of $L$ and $W^4$ compatible with the $\mathbb{Z}_2$-symmetry is
\begin{equation*}\begin{split} L(z) W^4(w) & \sim a_0 (z-w)^{-6} + a_1 L(w)(z-w)^{-4} + a_2 \partial L(w)(z-w)^{-3} \\ & + 4 W^4(w)(z-w)^{-2}+ \partial W^4(w)(z-w)^{-1}. \end{split} \end{equation*}
By imposing the Jacobi relations of type $(L,L,W^4)$ and $(L,W^3,W^3)$ this forces
\begin{equation}\label{ope:second} \begin{split}  L(z) W^4(w) & \sim 3 c (z-w)^{-6} + 10 L(w)(z-w)^{-4} + 3 \partial L(w)(z-w)^{-3} \\ & + 4 W^4(w)(z-w)^{-2} + \partial W^4(w)(z-w)^{-1}.\end{split} \end{equation}
Next, the most general OPE of $L$ and $W^5$ compatible with the $\mathbb{Z}_2$-symmetry is
\begin{equation*}\begin{split} L(z) W^5(w) & \sim b_0 W^3(z)(z-w)^{-4} +  b_1 \partial W^3(z)(z-w)^{-3}\\ & + 
5 W^5(w)(z-w)^{-2} + \partial W^5(w)(z-w)^{-1}.\end{split} \end{equation*}
Similarly, the most general OPE of $W^3$ and $W^4$ compatible with the $\mathbb{Z}_2$-symmetry is
\begin{equation*}\begin{split} W^3(z) W^4(w) & \sim  a_0 W^3(w)(z-w)^{-4} +a_1 \partial W^3(w)(z-w)^{-3}+ W^5(w)(z-w)^{-2} \\ & +\bigg(a_2 :L\partial W^3: + a_3 :(\partial L) W^3: +a_4 
\partial W^5 + a_5 \partial^3 W^3\bigg) (w)(z-w)^{-1},\end{split} \end{equation*}
It turns out that unlike \eqref{ope:first} and \eqref{ope:second}, if we impose all Jacobi identities of type $(L,W^3, W^4)$, $(L, L, W^5)$, and $(W^3, W^3, W^3)$ there is an additional free parameter $\lambda$ in the above OPEs. They have the form

\begin{equation} \label{ope:third}  \begin{split} L(z) W^5(w) & \sim \big(185 - 80 \lambda (2 + c)\big) W^3(z)(z-w)^{-4}  + \big(55 - 16\lambda (2 + c) \big) \partial W^3(z)(z-w)^{-3}\\ & + 
5 W^5(w)(z-w)^{-2} + \partial W^5(w)(z-w)^{-1}.\end{split} \end{equation}

\begin{equation} \label{ope:fourth} \begin{split} W^3(z) W^4(w) & \sim   \bigg(31-16 \lambda (2+c)\bigg) W^3(w)(z-w)^{-4}  + \frac{8}{3} \bigg(5 - 2\lambda (2 + c)  \bigg) \partial W^3(w)(z-w)^{-3} \\ & + W^5(w)(z-w)^{-2}  
+\bigg(\frac{2}{5}
\partial W^5 + \frac{32}{5} \lambda :L\partial W^3: -\frac{48}{5} \lambda :(\partial L) W^3: \\ & + \frac{2}{15} \bigg(-5 + 2 \lambda (-1 + c)\bigg)\partial^3 W^3\bigg) (w)(z-w)^{-1}.\end{split} \end{equation}

Similarly, letting $L = W^2$ and imposing all Jacobi identities of type $(W^i, W^j, W^k)$ for $i+j+k \leq 11$, all terms in the OPE of $W^i(z) W^j(w)$ for $2\leq i \leq j$ and $i+j \leq 9$ are uniquely determined as $\mathbb{C}[c, \lambda]$-linear combinations of normally ordered monomials in the generators $W^i$. These formulas are given in the appendix; see \eqref{ope:w2w6}-\eqref{ope:w4w5}.  By imposing the identities \eqref{deriv} and \eqref{commutator}, this uniquely determines the OPEs $$\partial^a W^i(z) \partial^b W^j(w),\qquad a,b\geq 0,\qquad i>j \geq 2,\qquad i+j \leq 9.$$

\subsection{Step 2: Induction}
We make the following inductive hypothesis. 
\begin{enumerate}
\item For $2 \leq i\leq j$ and $i+j \leq n$, all terms in the OPE of $W^i(z) W^j(w)$ have been expressed as $\mathbb{C}[c, \lambda]$-linear combinations of normally monomials in $L, W^3,\dots, W^{n-2}$ and their derivatives. The OPEs are weight homogeneous and compatible with $\mathbb{Z}_2$-symmetry, i.e., all terms appearing in $W^i_{(k)} W^j$ have weight $i+j-k-1$ and eigenvalue $(-1)^{i+j}$.

\item We impose \eqref{deriv} and \eqref{commutator}, which then determines all OPEs of the form $$\partial^a W^i(z) \partial^b W^j(w), \qquad a,b \geq 0,\qquad i+j \leq n.$$

\item The coefficients $a_{i,j}$ and $b_{i,j}$ are independent of $c, \lambda$ for $i+j \leq n$. Here $a_{i,j}$ and $b_{i,j}$ are given by \eqref{eq:aij} and \eqref{eq:bij}, respectively.
\end{enumerate}

By {\it inductive data}, we mean the collection of all OPE relations \begin{equation} \partial^a W^i(z) \partial^b W^j(w),\qquad a,b \geq 0,\qquad i+j \leq n.\end{equation} If we impose \eqref{nonasswick} and \eqref{ncw}, it is a consequence of our inductive hypothesis that if $\alpha$ is a normally ordered polynomial in the generators $L, W^3, \dots, W^{n-i}$ and their derivatives of weight $w \leq n+1-i$, the OPE $W^i(z) \alpha(w)$ is uniquely determined by the inductive data.

\begin{lemma} \label{main:1} If we impose the Jacobi identity 
\begin{equation}\label{main:1first} \begin{split} L_{(2)} (W^3_{(0)} W^{n-2}) & = (L_{(2)} W^3)_{(0)} W^{n-2} + W^3 _{(0)}( L_{(2)} W^{n-2}) \\ & + 2(L_{(1)} W^3)_{(1)} W^{n-2} + (L_{(0)} W^3)_{(2)} W^{n-2},\end{split} \end{equation} we must have $$b_{3,n-2} = \frac{2}{n-1}.$$ In particular, $b_{3,n-2}$ is independent of $c$ and $\lambda$.
\end{lemma}

\begin{proof} Recall that $$W^3_{(0)} W^{n-2} = b_{3,n-2} \partial W^{n-1} + D_{3,n-2}$$ where $D_{3,n-2}$ is a normally ordered polynomial in $L, W^3, \dots, W^{n-3}$ and their derivatives. Then $$ L_{(2)} (W^3_{(0)} W^{n-2}) =  b_{3,n-2} L_{(2)}  \partial W^{n-1} + L_{(2)} D_{3,n-2}.$$ Also, $L_{(2)}D_{3,n-2}$ has no terms depending on $W^{n-1}$ by inductive assumption, so it does not contribute to the coefficient of $\partial W^{n-1}$. We have 
$$L_{(2)} \partial W^{n-1} = -(\partial L)_{(2)} W^{n-1} + \partial (L_{(2)} W^{n-1}).$$ Note that $L_{(2)} W^{n-1}\in S_{n+1,2}$ and is therefore not yet known, but by weight and $\mathbb{Z}_2$-symmetry considerations it only depends on $L, W^3, \dots, W^{n-3}$ and their derivatives. Then $\partial (L_{(2)} W^{n-1})$ only depends on $L, W^3, \dots, W^{n-3}$. Modulo terms which only depend on $L, W^3, \dots, W^{n-3}$ and their derivatives, we have
$$L_{(2)} \partial W^{n-1} \equiv - (\partial L)_{(2)} W^{n-1} = 2 L_{(1)} W^{n-1} = 2(n-1) W^{n-1}.$$  So the left-hand side of \eqref{main:1first} is $2(n-1) b_{3,n-2} W^{n-1}$ up to terms which do not depend on $W^{n-1}$.

Next, the term $(L_{(2)} W^3)_{(0)} W^{n-2}$ from \eqref{main:1first} vanishes because $W^3$ is assumed primary. The term $$W^3 _{(0)}( L_{(2)} W^{n-2})$$ from \eqref{main:1first} has no contribution to the coefficient of $W^{n-1}$, since $L_{(2)} W^{n-2}$ only depends on $L, W^3, \dots, W^{n-4}$ and their derivatives. The term $$2(L_{(1)} W^3)_{(1)} W^{n-2}$$ from \eqref{main:1first} contributes $6 W^3_{(1)}W^{n-2} = 6 W^{n-1}$. The term $$(L_{(0)} W^3)_{(2)} W^{n-2} =  \partial W^3_{(2)} W^{n-2} =  -2 W^3_{(1)} W^{n-2} = -2 W^{n-1}.$$
We conclude that $$2(n-1) b_{3,n-2} = 4,$$ and hence the lemma follows. \end{proof}

\begin{lemma} \label{main:2}All coefficients $b_{i,n+1-i}$ for $3\leq i \leq \frac{n}{2}$ are independent of $c, \lambda$ and are determined uniquely by imposing Jacobi relations of type $(W^i, W^j, W^k)$ for $i+j+k = n+3$.
\end{lemma}

\begin{proof} We first impose the Jacobi relation
\begin{equation} \label{main:2first} W^3_{(1)} (W^3_{(0)} W^{n-3}) = (W^3_{(1)} W^3)_{(0)} W^{n-3} + W^3_{(0)} (W^3_{(1)} W^{n-3}) +  (W^3_{(0)} W^3)_{(1)} W^{n-3}.\end{equation}
Since $W^3_{(0)} W^{n-3} = \frac{2}{n-2} \partial W^{n-2} + D_{3,n-3}$, the left-hand side of \eqref{main:2first} is 

\begin{equation} \label{main:2second} \begin{split} W^3_{(1)} \bigg(\frac{2}{n-2} \partial W^{n-2} + D_{3,n-3} \bigg) & = -\frac{2}{n-2}(\partial W^3)_{(1)} W^{n-2} +\frac{2}{n-2} \partial  \bigg(W^3_{(1)} W^{n-2}\bigg) + W^3_{(1)}  D_{3,n-3}\\ & = \frac{2}{n-2} \bigg(\frac{2}{n-1} \partial W^{n-1} +D_{3,n-2}\bigg) + \frac{2}{n-2} \partial W^{n-1} +  W^3_{(1)}  D_{3,n-3}\\ & = \frac{2 (1 + n)}{(n-2) (n-1)} \partial W^{n-1} +  \frac{2}{n-2} D_{3,n-2} +  W^3_{(1)}  D_{3,n-3}.\end{split}\end{equation}

Next,
\begin{equation} \label{main:2third} \begin{split}(W^3_{(1)} W^3)_{(0)} W^{n-3} &= W^4_{(0)} W^{n-3}   = b_{4,n-3} \partial W^{n-1} + D_{4,n-3},
\\W^3_{(0)} (W^3_{(1)} W^{n-3}) & = W^3_{(0)} W^{n-2} = \frac{2}{n-1} \partial W^{n-1} + D_{3,n-2},
\\(W^3_{(0)} W^3)_{(1)} W^{n-3}  & = \bigg(\frac{1}{2} \partial W^4 - \frac{1}{12} \partial^3 L\bigg)_{(1)} W^{n-3} 
\\ & = -\frac{1}{2} W^4_{(0)} W^{n-3} = -\frac{1}{2} \bigg(b_{4,n-3} \partial W^{n-1} + D_{4,n-3}\bigg).\end{split}\end{equation} Collecting terms, we conclude that $$b_{4,n-3} = \frac{12}{(n-1)(n-2)}.$$

Inductively, we impose the Jacobi relation
\begin{equation} \label{main:2fourth}W^3_{(1)} (W^{i-1}_{(0)} W^{n+1-i}) = (W^3_{(1)} W^{i-1})_{(0)} W^{n+1-i} + W^{i-1}_{(0)} (W^3_{(1)} W^{n+1-i}) +  (W^3_{(0)} W^{i-1})_{(1)} W^{n+1-i}.\end{equation} The left side of \eqref{main:2fourth} is 

\begin{equation} \label{main:2fifth} \begin{split} W^3_{(1)} \bigg( b_{i-1,n+1-i} \partial W^{n-2} + D_{i-1,n+1-i}\bigg) = b_{i-1,n+1-i} W^3_{(1)} \partial W^{n-2} + W^3_{(1)}D_{i-1,n+1-i}
\\ = b_{i-1,n+1-i} \bigg(\frac{2}{n-1} \partial W^{n-1}+ D_{3,n-2} + \partial W^{n-1}\bigg) + W^3_{(1)}D_{i-1,n+1-i}. \end{split}\end{equation}
The right side of \eqref{main:2fourth} is
\begin{equation} \label{main2:sixth} b_{i,n+1-i} \partial W^{n-1} + D_{i,n+1-i} + b_{i-1, n+2-i} \partial W^{n-1} + D_{i-1,n+2-i}  -\frac{2}{i} \bigg(b_{i,n+1-i} \partial W^{n-1} + D_{i,n+1-i}\bigg).\end{equation}
Recall that $b_{i-1,n+1-i}$ and $D_{i-1,n+1-i}$ are part of our inductive data, and we are assuming inductively that $b_{i-1,n+1-i}$ is independent of $c, \lambda$. Since $D_{i-1,n+1-i}$ is a normally ordered polynomial of weight $n-1$ in $L, W^3, \dots, W^{n-4}$ and their derivatives, it follows that $W^3_{(1)} D_{n-1,n+1-i}$ is uniquely determined by inductive data. This shows that $b_{i,n+1-i}$ is uniquely determined from inductive data together with $b_{i-1,n+2-i}$ for $4\leq i < \frac{n}{2}$. Finally, since $b_{3,n-2}$ is independent of $c,\lambda$, it is clear that each $b_{i,n+1-i}$ is independent of $c,\lambda$. \end{proof}

Similarly, we will show that $\displaystyle a_{4,n-3} =\frac{4}{n-2}$ and that $a_{i,n+1-i}$ is determined from the Jacobi identities for all $i$; see Lemma \ref{main:4} below. Combining these observations, we obtain the following lemma.

\begin{lemma} \label{main:3} The products $W^i_{(0)} W^{n+1-i}$ for $3 \leq i \leq \frac{n}{2}$, are uniquely determined from inductive data, Jacobi relations of type $(W^i, W^j, W^k)$ for $i+j+k = n+3$, and elements of $S_{n+1,1}$.
\end{lemma}

\begin{proof} Since $b_{i,n+1-i}$ is determined from this data, it suffices to show that $D_{i,n+1-i}$ is also determined for $3\leq i \leq \frac{n}{2}$. It follows from \eqref{main:2second} and \eqref{main:2third} that modulo terms which are determined from inductive data,
\begin{equation} \label{main:3first} D_{4,n-3} \equiv -\frac{2 (n-4)}{n-2} D_{3,n-2}.\end{equation}
Using \eqref{main:2fifth} and \eqref{main2:sixth} in the case $i=5$, we get
\begin{equation} \label{main:3second}\frac{12}{(n-2)(n-3)} D_{3,n-2} \equiv D_{5,n-4} +  D_{4,n-3} - \frac{2}{5} D_{5,n-4},\end{equation} modulo inductive data.
Similarly, \eqref{main:2fifth} and \eqref{main2:sixth} show that there are nontrivial relations $D_{i,n+1-i} \equiv p_i(n) D_{3,n-2}$ for rational functions $p(n)$ for all $i$, modulo inductive data. Hence it is enough to find three linearly independent relations between $D_{3,n-2}$,  $D_{4,n-3}$, and $D_{5,n-4}$. From the Jacobi relation
$$W^3_{(0)} (W^4_{(1)} W^{n-4}) = (W^3_{(0)} W^4)_{(1)} W^{n-4} + W^4_{(1)} (W^3_{(0)} W^{n-4}),$$ we get 
$$ \frac{4}{n-3} D_{3,n-2} \equiv -\frac{2}{5} D_{5,n-4} + \frac{2}{n-3} D_{4,n-3} + \frac{2}{n-3} \partial C_{4,n-3},$$ modulo inductive data. Since $C_{4,n-3}$ is determined by $S_{n+1,1}$, we get the relation
\begin{equation} \label{main:3third} \frac{4}{n-3} D_{3,n-2} \equiv -\frac{2}{5} D_{5,n-4} + \frac{2}{n-3} D_{4,n-3}\end{equation} modulo inductive data together with data determined by $S_{n+1,1}$. Then \eqref{main:3first}-\eqref{main:3third} are the desired linearly independent relations.
\end{proof}

\begin{lemma} \label{main:4}
The set $S_{n+1,1}$ of products $W^i _{(1)} W^{n+1-i}$ is uniquely determined from inductive data, Jacobi relations of type $(W^i, W^j, W^k)$ for $i+j+k = n+3$, and the sets $S_{n+1,\ell}$ for $\ell\geq 2$.
\end{lemma}

\begin{proof} By assumption $L_{(1)} W^{n-1} = (n-1) W^{n-1}$ and $W^3_{(1)} W^{n-2} = W^{n-1}$, so we begin with the case $i\geq 4$. We impose the Jacobi relation
\begin{equation} \label{main:4first} W^{n-3} _{(1)} (W^3_{(1)} W^{3}) = W^3 _{(1)} (W^{n-3}_{(1)} W^{3}) + (W^{n-3} _{(1)} W^{3})_{(1)}W^3 +  (W^{n-3} _{(0)} W^3)_{(2)} W^{3}.\end{equation}

By \eqref{commutator}, the left-hand side of \eqref{main:4first} is $$W^{n-3}_{(1)} W^4 = \sum_{i\geq 0} (-1)^{i} \frac{1}{i!} \partial^{i}(W^4_{(i+1)} W^{n-3}).$$ Therefore modulo derivatives of elements of $S_{n+1,\ell}$ for $\ell \geq 2$, 
$$W^{n-3}_{(1)} W^4 \equiv W^4_{(1)} W^{n-3} = a_{4,n-3} W^{n-1} + C_{4,n-3}.$$

As for the right-hand side, modulo terms which are either derivatives of elements of $S_{n+1,\ell}$ for $\ell \geq 2$ or determined by inductive data, we have 
$$W^3 _{(1)} (W^{n-3}_{(1)} W^{3}) \equiv W^3 _{(1)} (W^{3}_{(1)} W^{n-3}) =W^3 _{(1)} W^{n-2} = W^{n-1},$$
$$(W^{n-3} _{(1)} W^{3})_{(1)}W^3 \equiv (W^{3} _{(1)} W^{n-3})_{(1)}W^3 \equiv W^{n-2}_{(1)} W^3 \equiv W^{n-1},$$
$$(W^{n-3} _{(0)} W^3)_{(2)} W^{3} =  \bigg(\sum_{i\geq 0} (-1)^{i+1} \frac{1}{i!} \partial^i (W^3_{(i)} W^{n-3}) \bigg)_{(2)} W^3 $$
$$= \bigg(- \frac{2}{n-2} \partial W^{n-2} + \partial W^{n-2}\bigg)_{(2)} W^3  = \bigg(\frac{4}{n-2} -2 \bigg)W^{n-2}_{(1)} W^3$$
$$\equiv  \bigg(\frac{4}{n-2} -2 \bigg)W^{3}_{(1)} W^{n-2} = \bigg(\frac{4}{n-2} -2 \bigg) W^{n-1}.$$
Collecting terms, we see that $$a_{4,n-3} =\frac{4}{n-2},$$ and that $C_{4,n-3}$ is uniquely determined modulo inductive data and derivatives of elements of $S_{n+1,\ell}$ for $\ell \geq 2$.

Next, we impose the Jacobi relation \begin{equation} \label{main:4second} 
W^3 _{(1)} (W^4_{(1)} W^{n-4}) = W^4 _{(1)} (W^3_{(1)} W^{n-4}) + (W^3 _{(1)} W^4)_{(1)} W^{n-4} +  (W^3 _{(0)} W^4)_{(2)} W^{n-4}.\end{equation} Since $\displaystyle W^4_{(1)} W^{n-4} =  \frac{4}{n-3} W^{n-2} + C_{4,n-4}$ where $C_{4,n-4}$ depends only on $L,W^3,\dots, W^{n-4}$ and their derivatives and is determined by inductive data, the left-hand side of \eqref{main:4second} is
$$ W^3_{(1)} \bigg(\frac{4}{n-3} W^{n-2} + C_{4,n-4}\bigg) \equiv   \frac{4}{n-3} W^{n-1}.$$ For the right-hand side, we have
$$W^4 _{(1)} (W^3_{(1)} W^{n-4}) = W^4 _{(1)} W^{n-3} = \frac{4}{n-2} W^{n-1} + C_{4,n-3},$$
$$(W^3 _{(1)} W^4)_{(1)} W^{n-4} = W^5_{(1)} W^{n-4} = a_{5,n-4} W^{n-1} + C_{5,n-4},$$
$$(W^3 _{(0)} W^4)_{(2)} W^{n-4} = \bigg(\frac{2}{5}
\partial W^5 + \frac{32}{5} \lambda :L\partial W^3: -\frac{48}{5} \lambda :(\partial L) W^3:  $$ $$+ \frac{2}{15} \big(-5 + 2 \lambda (-1 + c)\big)\partial^3 W^3\bigg)_{(2)} W^{n-4}$$
$$ \equiv -\frac{4}{5} \bigg(a_{5,n-4} W^{n-1} + C_{5,n-4}\bigg).$$
It is immediate that $$a_{5,n-4} = \frac{20}{(n-2) (n-3)},$$ and that $C_{5,n-4}$ is determined uniquely modulo inductive data and derivatives of elements of $S_{n+1,\ell}$ for $\ell \geq 2$.
Similarly, for $i>4$, by imposing the Jacobi relation 
\begin{equation} \label{main:4third} 
W^3 _{(1)} (W^i_{(1)} W^{n-i}) = W^i _{(1)} (W^3_{(1)} W^{n-i}) + (W^3 _{(1)} W^i)_{(1)} W^{n-i} +  (W^3 _{(0)} W^i)_{(2)} W^{n-i},\end{equation} the same argument shows that both $a_{i+1,n-i}$ and $C_{i+1,n-i}$ are uniquely determined modulo inductive data and derivatives of elements of $S_{n+1,\ell}$ for $\ell \geq 2$. This shows that $S_{n+1,1}$ is uniquely determined modulo this data.
 \end{proof}

\begin{lemma} \label{main:5}The set $S_{n+1, 2}$ of products $W^i _{(2)} W^{n+1-i}$ is uniquely determined from inductive data, Jacobi relations of type $(W^i, W^j, W^k)$ for $i+j+k = n+3$, and the sets $S_{n+1, \ell}$ for $\ell \geq 3$.\end{lemma}

\begin{proof} First, since $L_{(2)} W^{n-1} = L_{(2)} (W^3_{(1)}W^{n-2})$, we need to impose the Jacobi identity
\begin{equation} \label{main:5first} \begin{split} L_{(2)} (W^3_{(1)}W^{n-2}) & = W^3_{(1)}(L_{(2)} W^{n-2}) +(L_{(2)} W^3)_{(1)} W^{n-2}
\\ & +2(L_{(1)} W^3)_{(2)} W^{n-2} +(L_{(0)} W^3)_{(3)} W^{n-2}.\end{split} \end{equation}
 
By inductive assumption, $L_{(2)} W^{n-2}$ is known and is expressible in terms of $L, W^i$ for $i\leq n-4$. Then $W^3_{(1)}(L_{(2)} W^{n-2})$ is also determined by inductive data. The term $(L_{(2)} W^3)_{(1)} W^{n-2}$ vanishes because $W^3$ is primary, and the remaining terms are expressible in terms of $W^3_{(2)} W^{n-2}$ together with inductive data. Hence determining $L_{(2)} W^{n-1}$ is equivalent to determining $W^3_{(2)} W^{n-2}$.

For the purpose of determining $W^3_{(2)} W^{n-2}$, we impose the Jacobi relation
\begin{equation} \label{main:5second}  W^{n-3}_{(1)} (W^3_{(2)} W^3) =  W^3_{(2)} (W^{n-3}_{(1)} W^3) + (W^{n-3}_{(1)} W^3)_{(2)} W^3 + (W^{n-3}_{(0)} W^3)_{(3)} W^3.\end{equation}

The left-hand side of \eqref{main:5second} is $W^{n-3}_{(1)} \partial L$, which is known by inductive data. Next, by \eqref{commutator}, we have 

$$W^3_{(2)} (W^{n-3}_{(1)} W^3) + (W^{n-3}_{(1)} W^3)_{(2)} W^3 = \sum_{i\geq 1} (-1)^{i+1}\frac{1}{i!} \partial^i\bigg((W^{n-3}_{(1)} W^3)_{(i+2)} W^3\bigg).$$
Since $(W^{n-3}_{(1)} W^3) = W^{n-2}$ modulo terms which depend only on $L, W^3, \dots, W^{n-4}$, and is known by inductive data,
the sum $W^3_{(2)} (W^{n-3}_{(1)} W^3) + (W^{n-3}_{(1)} W^3)_{(2)} W^3$ in  \eqref{main:5second} is determined by inductive data together with derivatives of elements of $S_{n+1,\ell}$ for $\ell \geq 3$. Finally, the remaining term in \eqref{main:5second} is
$$(W^{n-3}_{(0)} W^3)_{(3)} W^3 \equiv - (W^3_{(0)} W^{n-3})_{(3)} W^3 = -\frac{2}{n-2} \partial W^{n-2}_{(3)} W^3 = \frac{6}{n-2} W^{n-2}_{(2)} W^3 $$ $$ = \frac{6}{n-2} \sum_{i\geq 0} (-1)^{i+1}\frac{1}{i!} \partial^i (W^3_{(i+2)} W^{n-2}),$$ modulo inductive data. Therefore $W^3_{(2)} W^{n-2}$ is expressible in terms of inductive data together with derivatives of elements of $S_{n+1,\ell}$ for $\ell >2$. Since $L_{(2)} W^{n-1}$ can be expressed in terms of $W^3_{(2)} W^{n-2}$, the same holds for $L_{(2)} W^{n-1}$.

Next, for $i\geq 3$ we impose the Jacobi relation 
\begin{equation} \label{main:5third} \begin{split}W^3_{(2)} (W^i_{(1)} W^{n-i} ) & = W^i_{(1)} (W^3_{(2)} W^{n-i}) + (W^3_{(2)} W^i)_{(1)} W^{n-i} + 2(W^3_{(1)} W^i)_{(2)} W^{n-i}  \\ & + (W^3_{(0)} W^i)_{(3)} W^{n-i}.\end{split} \end{equation}
This allows us to express $W^i_{(2)} W^{n+1-i}$ for all $i$ in terms of inductive data together with derivatives of $S_{n+1,\ell}$ for $\ell \geq 3$. \end{proof}

 \begin{lemma} \label{main:6} For all $k>2$, the set $S_{n+1, k}$ of products $W^i _{(k)} W^{n+1-i}$ is uniquely determined from inductive data, Jacobi relations of type $(W^i, W^j, W^r)$ for $i+j+r = n+3$, and the sets $S_{n+1, \ell}$ for $\ell >k $.\end{lemma}
 
\begin{proof} The argument is the same as the proof of Lemma \ref{main:5}. Imposing the Jacobi relation
\begin{equation} \label{main:6first} L_{(k)} (W^3_{(1)}W^{n-2}) = W^3_{(1)}(L_{(k)} W^{n-2}) +\sum_{i\geq 0} \binom{k}{i} (L_{(i)} W^3)_{(k+1-i)} W^{n-2}, \end{equation} shows that determining $L_{(k)} W^{n-1}$ is equivalent to determining $W^3_{(k)} W^{n-2}$. Imposing the Jacobi relation
\begin{equation} \label{main:6second}  W^{n-3}_{(1)} (W^3_{(k)} W^3) =  W^3_{(k)} (W^{n-3}_{(1)} W^3) + (W^{n-3}_{(1)} W^3)_{(k)} W^3 + (W^{n-3}_{(0)} W^3)_{(k+1)} W^3,\end{equation} shows that $W^3_{(k)} W^{n-2}$, and hence $L_{(k)} W^{n-1}$, are determined from inductive data together with $S_{n+1,\ell}$ for $\ell >k$. Finally, imposing the Jacobi relation
\begin{equation} \label{main:6third} W^3_{(k)} (W^i_{(1)} W^{n-i} )= W^i_{(1)} (W^3_{(k)} W^{n-i}) + \sum_{r\geq 0} \binom{k}{r} (W^3_{(r)} W^i)_{(k+1-r)} W^{n-i}, \end{equation} shows that $W^i_{(k)} W^{n+1-i}$ can be expressed in terms of inductive data together with $S_{n+1,\ell}$ for $\ell > k$. \end{proof} 

This process terminates after finitely many steps since all elements of $S_{n+1,k}$ vanish for $k >n$. Therefore we have proven the following

\begin{thm} \label{maintheorem} There exists a nonlinear conformal algebra $\cL(c,\lambda)$ over the ring $\mathbb{C}[c,\lambda]$ satisfying \eqref{kds-skew}, whose universal enveloping vertex algebra $\cW(c,\lambda)$ has the following properties.

\begin{enumerate}
\item It has conformal weight grading $$\cW(c,\lambda) = \bigoplus_{n\geq 0} \cW(c,\lambda)[n],$$ and $\cW(c,\lambda)[0] \cong \mathbb{C}[c,\lambda]$. 
\item It is strongly generated by $\{L,W^i|\ i\geq 3\}$, and satisfies the OPE relations \eqref{ope:first}-\eqref{ope:fourth} and \eqref{ope:w2w6}-\eqref{ope:w4w5}, together with the Jacobi identities \eqref{main:1first}-\eqref{main:2first}, \eqref{main:2fourth}, \eqref{main:4first}-\eqref{main:6third}, which appear in the above lemmas.
\item It is the unique initial object in the category of vertex algebras with the above properties. 
\end{enumerate}
\end{thm}

It is not yet apparent that all Jacobi identities of the form $(W^i, W^j, W^k)$ hold as consequences of \eqref{deriv}-\eqref{ncw} alone, or equivalently, that $\cL(c,\lambda)$ is a nonlinear Lie conformal algebra and $\cW(c,\lambda)$ is freely generated.

\subsection{Step 3: Free generation of $\cW(c,\lambda)$} In order to prove that $\cW(c,\lambda)$ is freely generated, we need to consider certain simple quotients of $\cW(c,\lambda)$. It is useful to think of such quotients as being obtained by a two-step procedure. First, let $$I \subseteq \mathbb{C}[c,\lambda] \cong \cW(c,\lambda)[0]$$ be an ideal, and let $I\cdot \cW(c,\lambda)$ denote the vertex algebra ideal generated by $I$, which is the set of all $I$-linear combinations of elements of $\cW(c,\lambda)$. We define the quotient
\begin{equation}\label{eq:quotientbyi} \cW^I(c,\lambda) = \cW(c,\lambda) / I\cdot \cW(c,\lambda),\end{equation} 
which has strong generators $\{L, W^i|\ i\geq 3\}$ satisfying the same OPE algebra as the corresponding generators of $\cW(c,\lambda)$ where all structure constants in $\mathbb{C}[c,\lambda]$ are replaced by their images in $\mathbb{C}[c,\lambda] / I$. Even though $\cW(c,\lambda)$ will turn out to be simple as a vertex algebra over $\mathbb{C}[c,\lambda]$, $\cW^I(c,\lambda)$ need not be simple as a vertex algebra over $\mathbb{C}[c,\lambda] / I$.

We now consider localizations of $\cW^I(c,\lambda)$. Let $D\subseteq \mathbb{C}[c,\lambda]/I$ be a multiplicatively closed subset, and let $R = D^{-1}\mathbb{C}[c,\lambda]/I$ denote the localization of $\mathbb{C}[c,\lambda]/I$ along $D$. Then we have the localization of $\mathbb{C}[c,\lambda]/I$-modules
$$\cW_R^I(c,\lambda) = R \otimes_{\mathbb{C}[c,\lambda]/I} \cW^I(c,\lambda),$$ which is a vertex algebra over $R$.

\begin{thm} \label{thm:simplequotient} Let $I$, $D$, and $R$ be as above, and let $\cW$ be a simple vertex algebra over $R$ with the following properties.
\begin{enumerate}
\item It is generated by a Virasoro field $\tilde{L}$ of central charge $c$ and a weight $3$ primary field $\tilde{W}^3$.
\item Setting $\tilde{W}^2 = \tilde{L}$ and $\tilde{W}^i  = \tilde{W}^3_{(1)} \tilde{W}^{i-1}$ for $i\geq 4$, the OPE relations \eqref{ope:first}-\eqref{ope:fourth} and \eqref{ope:w2w6}-\eqref{ope:w4w5} are satisfied if the structure constants are replaced with their images in $R$. 
\end{enumerate}
Then $\cW$ is the simple quotient of $\cW^I_R(c,\lambda)$ by its maximal graded ideal $\cI$.
\end{thm}

\begin{proof} The assumption that $\{\tilde{L}, \tilde{W}^i|\ i\geq 3\}$ satisfy \eqref{ope:first}-\eqref{ope:fourth} and \eqref{ope:w2w6}-\eqref{ope:w4w5} is equivalent to the statement that the Jacobi relations of type $(\tilde{W}^i, \tilde{W}^j, \tilde{W}^k)$ for $i+j+k \leq 11$ hold as consequences of \eqref{deriv}-\eqref{ncw} alone, in the corresponding (possibly degenerate) nonlinear conformal algebra. Then all OPE relations among the generators of $\cW^I_R(c,\lambda)$ must hold among the fields $\{\tilde{L}, \tilde{W}^i|\ i \geq 3\}$, since they are formal consequences of \eqref{ope:first}-\eqref{ope:fourth} and \eqref{ope:w2w6}-\eqref{ope:w4w5} together with the Jacobi identities, which hold in $\cW$. It follows that $\{\tilde{L},\tilde{W}^i|\ i\geq 3\}$ closes under OPE and strongly generates a vertex subalgebra $\cW' \subseteq \cW$, which must coincide with $\cW$ since $\cW$ is assumed to be generated by $\{\tilde{L},\tilde{W}^3\}$ as a vertex algebra. Hence $\cW$ has the same strong generating set and OPE algebra as $\cW^I_R(c, \lambda)$. Since $\cW$ is simple and the category of vertex algebras over $R$ with this strong generating set and OPE algebra has a unique simple graded object, $\cW$ must be the simple quotient of $\cW^I_R(c, \lambda)$ by its maximal graded ideal. \end{proof}

\begin{remark} \label{rem:simplicity} Although we are mostly interested in simple vertex algebras, if we drop the assumption that $\cW$ is simple but keep the remaining hypotheses in Theorem \ref{thm:simplequotient}, the same argument shows that $\cW$ is a quotient of $\cW^I_R(c,\lambda)$ by some graded ideal $\cI'$ which need not be maximal. \end{remark}

\begin{thm} \label{wprinquot} For all $n\geq 3$, the algebra $\cW^k(\gs\gl_n, f_{\text{prin}})$ which has central charge 
$$c(k) =  -\frac{(n-1) (n^2 + n k - n-1) (n^2 + k + n k)}{n + k},$$ is the simple quotient of $\cW^I_R(c,\lambda)$ for some prime ideal $I \subseteq \mathbb{C}[c,\lambda]$ and some localization $R$ of $\mathbb{C}[c,\lambda] / I$. Moreover, the maximal proper graded ideal $\cI \subseteq \cW^I_R(c,\lambda)$ is generated by a singular vector of weight $n+1$ of the form
\begin{equation} \label{sing:genfirstproof} W^{n+1} - P(L, W^3, \dots, W^{n-1}),\end{equation} where $P$ is a normally ordered polynomial in $L, W^3,\dots, W^{n-1}$, and their derivatives.
\end{thm}
 
\begin{proof} This is straightforward to verify by computer for $n\leq 7$, so assume that $n\geq 8$. Recall that $\cW^k(\gs\gl_n, f_{\text{prin}})$ is freely generated by the Virasoro field $L$, a weight $3$ primary field $W^3$ satisfying $W^3_{(5)} W^3 = \frac{c}{3} 1$, and fields $W^i = W^3_{(1)} W^{i-1}$ for $4 \leq i \leq n$. Since $n\geq 8$, there are no normally ordered polynomial relations among $L,W^3,\dots, W^8$, and their derivatives, so all Jacobi identities of type $(W^i, W^j, W^{\ell})$ hold as consequences of \eqref{deriv}-\eqref{ncw} alone, for $i+j+\ell \leq 11$. Since the structure constants in $\cW^k(\gs\gl_n, f_{\text{prin}})$ are rational functions of $k$, it follows from Theorem \ref{thm:simplequotient} that there exists some rational function
$$\lambda(k) = \frac{f(k)}{g(k)}$$ such that \eqref{ope:first}-\eqref{ope:fourth} and \eqref{ope:w2w6}-\eqref{ope:w4w5} are satisfied if $c$ and $\lambda$ are replaced by $c(k)$ and $\lambda(k)$, respectively.

Let $J \subseteq \mathbb{C}[c,\lambda, k]$ be the ideal generated by 
$$c(n+k) + (n-1) (n^2 + n k - n -1) (n^2 + k + n k),\qquad g(k) \lambda - f(k).$$ 
By standard methods of elimination theory, we can eliminate $k$ to obtain a prime ideal $I \subseteq \mathbb{C}[c,\lambda]$ such that some localization $R$ of $\mathbb{C}[c,\lambda] /I$ is isomorphic to some localization $D^{-1} \mathbb{C}[k]$. Here $D$ is a multiplicatively closed subset of $\mathbb{C}[k]$ which contains $(n+k)$ and all roots of $g(k)$.

Since $\cW^k(\gs\gl_n, f_{\text{prin}})$ is of type $\cW(2,3,\dots, n)$, there must be a singular vector in $\cW^I_R(c,\lambda)$ of weight $n+1$ such that the coefficient of $W^{n+1}$ is nonzero. If this coefficient is not invertible in $R$, we may localize $R$ further (without changing notation) so it becomes invertible, and the singular vector has the form \eqref{sing:genfirstproof}. Since $\cW^k(\gs\gl_n, f_{\text{prin}})$ is simple for generic values of $k$, it is simple as a vertex algebra over $R$, and hence must be the simple quotient $\cW^{I}_{R}(c,\lambda)/ \cI$, by Theorem \ref{thm:simplequotient}. Here $\cI$ is the maximal proper graded ideal of $\cW^{I}_{R}(c,\lambda)$. Finally, we need to show that \eqref{sing:genfirstproof} generates $\cI$.

Let $\cI' \subseteq \cI$ be the ideal in $\cW^{I}_R(c,\lambda)$ generated by \eqref{sing:genfirstproof}. Since $\cW^k(\gs\gl_n, f_{\text{prin}}) \cong\cW^I_R(c,\lambda)/ \cI$, $\cW^k(\gs\gl_n, f_{\text{prin}})$ is also a quotient of $\cW^I_R(c,\lambda) / \cI'$. Also, $\cW^I_R(c,\lambda)/ \cI'$ is of type $\cW(2,3,\dots, n)$; see Lemma \ref{decoup}. Since $\cW^k(\gs\gl_n, f_{\text{prin}})$ is freely generated, there can be no more relations in $\cW^I_R(c,\lambda)/ \cI$ than in $\cW^I_R(c,\lambda) / \cI'$, so $\cI'  = \cI$. \end{proof}

\begin{remark} It is not obvious at this stage that the generator of the ideal $I$ is a polynomial in both $n$ and $c$. Later, we will give an explicit formula for this generator; see \eqref{ideal:wslna}.
\end{remark}

\begin{cor} All Jacobi identities of type $(W^i, W^j, W^k)$ hold as consequences of \eqref{deriv}-\eqref{ncw} alone in $\cL(c,\lambda)$, so $\cL(c,\lambda)$ is a nonlinear Lie conformal algebra with generators $\{L, W^i|\ i\geq 3\}$. Equivalently, $\cW(c,\lambda)$ is freely generated by $\{L, W^i|\ i\geq 3\}$ and has graded character
\begin{equation} \label{grchar} \chi(\cW(c,\lambda), q) = \sum_{n\geq 0} \text{rank}_{\mathbb{C}[c,\lambda]}( \cW(c,\lambda)[n]) q^n = \prod_{n\geq 2} \frac{1}{(1-q^n)^{n-1}}.\end{equation} Moreover, for any prime ideal $I \subseteq \mathbb{C}[c,\lambda]$, $\cW^I(c,\lambda)$ is freely generated by $\{L, W^i|\ i\geq 3\}$ as a vertex algebra over $\mathbb{C}[c,\lambda] / I$, and
\begin{equation} \label{grcharq} \chi(\cW^I(c,\lambda), q) = \sum_{n\geq 0} \text{rank}_{\mathbb{C}[c,\lambda]/I}( \cW^I(c,\lambda)[n]) q^n = \prod_{n\geq 2} \frac{1}{(1-q^n)^{n-1}},\end{equation} 
and for any localization $R = D^{-1} \mathbb{C}[c,\lambda]/I$ along a multiplicatively closed set $D \subseteq \mathbb{C}[c,\lambda]/I$, $\cW^I_R(c,\lambda)$ is freely generated by $\{L, W^i|\ i\geq 3\}$ and \begin{equation} \label{grcharql} \chi(\cW^I_R(c,\lambda), q) = \sum_{n\geq 0} \text{rank}_{R}( \cW^I_R(c,\lambda)[n]) q^n = \prod_{n\geq 2} \frac{1}{(1-q^n)^{n-1}}.\end{equation} 
\end{cor}

\begin{proof} If some Jacobi identity of type $(W^i, W^j, W^k)$ does not hold as a consequence of \eqref{deriv}-\eqref{ncw}, there would be a null vector of weight $N$ in $\cW(c,\lambda)$ for some $N$. Then the rank of $\cW(c,\lambda)[N]$ would be strictly smaller than that given by \eqref{grchar}, and the same would hold in any quotient of $\cW(c,\lambda)[N]$, as well as any localization of such a quotient. But since $\cW^k(\gs\gl_N, f_{\text{prin}})$ is a localization of such a quotient and is freely generated of type $\cW(2,3,\dots, N)$, this is impossible. 
\end{proof}

\begin{cor} The vertex algebra $\cW(c,\lambda)$ is simple.
\end{cor}

\begin{proof} If $\cW(c,\lambda)$ is not simple, it would have a singular vector in weight $N$ for some $N$. Let $p \in \mathbb{C}[c,\lambda]$ be an irreducible polynomial and let $I = (p) \subseteq \mathbb{C}[c,\lambda]$. By rescaling if necessary, we can assume without loss of generality that the singular vector is not divisible by $p$ and hence descends to a nontrivial singular vector in $\cW^I(c,\lambda)$. Then for any localization $R$ of $\mathbb{C}[c,\lambda] / I$, the simple quotient of $\cW^I_R(c,\lambda)$ would have a smaller weight $N$ submodule than $\cW(c,\lambda)$ for all such $I$. This contradicts the fact that $\cW^k(\gs\gl_N, f_{\text{prin}})$ is such a quotient.
\end{proof}

\begin{cor}\label{cor:autgroup} The vertex algebra $\cW(c,\lambda)$ has full automorphism group $\mathbb{Z}_2$.
\end{cor}
\begin{proof}
By construction, $\cW(c,\lambda)$ has a nontrivial involution determined by $L \mapsto L$ and $W^3 \mapsto -W^3$, so that $W^i\mapsto (-1)^i W^i$ for all $i\geq 3$. If $\cW(c,\lambda)$ had another nontrivial automorphism $\phi$, we must have $\phi(L) = L$ and $\phi(W^3) = a W^3$ for some nonzero $a \in \mathbb{C}$, since $W^3$ is the unique weight $3$ primary field up to scalar. But the OPE relation \eqref{ope:first} forces $a = \pm 1$, so $\phi$ is either the identity map, or the above involution. 
\end{proof}	

The inductive procedure for proving Theorem \ref{maintheorem} can be regarded as an {\it algorithm} for computing the OPE algebra of $\cW(c,\lambda)$ by starting from the OPE formulas \eqref{ope:first}-\eqref{ope:fourth} and \eqref{ope:w2w6}-\eqref{ope:w4w5}, and imposing all Jacobi identities among the generators. Using the Mathematica package of Thielemans \cite{T}, we have applied this algorithm to compute all OPEs $W^i(z) W^j(w)$ for $2\leq i\leq j$ and $i+j \leq 14$. This data is too complicated to reproduce in this paper, but we used it to obtain some of the results and conjectures in Sections \ref{section:prinw}-\ref{section:oneplus}.

\subsection{Zhu functor}
The Zhu functor is a fundamental tool in the representation theory of vertex algebras \cite{Z}. Let $\cV$ be a vertex algebra with weight grading $\cV = \bigoplus_{n\geq 0} \cV[n]$. For $a\in \cV[m]$ and $b\in\cV$, define
\begin{equation}\label{defzhu} a*b = \text{Res}_z \bigg (a(z) \frac{(z+1)^{m}}{z}b\bigg),\end{equation} and extend $*$ by linearity to a bilinear operation $\cV\otimes \cV\ra \cV$. Let $O(\cV)$ denote the subspace of $\cV$ spanned by elements of the form \begin{equation}\label{zhuideal} a\circ b = \text{Res}_z \bigg (a(z) \frac{(z+1)^{m}}{z^2}b\bigg)\end{equation} where $a\in \cV[m]$, and define the {\it Zhu algebra} $\text{Zhu}(\cV)$ to be the quotient vector space $\cV/O(\cV)$, with projection \begin{equation} \label{zhumap} \pi_{\text{Zhu}}:\cV\ra \text{Zhu}(\cV).\end{equation} Then $O(\cV)$ is a two-sided ideal in $\cV$ under the product $*$, and $(\text{Zhu}(\cV),*)$ is a unital, associative algebra. The assignment $\cV\mapsto \text{Zhu}(\cV)$ is functorial, and if $\cI \subseteq \cV$ is a vertex algebra ideal, \begin{equation} \label{zhuquotient} \text{Zhu}(\cV/\cI)\cong \text{Zhu}(\cV)/ I,\qquad I = \pi_{\text{Zhu}}(\cI).\end{equation}

If $\cV$ is strongly generated by homogeneous elements $\{\alpha^1, \alpha^2,\dots\}$, $\text{Zhu}(\cV)$ is generated by $\{ a^i = \pi_{\text{Zhu}}(\alpha^i)\}$. A $\mathbb{Z}_{\geq 0}$-graded $\cV$-module $M = \bigoplus_{n\geq 0} M[n]$ is called a {\it positive energy module} if for every $a\in\cV[m]$, $a(n) M_k \subseteq M[m+k -n-1]$, for all $n$ and $k$. Given $a\in\cV[m]$, $a(m-1)$ acts on each $M[k]$. The subspace $M[0]$ is then a $\text{Zhu}(\cV)$-module with action $\pi_{\text{Zhu}}(a)\mapsto a(m-1) \in \text{End}(M[0])$. In fact, $M\mapsto M[0]$ provides a one-to-one correspondence between irreducible, positive energy $\cV$-modules and irreducible $\text{Zhu}(\cV)$-modules. If $\text{Zhu}(\cV)$ is commutative, all its irreducible modules are one-dimensional. The corresponding irreducible $\cV$-modules $M = \bigoplus_{n\geq 0} M[n]$ are then cyclic, and will be called {\it highest-weight modules}.

\begin{thm} \label{thm:zhu} The Zhu algebra $\text{Zhu}(\cW(c,\lambda))$ is the polynomial algebra \begin{equation} \label{zhu} \mathbb{C}[c,\lambda] \otimes_{\mathbb{C}} \mathbb{C}[\ell, w^i|\ i\geq 3],\end{equation} where the generators $\ell, w^i$ are the images of $L, W^i$ under the Zhu map \eqref{zhumap}. For any ideal $I\subseteq \mathbb{C}[c,\lambda]$, any localization $R$ of $\mathbb{C}[c,\lambda] / I$, and any quotient $\cW^I_R(c,\lambda) /\cI$, the Zhu algebra $\text{Zhu}(\cW^I_R(c,\lambda) /\cI)$ is a quotient of a localization of \eqref{zhu}, and hence is abelian.
\end{thm}

\begin{proof} Since $\cW(c,\lambda)$ is strongly generated by $\{L, W^i|\ i \geq 3\}$, $\text{Zhu}(\cW(c,\lambda))$ is generated by $\{\ell, w^i|\ i\geq 3\}$, and it is well known that $\ell$ is central. The commutator $[w^i, w^j]$ in the Zhu algebra is expressed in terms of the OPE algebra and hence is a polynomial in $\{\ell, w^i|\ i\geq 3\}$ with structure constants in $\mathbb{C}[c,\lambda]$. Since the Zhu algebra of $\cW^k(\gs\gl_n, f_{\text{prin}})$ is known to be abelian, each structure constant is divisible by each $p_n$, where $p_n \in \mathbb{C}[c,\lambda]$ generates the ideal $I_n$ such that $\cW^k(\gs\gl_n,f_{\text{prin}})$ is a quotient of a localization of $\cW^{I_n}(c,\lambda)$. Since $\cW^k(\gs\gl_n, f_{\text{prin}})$ is generated by the weight $3$ field for all $n \geq 3$, the polynomials $p_n$ must all be distinct, so all of the above structure constants must vanish. The remaining statements follow from \eqref{zhuquotient}.
\end{proof}

\begin{cor} For any vertex algebra $\cW =  \cW^I_R(c,\lambda) / \cI$ for some $I$ and $R$ as above, all irreducible, positive energy modules are highest-weight modules, and are parametrized by the variety $\text{Specm}(\text{Zhu}(\cW))$. \end{cor}

\subsection{Poisson vertex algebra structure}
For any vertex algebra $\cV$, we have Li's canonical decreasing filtration
$$F^0(\cV) \supseteq F^1(\cV) \supseteq \cdots.$$ Here $F^p(\cV)$ is spanned by the elements
$$:(\partial^{n_1} a^1) (\partial^{n_2} a^2) \cdots (\partial^{n_r} a^r):,$$ 
where $a^1,\dots, a^r \in \cV$, $n_i \geq 0$, and $n_1 + \cdots + n_r \geq p$ \cite{LiIII}. Clearly $\cV = F^0(\cV)$ and $\partial F^i(\cV) \subseteq F^{i+1}(\cV)$. Set $$\text{gr}^F(\cV) = \bigoplus_{p\geq 0} F^p(\cV) / F^{p+1}(\cV),$$ and for $p\geq 0$ let 
$$\sigma_p: F^p(\cV) \ra F^p(\cV) / F^{p+1}(\cV) \subseteq \text{gr}^F(\cV)$$ be the projection. Then $\text{gr}^F(\cV)$ is a graded commutative algebra with product
$$\sigma_p(a) \sigma_q(b) = \sigma_{p+q}(a_{(-1)} b),\qquad a \in F^p(\cV),\qquad b \in F^q(\cV).$$ There is a differential $\partial$ on $\text{gr}^F(\cV)$, $$\partial( \sigma_p(a) ) = \sigma_{p+1} (\partial a),\qquad a \in F^p(\cV).$$ Finally, $\text{gr}^F(\cV)$ has the structure of a Poisson vertex algebra \cite{LiIII}; for $n\geq 0$, $a \in F^p(\cV)$, and $b\in F^q(\cV)$, define
$$\sigma_p(a)_{(n)} \sigma_q(b) = \sigma_{p+q-n} a_{(n)} b.$$
 
The subalgebra $F^0(\cV) / F^1(\cV)$ coincides with Zhu's commutative algebra $C(\cV)$ \cite{Z}, and is known to generate $\text{gr}^F(\cV)$ as a differential graded algebra \cite{LiIII}. We change notation slightly and denote by $\bar{a}$ the image of $a$ in $C(\cV)$. It is a Poisson algebra with product $\bar{a} \cdot \bar{b} = \overline{a_{(-1)} b}$ and Poisson bracket $\{ \bar{a}, \bar{b}\} = \overline{a_{(0)} b}$. If the Poisson bracket on $C(\cV)$ is trivial, it follows that the Poisson vertex algebra structure on $\text{gr}^F(\cV)$ is trivial in the sense that for all $a \in F^p(\cV)$, $b\in F^q(\cV)$ and $n\geq 0$, $\sigma_p(a)_{(n)} \sigma_q(b) = 0$; see Remark 4.32 of \cite{AMII}

\begin{thm} 
For any vertex algebra $\cW =  \cW^I_R(c,\lambda) / \cI$ for some $I$ and $R$ as above, the Poisson structure on $C(\cW)$ and the Poisson vertex algebra structure on $\text{gr}^F(\cW)$ are both trivial.
 \end{thm}

\begin{proof} Since $C(\cW)$ is generated by $\{\bar{L}, \bar{W}^i |\ i \geq 3\}$ and $\{\bar{L},-\}$ acts trivially on $C(\cW)$, it suffices to show that$\{\bar{W}^j, \bar{W}^k\} = 0$ for all $j,k \geq 3$. But $W^j_{(0)} W^k$ is a sum of normally ordered monomials in $\{L,W^i |\ 3\leq i \leq j+k -1\}$ and their derivatives, which each have weight $j+k-1$ and eigenvalue $(-1)^{j+k}$ under the $\mathbb{Z}_2$-action. It follows that each monomial must lie in $F^1(\cW)$, so that $\overline{W^{j}_{(0)} W^{k}} = 0$. \end{proof}

\section{Quotients of $\cW(c,\lambda)$ and the classification of vertex algebras of type $\cW(2,3,\dots, N)$} \label{section:classification}
As we have seen, any simple vertex algebra of type $\cW(2,3,\dots, N)$ for some $N$, which is generated by a Virasoro field and a weight $3$ primary field $W^3$, and satisfies some natural conditions, is a quotient of $\cW^I_R(c,\lambda)$ for some $I$ and $R$. This reduces the classification of vertex algebras with these properties to the classification of ideals $I \subseteq \mathbb{C}[c,\lambda]$ such that $\cW^I(c,\lambda)$ is not simple. In this section, we restrict to the case where $I = (p)$ for some irreducible $p \in \mathbb{C}[c,\lambda]$, so that $\cW^I_R(c,\lambda)$ is a one-parameter vertex algebra in the sense that $R$ has Krull dimension $1$. Later, in Section \ref{section:coincidences} we will consider the case where $I$ is a maximal ideal.

Recall that $\cW(c,\lambda)[n]$ is a free $\mathbb{C}[c,\lambda]$-module whose rank is given by \eqref{grchar}. It has a symmetric bilinear form $$\langle,\rangle_n: \cW(c,\lambda)[n]\otimes_{\mathbb{C}[c,\lambda]} \cW(c,\lambda)[n] \ra \mathbb{C}[c,\lambda],\qquad \langle \omega,\nu \rangle_n = \omega_{(2n-1)}\nu.$$ 
Recall that the Shapovalov determinant $\text{det}_n  \in \mathbb{C}[c,\lambda]$ is the determinant of the matrix of this form, and $\text{det}_n \neq 0$ for all $n$ since $\cW(c,\lambda)$ is a simple vertex algebra. Recall that an irreducible polynomial $p \in \mathbb{C}[c,\lambda]$ is said to lie in the level $n$ Shapovalov spectrum if $p$ divides $\text{det}_n$ but does not divide $\text{det}_m$ for any $m<n$. Clearly $\cW^I(c,\lambda)$ is simple for a generic choice of $I$, since each $\text{det}_n$ has finitely many irreducible factors.

Let $p$ be an irreducible factor of $\text{det}_{N+1}$ of level $N+1$. Letting $I = (p) \subseteq \mathbb{C}[c,\lambda]$, $\cW^I(c,\lambda)$ will then have a singular vector in weight $N+1$. The coefficient of $W^{N+1}$ in this singular vector is often nonzero. By inverting this coefficient, we obtain a localization $R$ of $\mathbb{C}[c,\lambda] / I$ such that this singular vector has the form 
\begin{equation} \label{elim:first} W^{N+1} - P_{N+1}(L, W^3,\dots, W^{N-1})\end{equation} in $\cW_R^I(c,\lambda)$. Here $P_{N+1}$ is a normally ordered polynomial in the fields $L, W^3, \dots, W^{N-1}$ and their derivatives, with coefficients in $R$. This implies that $W^{N+1}$ decouples in the quotient of $\cW_R^I(c,\lambda) / \cJ$, where $\cJ$ denotes the vertex algebra ideal generated by \eqref{elim:first}. In other words, we have the relation $$W^{N+1} = P_{N+1}(L, W^3,\dots, W^{N-1})$$ in $\cW_R^I(c,\lambda) / \cJ$. Applying the operator $W^3_{(1)}$ to this relation and using the fact that $W^3_{(1)} W^{N+1} = W^{N+2}$ and $W^3_{(1)} W^{N-1} = W^{N}$, we obtain a relation $$W^{N+2} = P_{N+2}(L, W^3, \dots, W^{N}).$$ Applying $W^3_{(1)}$ again yields a relation
$$W^{N+3} = Q_{N+3}(L, W^3, \dots, W^{N+1}).$$ If the terms $\partial^2 W^{N+1}$ or $:L W^{N+1}:$ appear in $Q_{N+3}$, they can be eliminated using \eqref{elim:first} to obtain a relation 
$$W^{n+3} = P_{N+3}(L,W^3, \dots, W^{N}).$$ Inductively, by applying $W^3_{(1)}$ repeatedly and using \eqref{elim:first} to eliminate $W^{N+1}$ if necessary, we obtain relations 
$$W^m = P_m(L, W^3, \dots, W^{N})$$ in $\cW_R^I(c,\lambda) / \cJ$, for all $m > N+1$. This implies the following.

\begin{lemma} \label{decoup}
Let $p$ be an irreducible factor of $\text{det}_{N+1}$ that does not divide $\text{det}_m$ for $m< N+1$, and let $I = (p)$. Suppose that there exists a localization $R$ of $\mathbb{C}[c,\lambda] / I$ such that $\cW_R^I(c,\lambda)$ has a singular vector of the form \begin{equation} \label{eq:decoup}W^{N+1} - P_{N+1}(L, W^3,\dots, W^{N-1}).\end{equation} Let $\cJ \subseteq \cW_R^I(c,\lambda)$ be the vertex algebra ideal generated by \eqref{eq:decoup}. Then the quotient $\cW_R^I(c,\lambda) / \cJ$ has a minimal strong generating set $\{L, W^i|\ 3 \leq i \leq N\}$, and in particular is of type $\cW(2,3,\dots, N)$. 
\end{lemma}

\begin{remark} The ideal $\cJ$ is often the maximal graded ideal $\cI \subseteq \cW_R^I(c,\lambda)$, although this need not be the case. However, the assumption that $p$ does not divide $\text{det}_m$ for $m<N+1$ implies that there are no singular vectors in weight $m<N+1$, so there can be no decoupling relations of the form $W^m = P_m(L,W^3,\dots W^{m-1})$ for $m<N+1$. Therefore the simple quotient $\cW_R^I(c,\lambda) / \cI$ is also of type $\cW(2,3\dots,N)$.\end{remark}

\begin{thm} For all $N\geq 3$, there are finitely many isomorphism classes of simple one-parameter vertex algebras of type $\cW(2,3,\dots, N)$, which are generated by the Virasoro field $L$ and a weight $3$ primary field $W^3$ satisfying \eqref{ope:first}-\eqref{ope:fourth} and \eqref{ope:w2w6}-\eqref{ope:w4w5}.
\end{thm}

\begin{proof} Any such vertex algebra must be the simple quotient of $\cW^I_R(c,\lambda)$ for some $I$ and some localization $R$ of $\mathbb{C}[c,\lambda] / I$, such that $\cW^I_R(c,\lambda)$ has a singular vector in weight $ m \leq N+1$. But there are only finitely many divisors of $\text{det}_{m}$ for $m\leq N+1$. \end{proof}

In Section \ref{section:coincidences}, we will show that for $N\geq 3$, the correspondence between these ideals and the isomorphism classes of such vertex algebras, is a bijection; see Corollary \ref{cor:uniqueness}.

There is a useful criterion for proving that a vertex algebra of type $\cW(2,3\dots, N)$ is a quotient of $\cW^I_R(c,\lambda)$ for some $I$ and $R$.

\begin{thm} \label{refinedsimplequotient} Let $\cW$ be a vertex algebra of type $\cW(2,3,\dots, N)$ which is defined over some localization $R$ of $\mathbb{C}[c,\lambda] / I$, for some prime ideal $I$. Suppose that $\cW$ is generated by the Virasoro field $L$ and a weight $3$ primary field $W^3$. If in addition, the graded character of $\cW$ agrees with that of $\cW(c,\lambda)$ up to weight $8$, then $\cW$ is a quotient of $\cW^I_R(c,\lambda)$. In particular, if $\cW$ is simple as a vertex algebra over $R$, then it is the simple quotient of $\cW^I_R(c,\lambda)$.
\end{thm}

\begin{proof} By Theorem \ref{thm:simplequotient} and Remark \ref{rem:simplicity}, it suffices to prove that \eqref{ope:first}-\eqref{ope:fourth} and \eqref{ope:w2w6}-\eqref{ope:w4w5} are satisfied. But this is automatic because the graded character assumption implies that there are no null vectors of weight $w\leq 8$ in the (possibly degenerate) nonlinear conformal algebra corresponding to $\{L, W^i|\ 3 \leq i \leq N\}$, where $W^i = W^3_{(1)} W^{i-1}$ for $i\geq 4$. \end{proof}

 \section{Principal $\cW$-algebras of type $A$} \label{section:prinw}

By Theorem \ref{wprinquot}, there is some ideal $I\subseteq \mathbb{C}[c,\lambda]$ and some localization $R$ of $\mathbb{C}[c,\lambda] / I$ such that $\cW^k(\gs\gl_n, f_{\text{prin}})$ can be obtained as the simple quotient of $\cW^I_R(c,\lambda)$. In this section, we shall give an explicit generator of this ideal. We need three preliminary results that are easily obtained by computer calculation. First, recall the parafermion algebra $$N^k(\gs\gl_2) = \text{Com}(\cH, V^k(\gs\gl_2)),$$ where $\cH$ is the Heisenberg algebra corresponding to the Cartan subalgebra of $\gs\gl_2$. It is of type $\cW(2,3,4,5)$ for all $k\neq 0$, and is generated by the Virasoro field $L$ of central charge $\displaystyle  c = \frac{2 (k-1)}{k+2}$ and a weight $3$ primary field $W^3$ \cite{DLY}.

\begin{thm} \label{thm:sl2para} The polynomial 
\begin{equation} \label{curve:para} p =  4\lambda (c+7) (2c-1) + (c-2) (c+4),\end{equation} is an irreducible factor of $\text{det}_6$ of level $6$. Let $I = (p) \subseteq\mathbb{C}[c,\lambda]$, let $D$ be the multiplicative set generated by $(c+7)$ and $(2c -1)$, and let $$R = D^{-1} \mathbb{C}[c,\lambda] / I \cong D^{-1} \mathbb{C}[c].$$ Then $$N^k(\gs\gl_2) \cong \cW^I_R(c,\lambda) / \cI,$$ where $\cI$ is the maximal proper graded ideal of $\cW^I_R(c,\lambda)$. In particular, $N^k(\gs\gl_2)$ is obtained from $\cW(c,\lambda)$ by setting
\begin{equation} \label{ratpara:para} c =  \frac{2 (k-1)}{k+2},\qquad \lambda = \frac{k+1}{(k-2) (3k+4)},\end{equation} and then taking the simple quotient. 
\end{thm}

Next, recall the coset 
$$\cC^{\ell} = \text{Com}(\cH, \cW^{\ell}),$$ where $\cW^{\ell} = \cW^{\ell - 3/2}(\gs\gl_3, f_{\text{min}})$ is the Bershadsky-Polyakov algebra. By Theorem 6.1 of \cite{ACLI}, $\cC^{\ell}$ is of type $\cW(2,3,4,5,6,7)$ for all $\ell \neq 0, 1/2$.

\begin{thm} \label{thm:bpcoset} The polynomial 
\begin{equation} \label{curve:bp} p = 48 + 8 c + 240 \lambda - 62 c \lambda - 5 c^2 \lambda + 300 \lambda^2 + 524 c \lambda^2 + 40 c^2 \lambda^2,\end{equation} is an irreducible factor of $\text{det}_8$ at level $8$. The corresponding variety $V(I)$ for $I = (p)$ is a rational curve with parametrization
\begin{equation} \label{ratparam:cell} c = -\frac{ 3 (2 \ell -1)^2}{2 \ell + 3},\qquad   \lambda  = \frac{(2 \ell+ 1) ( 2 \ell+3)}{8 (\ell -1) (4 \ell +3)},\end{equation} and $\cC^{\ell}$ is obtained from $\cW(c,\lambda)$ by substituting \eqref{ratparam:cell} and then taking the simple quotient. In particular, there exists a localization $R$ of $\mathbb{C}[c,\lambda] / I$ such that 
$$\cC^{\ell} \cong \cW^I_R(c,\lambda)/ \cI,$$ where $\cI$ is the maximal graded proper ideal of $\cW^I_R(c,\lambda)$. 
\end{thm}

Finally, recall that the coset $$\cC^k = \text{Com}(\cH, \cW^k(\gs\gl_4, f_{\text{subreg}}))$$ is of type $\cW(2,3,4,5,6,7,8,9)$ by Theorem 5.1 of \cite{CLIII}. 

\begin{thm} \label{thm:sl4subregcoset} 
The polynomial
\begin{equation} \label{curve:subregn=4} p = 320 + 40 c + 1536 \lambda - 804 c \lambda - 57 c^2 \lambda + 1456 \lambda^2 + 1536 c \lambda^2 + 
 444 c^2 \lambda^2 + 20 c^3 \lambda^2,\end{equation} is an irreducible factor of $\text{det}_{10}$ of level $10$. The corresponding variety $V(I)$ for $I = (p)$ is a rational curve with parametrization
\begin{equation} \label{subregn=4} c =  -\frac{4 (5 + 2 k) (7 + 3 k)}{4 + k} ,\qquad \lambda = -\frac{ (3 + k) (4 + k)}{3 (2 + k)^2 (16 + 5 k)},\end{equation} 
and $\cC^k$ is obtained from $\cW(c,\lambda)$ by substituting \eqref{subregn=4} and then taking the simple quotient. In particular, there exists a localization $R$ of $\mathbb{C}[c,\lambda] / I$ such that 
$$\cC^{k} \cong \cW^I_R(c,\lambda)/ \cI,$$ where $\cI$ is the maximal graded proper ideal of $\cW^I_R(c,\lambda)$. 
\end{thm}

The main result in this section is the following.
\begin{thm} \label{thm:prinw} 
For all $n\geq 3$, let $I_n \subseteq \mathbb{C}[c,\lambda]$ be the ideal generated by
\begin{equation} \label{ideal:wslna} p_n =  \lambda  (n-2) (3 n^2  - n -2+ c (n + 2))-  (n-1) (n+1).\end{equation}
Then we have the following.
\begin{enumerate}
\item The polynomial $p_n$ is an irreducible factor of $\text{det}_{n+1}$ of level $n+1$.
\item There exists a localization $R_n$ of $\mathbb{C}[c,\lambda] / I_n$ in which $(3 n^2  - n -2+ c (n + 2))$ is invertible, so that  \begin{equation} \label{ideal:wsln} \lambda =  \frac{(n-1) (n+1)}{(n-2) (3 n^2  - n -2+ c (n + 2))}\end{equation} in $R_n$, and there is a unique singular vector 
 \begin{equation} \label{sing:gen} W^{n+1} - P_n(L, W^3,\dots, W^{n-1}),\end{equation} which generates the maximal graded proper ideal $\cI_n \subseteq \cW^{I_n}_{R_n} (c,\lambda)$.
\item We have an isomorphism $$\cW^{I_n}_{R_n}(c,\lambda)/ \cI_n \cong \cW^k(\gs\gl_n, f_{\rm{prin}}),$$ where $c$ and $\lambda$ are related to $k$ by \begin{equation} \begin{split}
& c  = -\frac{(n-1) (n^2 + n k - n-1 ) (n^2 + k + n k)}{n + k},\\ & \lambda = -\frac{n + k}{ (n-2) (n^2 + n k - n-2) (n + n^2 + 2 k + n k)}.\end{split}\end{equation}
\end{enumerate}
In particular, the structure constants appearing in the OPEs of the generators $\{L, W^i|\ i \geq 3\}$ of $\cW^k(\gs\gl_n, f_{\rm{prin}})$ are rational functions of $k$ and $n$. \end{thm}

\begin{proof}
In view of Theorems \ref{ACLIImain} and \ref{wprinquot}, it suffices to show that the coset $$\cC^k(\gs\gl_n)  = \text{Com}(V^{k+1}(\gs\gl_n), V^{k}(\gs\gl_n) \otimes L_1(\gs\gl_n))$$ is obtained from $\cW(c,\lambda)$ by setting 
\begin{equation} \label{prinw:cosetreal} c = \frac{k (n-1) (1 + k + 2 n)}{(k + n) (1 + k + n)},\qquad \lambda = \frac{ (k + n) (1 + k + n)}{(n-2) (2 k + n) (2 + 2 k + 3 n)},\end{equation} and taking the simple quotient. The structure of this coset has been studied in \cite{BBSSI,BBSSII}. The relevant calculations are contained in these papers, and we follow their notation, taking an antihermitian basis for $\gs\gl_n$ and corresponding generators $J^a_{(1)}$ for $V^k(\gs\gl_n)$ satisfying $$J^a_{(1)} (z) J^b_{(1)} (w) \sim - \delta_{a,b} k (z-w)^{-2} + f^{abc} J^{c}_{(1)} (z-w)^{-1}.$$ Here $f^{abc}$ is the $f$-tensor as defined in the appendix of \cite{BBSSI}, and we are summing over repeated indices. Similarly, we have generators $J^a_{(2)}$ for $L_1(\gs\gl_n)$ satisfying 
$$J^a_{(2)} (z) J^b_{(2)} (w) \sim - \delta_{a,b}(z-w)^{-2} + f^{abc} J^{c}_{(2)} (z-w)^{-1}.$$
We also need the fields $Q^a_{(1)} \in V^k(\gs\gl_n)$ and $Q^a_{(2)} \in L_1(\gs\gl_n)$ given by
$$Q^a_{(1)} = d^{abc} :J^b_{(1)} J^c_{(1)}:,\qquad Q^a_{(2)} = d^{abc} :J^b_{(2)} J^c_{(2)}:,$$ where the $d$-tensor is defined as in \cite{BBSSI}. These are primary of weight $2$ and satisfy
\begin{equation} \begin{split} \label{ope:jq} J^a_{(1)}(z) Q^b_{(1)}(z) \sim -(2k+n) d^{abc} J^c_{(1)}(w)(z-w)^{-2} + f^{abc} Q^c_{(1)}(w)(z-w)^{-1},\\
J^a_{(2)}(z) Q^b_{(2)}(z) \sim -(2+n) d^{abc} J^c_{(2)}(w)(z-w)^{-2} + f^{abc} Q^c_{(2)}(w)(z-w)^{-1}.\end{split}\end{equation}
We also need \begin{equation}  \label{ope:qq} (Q^a_{(1)})_{(3)} Q^b_{(1)} = \delta_{a,b} \frac{2k(n+2k)(n^2-4)}{n} 1,\qquad (Q^a_{(2)})_{(3)} Q^b_{(2)} = \delta_{a,b} \frac{2(n+2)(n^2-4)}{n} 1.\end{equation}

In terms of these fields, the weight $3$ primary field $W^3 \in \cC^k(\gs\gl_n)$ is given as follows.
\begin{equation}  \label{def:cosetw3} \begin{split} W^3 &= B(n,k)  \bigg((n+1)(n+2) :J^a_{(1)} Q^a_{(1)}: - 3(n+k)(n+1)(n+2) :J^a_{(2)} Q^a_{(1)}: \\ & + 3(n+k)(n+1)(n+2k) :J^a_{(1)} Q^a_{(2)}: - k(n+k)(n+2k) :J^a_{(2)} Q^a_{(2)}:\bigg),\end{split} \end{equation} where $B(n,k)$ is given by
\begin{equation} \label{def:bnk} \frac{i}{3(n+k)(n+1)(n+k+1)} \sqrt{\frac{n}{2(n+2k)(n+2)(2+ 2k+3n)(n^2-4)}}.\end{equation}  As usual, $W^3$ is normalized so that $$W^3_{(5)} W^3 = \frac{k (n-1) (1 + k + 2 n)}{3(k + n) (1 + k + n)} 1 = \frac{c}{3} 1.$$

As in Section \ref{section:main}, set $W^i = W^3_{(1)} W^{i-1}$ for $i\geq 4$. It is apparent from the OPE formulas \eqref{ope:first}-\eqref{ope:fourth} and \eqref{ope:jq}-\eqref{ope:qq}, as well as properties of the $f$-tensor and $d$-tensor appearing in \cite{BBSSI}, that all structure constants in the algebra generated by $\{L, W^i|\ i\geq 3\}$ are rational functions in $k$ and $n$. In particular, \begin{equation} W^3_{(3)} W^4 = p(n,k) W^3,\end{equation} for some rational function $p(n,k)$ in $k$ and $n$. To prove the theorem, it suffices to show that \begin{equation} \label{formpnk} \begin{split} p(n,k) & = \frac{q(n,k)}{(n-2) (2 k + n) (2 + 2 k + 3 n)}, \\ & q(n,k) = 3 (-88 k - 88 k^2 - 52 n - 140 k n + 36 k^2 n - 52 n^2 + 72 k n^2 + 31 n^3).\end{split} \end{equation}

It follows from \eqref{def:cosetw3} and the OPE relations \eqref{ope:jq}-\eqref{ope:qq} that up to a constant, the denominator of $p(n,k)$ can contain at most the factors \begin{equation} \label{pdenom} (n+k)^2(n+1)^2(n+k+1)^2 (n+2k)(n+2)(2+ 2k+3n)(n^2-4),\end{equation} which appear in the denominator of $B(n,k)^2$. Here we are using the fact that at most one factor of $n$ can appear in the denominator as a result of contractions of the form \eqref{ope:qq}, but this factor will then cancel the factor of $n$ appearing in the numerator of $B(n,k)^2$. Therefore without loss of generality we may assume that the denominator of $p(n,k)$ is given by \eqref{pdenom}.

Note also that $p(n,k)$ is the same as the coefficient of $-B(n,k)  k(n+k)(n+2k) :J^a_{(2)} Q^a_{(2)}:$ appearing in $W^3_{(3)} W^4$. The only contributions to this coefficient will come from $$A_{(3)} (B_{(1)} C),$$ where $A,B,C$ are terms appearing in \eqref{def:cosetw3}, and two of the three depend either on $Q^a_{(2)}$ of $J^a_{(2)}$. All such terms are divisible by $(n+k)^2$. Therefore without loss of generality, we may write 
\begin{equation} \label{pnksecond} p(n,k) =  \frac{ (n+1)^2 (n+2)^2 (n+k+1)^2 (k + n)^2  q(n,k)  + (n+k)^2 r(n,k)}{(n-2)(n+1)^2(n+2)^2 (n+k)^2 (2k+n) (1 + k + n)^2(2 + 2 k + 3 n)},\end{equation} where $r(n,k)$ is a polynomial function in $n,k$. We now need to show that $r(n,k) = 0$.

First, it is easy to verify by computer that Theorem \ref{thm:prinw} holds for $3 \leq n\leq 7$. Therefore \eqref{formpnk} must hold when $3\leq n \leq 7$ for generic $k$, so $r(n,k)$ is divisible by $(n-3)(n-4)(n-5)(n-6)(n-7)$. Similarly, in view of Theorems \ref{thm:sl2para} and \ref{thm:aly}, \eqref{formpnk} must hold whenever $k+2(1+n)= 0$, so $r(n,k)$ is divisible by $k+2(1+n)$. In view of Theorems \ref{thm:bpcoset} and \ref{thm:acl}, $r(n,k)$ is divisible by both $ 2k+n-1$ and $2k+3n+3$. In view of Theorems \ref{thm:sl4subregcoset} and \ref{thm:cl}, $r(n,k)$ is divisible by $3k+2n-1$ and $3k+4n+4$. 

Therefore if $r(n,k) \neq 0$, it must have degree at least $10$ in the variable $n$, so the numerator of $p(n,k)$, written as in \eqref{pnksecond}, must have degree at least $12$ in $n$. But if we examine the contribution of each term of the form $$A_{(3)} (B_{(1)} C)$$ where $A,B,C$ are the summands in \eqref{def:cosetw3}, to the coefficient of $$-B(n,k)  k(n+k)(n+2k) :J^a_{(2)} Q^a_{(2)}:$$ appearing in $W^3_{(3)} W^4$, we find that the maximum degree in $n$ of the numerator of $p(n,k)$ is $11$. It follows that $r(n,k) = 0$, and the theorem is proved. \end{proof}

\begin{cor} Since $$\lim_{n\ra \infty} \frac{(n-1) (n+1)}{(n-2) (3 n^2  - n -2+ c (n + 2))} = 0,$$ there is a meaningful limit of the vertex algebras $\cW^k(\gs\gl_n, f_{\rm{prin}})$ as $n\ra \infty$, which is obtained from $\cW(c,\lambda)$ by setting $\lambda = 0$. This vertex algebra is of type $\cW(2,3,\dots)$ with strong generators $\{L, W^i|\ i\geq 3\}$, and the structure constants are polynomials in $c$. However, it is not isomorphic to the coset of $\cH$ inside the $\cW_{1+\infty}$-algebra, which is also of type $\cW(2,3,\dots)$; see Theorem \ref{thm:winfreal}. \end{cor}

We remark that by computing the irreducible factors of the Shapovalov determinants $\text{det}_N$ for $N\leq 7$, all one-parameter vertex algebras of type $\cW(2,3,\dots, N)$ satisfying \eqref{ope:first}-\eqref{ope:fourth} and \eqref{ope:w2w6}-\eqref{ope:w4w5}, can be classified. It turns out that the complete list is $\cW^k(\gs\gl_n, f_{\text{prin}})$ for $n\leq 7$, together with the cosets $N^k(\gs\gl_2)$, and $\cC^{\ell}$ given by Theorems \ref{thm:sl2para} and \ref{thm:bpcoset}. The classification for $N\leq 6$ was known previously in the physics literature \cite{B-V,BS,HI}, but the result for $N=7$ appears to be new. Note that $N^k(\gs\gl_2)$ and $\cC^{\ell}$ are not freely generated, so the folklore conjecture that the only freely generated ones are the principal $\cW$-algebras of type $A$ is true for $N \leq 7$.

\section{Generalized parafermions} \label{section:genpara}
We now consider another family of vertex algebras $\cC^k(n)$ arising as quotients of $\cW(c,\lambda)$, which we shall call {\it generalized parafermion algebras}. For $n\geq 1$, the embedding $\gg\gl_n \rightarrow \gs\gl_{n+1}$ defined by $$ a \mapsto \bigg(\begin{matrix} a & 0\\ 0 & -\text{tr}(a) \end{matrix}\bigg),$$ induces a vertex algebra homomorphism $$V^k(\gg\gl_n) \rightarrow V^k(\gs\gl_{n+1}).$$ We define
\begin{equation} \cC^k(n) = \text{Com}(V^k(\gg\gl_n), V^k(\gs\gl_{n+1})).\end{equation} Clearly $\cC^k(1) \cong N^k(\gs\gl_2)$, and $\cC^k(n)$ has Virasoro element $L^{\gs\gl_{n+1}} - L^{\gg\gl_n}$ with central charge 
\begin{equation} \label{ccgenpar} c  = \frac{n (k-1) (1 + n + 2 k)}{(n + k) (1 + n + k)}.\end{equation}
The following result was conjectured in \cite{B-H}, and generalizes the fact that $N^k(\gs\gl_2)$ is of type $\cW(2,3,4,5)$.

\begin{thm} \label{genpara:strong} For all $n\geq 1$, $\cC^k(n)$ is of type $\cW(2,3,\dots, n^2+3n+1)$ and is generated by the Virasoro field $L$ and a weight $3$ primary field $W^3$, for generic values of $k$. \end{thm}

\begin{proof} By Theorem 6.10 of \cite{CLII}, we have $$\lim_{k\rightarrow \infty} \cC^k(n) \cong \cH(2n)^{GL_n},$$ and a strong generating set for $\cH(2n)^{GL_n}$ will correspond to a strong generating set for $\cC^k(n)$ for generic values of $k$. Here $\cH(2n)$ denotes the rank $2n$ Heisenberg vertex algebra. The method of describing $\cH(2n)^{GL_n}$ is similar to the description of orbifolds of other free field algebras in \cite{LII} and we only sketch the proof. First, $\cH(2n)$ has a good increasing filtration \cite{LiIII}, such that the associated graded algebra $\text{gr}(\cH(2n)^{GL_n})$ is isomorphic to the classical invariant ring $$\big(\text{Sym} \bigoplus_{i \geq 0} (V_i \oplus V_i^*)\big)^{GL_n},$$ where $V_i = \mathbb{C}^n$ as $GL_n$-modules and $V^*_i \cong (\mathbb{C}^n)^*$. Generators and relations for this ring are given by Weyl's first and second fundamental theorems of invariant theory for the standard representation of $GL_n$ \cite{We}. The generators $q_{i,j}$ are quadratic and correspond to the pairing between $V_i$ and $V^*_j$, and the relations are $(n+1) \times (n+1)$ determinants $d_{I,J}$ corresponding to lists of indices $I = (i_0, i_1,\dots, i_n)$ and $J = (j_0, j_1,\dots, j_n)$ satisfying $$0\leq i_0< i_i < \cdots < i_n,\qquad 0\leq j_0 < j_1 < \cdots < j_n.$$ The corresponding fields $\omega_{i,j}$, which have weight $2+i+j$ are then a strong generating set for $\cH(2n)^{GL_n}$. Also, each relation $d_{I,J}$ corresponds to a normally ordered relation $D_{I,J}$ of weight $|I| +|J| + 2n+2$, whose leading term is $d_{I,J}$ in an appropriate sense. Here $|I| = i_0 + i_1+\dots + i_n$ and $|J| = j_0 + j_1 + \dots + j_n$.

There is some redundancy in the strong generating set $\{\omega_{i,j}|\ i,j \geq 0\}$ because $\partial \omega_{i,j} = \omega_{i+1, j} + \omega_{i,j+1}$, and the smaller set $\{\omega_{0,i}|\ i\geq 0\}$ suffices. The first normally ordered relation among these generators is $$D_{I_0,J_0},\qquad I_0 = (0,1,\dots, n) = J_0,$$ which has weight $n^2+3n+2$. The key step in the proof is showing that the coefficient of  $\omega_{0,n^2+3n}$ in this relation is independent of all choices of normal ordering, and is nonzero. This involving finding a recursive formula for the coefficient of $\omega_{0,m-2}$ in $D_{I,J}$ whenever the weight $m = |I| +|J| + 2n+2$ is even. Therefore up to rescaling, $D_{I_0, J_0}$ has the form 
\begin{equation} \label{gpara:firstrelation} \omega_{0,n^2+3n} - P(\omega_{0,0}, \omega_{0,1},\dots, \omega_{0,n^2+3n-1})=0,\end{equation} for some normally ordered polynomial $P$ in $\{\omega_{0,i}|\ 0\leq i \leq n^2+3n-1\}$, and their derivatives. This allows $\omega_{0,n^2+2n}$ to be eliminated. By applying the operator $(\omega_{0,1})_{(1)}$ repeatedly to \eqref{gpara:firstrelation} one can construct similar relations $$\omega_{0,m} - P_m(\omega_{0,0},\omega_{0,1}, \dots, \omega_{0,n^2+3n-1})=0,$$ for all $m > n^2+3n$. This shows that $\cH(2n)^{GL_n}$ has a minimal strong generating set $$\{\omega_{0,i}|\ 0\leq i\leq n^2+3n-1\},$$ and in particular is of type $\cW(2,3,\dots, n^2+3n+1)$. The fact that the weight $3$ field can be chosen to be primary and generates the algebra is easy to verify. Finally, the statement that $\cC^k(n)$ inherits these properties of $\cH(2n)^{GL_n}$ for generic values of $k$ is also clear; the argument is similar to the proof of Corollary 8.6 of \cite{CLI}. \end{proof}

\begin{cor} \label{cor:genpara} For all $n\geq 1$, there exists an ideal $I\subseteq \mathbb{C}[c,\lambda]$ and a localization $R$ of $\mathbb{C}[c,\lambda] / I$ such that $\cC^k(n)$ is the simple quotient of $\cW^{I}_{R}(c,\lambda)$.
\end{cor}

\begin{proof} This holds for $n=1$ by Theorem \ref{thm:sl2para}. For $n>1$, the simplicity of $\cC^k(n)$ as a vertex algebra over $\mathbb{C}[k]$ follows from the simplicity of $\cH(2n)^{GL_n}$, which can be deduced from \cite{DLM} if $GL_n$ is replaced by a compact form. In view of Theorems \ref{refinedsimplequotient} and \ref{genpara:strong}, it then suffices to show that the graded characters of $\cC^k(n)$ and $\cW(c,\lambda)$ agree up to weight $8$. This follows from Weyl's second fundamental theorem of invariant theory for $GL_n$, since there are no relations among the generators of weight less than $n^2+3n+2$. \end{proof}

We now give an explicit description of these ideals. For $n>1$, let $K_n \subseteq \mathbb{C}[c,\lambda]$ be the ideal generated by the following polynomial:
\begin{equation} \label{gpara:ideal}  \begin{split} p_n(c,\lambda) & = 
-3 c^2 \lambda - 12 c^2 \lambda^2 + 12 c^3 \lambda^2 - 2 c n - 13 c \lambda n + 10 c^2 \lambda n - 20 c \lambda^2 n + 28 c^2 \lambda^2 n - 8 c^3 \lambda^2 n  
\\ & + 4 n^2   - c n^2 + 20 \lambda n^2 - 34 c \lambda n^2 + 8 c^2 \lambda n^2 + 25 \lambda^2 n^2 - 45 c \lambda^2 n^2 + 99 c^2 \lambda^2 n^2 - 7 c^3 \lambda^2 n^2 
\\ & + 2 n^3     + 2 c n^3     + \lambda n^3 + 18 c \lambda n^3 - 4 c^2 \lambda n^3 - 10 \lambda^2 n^3 + 154 c \lambda^2 n^3 - 2 c^2 \lambda^2 n^3 + 2 c^3 \lambda^2 n^3  
\\ &  - 4 n^4 + c n^4 - 59 \lambda n^4   + 37 c \lambda n^4 - 2 c^2 \lambda n^4 - 163 \lambda^2 n^4 + 267 c \lambda^2 n^4 - 33 c^2 \lambda^2 n^4 + c^3 \lambda^2 n^4  
\\ & - 2 n^5 - 31 \lambda n^5    + 10 c \lambda n^5 - 92 \lambda^2 n^5 + 100 c \lambda^2 n^5 - 8 c^2 \lambda^2 n^5 - 3 \lambda n^6 - 12 \lambda^2 n^6 + 12 c \lambda^2 n^6.\end{split} \end{equation}
The corresponding variety $V(K_n) \subseteq \mathbb{C}^2$ is a rational curve with parametrization

\begin{equation} \label{lambdagenpar} \lambda  = \frac{(n + k) (1 + n + k)}{(k-2 ) (2 n + k) (2 + 2 n + 3 k)},\qquad c = \frac{n (k-1) (1 + n + 2 k)}{(n + k) (1 + n + k)} .\end{equation}

\begin{thm} \label{thm:genpara} For all $n> 1$, let $K_n$ be the ideal generated by $p_n(c,\lambda)$, as above.
\begin{enumerate}

\item The generator $p_n(c,\lambda)$ lies in the Shapovalov spectrum of level $n^2+3n+2$.
\item There exists a localization $R_n$ of $\mathbb{C}[c,\lambda] / K_n$ such that $\mathcal{W}^{K_n}_{R_n}(c,\lambda)$ has a unique up to scalar singular vector in weight $n^2+3n+2$ of the form
$$W^{n^2+3n+2} - P(L, W^3, \dots, W^{n^2+3n}).$$ 

\item Letting $\cK_n$ be the maximal proper graded ideal of $\cW^{K_n}_{R_n}(c,\lambda)$, we have
$$\cW^{K_n}_{R_n}(c,\lambda)/ \cK_n \cong \mathcal{C}^k(n).$$ 
\end{enumerate}

\end{thm}

\begin{proof}
Let $n$ be fixed. Clearly there is some rational function $\lambda_n(k)$ of $k$ such that $\cC^k(n)$ is obtained from $\cW(c,\lambda)$ by setting $\displaystyle c = \frac{n (k-1) (1 + n + 2 k)}{(n + k) (1 + n + k)}$ and $\lambda = \lambda_n(k)$, and then taking the simple quotient. It is not obvious yet that $\lambda_n(k)$ is a rational function of $n$ as well.

For $k$ a positive integer, it is well known that map $V^k(\gg\gl_n) \ra V^k(\gs\gl_{n+1})$ descends to a homomorphism of simple algebras $L_k(\gg\gl_n) \ra L_k(\gs\gl_{n+1})$. Set $$\cC_k(n) =  \text{Com}(L_k(\gg\gl_n), L_k(\gs\gl_{n+1})).$$ By Theorem 8.1 of \cite{CLII}, $\cC_k(n)$ is a simple vertex algebra and the map $\cC^k(n) \ra \cC_k(n)$ is surjective, so $\cC_k(n)$ is the simple quotient of $\cC^k(n)$.
Next, by Theorem 13.1 of \cite{ACLII}, for all positive integers $k$ and $n$, we have 
\begin{equation} \label{thm:genparaACL} \text{Com}(L_{k}(\gg\gl_n), L_{k}(\gs\gl_{n+1})) \cong \cW_{k'} (\gs\gl_k, f_{\text{prin}}),\qquad k' = -k + \frac{1 + k + n}{k + n},\  -k + \frac{k + n}{1 + k + n}.\end{equation}
It follows from Theorem \ref{thm:prinw} that whenever $k$ is a positive integer, $$\lambda_n(k) = \frac{ (n + k) (1 + n + k)}{(k-2 ) (2 n + k) (2 + 2 n + 3 k)},$$ and since $ \lambda_n(k)$ is a rational function of $k$, this equality holds for all $k$ where it is defined. This completes the proof. \end{proof}

\begin{remark} We expect that for all $n>1$, the above singular vector of weight $n^2+3n+2$ generates the maximal graded ideal $\cK_n$, but we do not prove this. 
\end{remark}

\begin{remark} In the case $n=1$, \eqref{gpara:ideal} specializes to $$p_1(c,\lambda) = 9 \lambda \big(4 \lambda (c + 7) (2 c - 1) + (c - 2) (c + 4)\big).$$ Unlike the case $n>1$, this does not divide $\text{det}_6$, but the irreducible factor $4 \lambda (c + 7) (2 c - 1) + (c - 2) (c + 4)$ does; see Theorem \ref{thm:sl2para}.
\end{remark}

\section{Families of ideals in the Shapovalov spectrum} \label{section:shap}
It is an interesting problem to compute the irreducible factors of $\text{det}_n$ for all $n$. A natural question is whether these factors are organized into families that admit a uniform description. In the previous section, we found two infinite families \eqref{ideal:wslna} and \eqref{gpara:ideal} of principal ideals of $\mathbb{C}[c,\lambda]$, which lie in the Shapovalov spectrum and correspond to families of vertex algebras of type $\cW(2,3,\dots, N)$.  In this section, we introduce two more such families of ideals which we conjecture to lie in the Shapovalov spectrum, and to correspond to vertex algebras of type $\cW(2,3,\dots, N)$. Another family will be discussed in Section \ref{section:oneplus} in the context of the $\cW_{1+\infty}$-algebra with negative integer central charge.

A remarkable feature of these ideals $I$ is that the corresponding varieties $V(I) \subseteq \mathbb{C}^2$ are all {\it rational curves}, possibly singular. It is no doubt speculative, but we expect that {\it all} ideals in the Shapovalov spectrum of $\cW(c,\lambda)$ are organized into similar infinite families, and all correspond to rational curves. Similar questions can be asked for other two-parameter vertex algebras such as the affine vertex algebra associated to the exceptional one-parameter Lie superalgebra $D(2,1;\alpha)$, as well as its orbifolds, cosets, and quantum Hamiltonian reductions.

\subsection{Cosets of minimal $\cW$-algebras in type $A$} 
For $n\geq 4$, recall that $\mathcal{W}^k(\mathfrak{s} \mathfrak{l}_{n}, f_{\text{min}})$ contains a copy of $V^{k+1}(\gg\gl_{n-2})$. The coset
\begin{equation} \mathcal{C}^k(n) = \text{Com}(V^{k+1}(\mathfrak{g} \mathfrak{l}_{n-2}), \mathcal{W}^k(\mathfrak{s} \mathfrak{l}_{n}, f_{\text{min}})),\end{equation} which has central charge \begin{equation} \label{cc:minwcoset} c =-\frac{(1 + k) (2 k + n-1) (3 k + 2 n)}{(k + n-1) (k + n)},\end{equation} was studied in \cite{ACKL}.

\begin{thm} \label{thm:minwcoset} For all $n\geq 4$, there exists an ideal $K_n \subseteq \mathbb{C}[c,\lambda]$ and a localization $R_n$ of $\mathbb{C}[c,\lambda] / K_n$ such that the following hold.
\begin{enumerate} 
\item The vertex algebra $\mathcal{W}^{K_n}_{R_n}(c,\lambda)$ has a singular vector in weight $n^2-1$ of the form 
$$W^{n^2-1} - P(L, W^3, \dots, W^{n^2-3}).$$ 
\item Letting $\cK_n$ be the maximal proper graded ideal of $\cW^{K_n}_{R_n}(c,\lambda)$, we have $$\cW^{K_n}_{R_n}(c,\lambda)/ \cK_n \cong \mathcal{C}^k(n),$$ where $k$ and $c$ are related by \eqref{cc:minwcoset}.
\end{enumerate}
\end{thm}

\begin{proof} In the case $n=4$, it was shown explicitly in \cite{ACKL} that $\cC^k(4)$ is generated by the Virasoro field $L$ and a weight $3$ primary field $W^3$, and a similar argument shows that this holds for all $n$. By the same argument as the proof of Corollary \ref{cor:genpara}, $\cC^k(n)$ is simple. Also, it follows from Weyl's second fundamental theorem of invariant theory for the standard representation of $GL_{n-2}$ that there are no relations among the generators of $\cG_{\text{ev}}(n-2)^{GL_{n-2}}$ at weight $w \leq n^2-2$, so the graded characters of $\cC^k(n)$ and $\cW(c,\lambda)$ agree in weights $w\leq n^2-2$. Here $\cG_{\text{ev}}(n-2)$ is the generalized free field algebra introduced in Section 4 of \cite{ACKL}. By Theorem \ref{refinedsimplequotient}, $\cC^k(n)$ is then the simple quotient of $\cW^{K_n}_{R_n}(c,\lambda)$ for some $K_n$ and $R_n$. Finally, by Theorem 5.1 of \cite{ACKL}, $\cC^k(n)$ is of type $\mathcal{W}(2,3,\dots, n^2-2)$ for generic values of $k$. This was proven by passing to the limit
$$\lim_{k \ra \infty} \cC^k(n) \cong \cG_{\text{ev}}(n-2)^{GL_{n-2}}.$$ Since the weight $n^2-1$ field decouples, the singular vector in $\cW^{K_n}_{R_n}(c,\lambda)$ must occur in weight $n^2-1$ and have the desired form.
\end{proof}

For all integers $k\geq 0$, it was shown in \cite{ACKL} that there is a map of simple vertex algebras
$L_{k+1}(\gg\gl_{n-2}) \rightarrow \cW_k(\gs\gl_n, f_{\text{min}})$, and that the simple quotient of $\cC^k(n)$ coincides with the coset
$$\cC_k(n) = \text{Com}(L_{k+1}(\gg\gl_{n-2}) , \cW_k(\gs\gl_n, f_{\text{min}})).$$ The following conjecture is due to Kawasetsu (K. Kawasetsu, private communication 2015); it is based on the equality of central charges.
\begin{conj} \label{conj:kawasetsu} For all integers $n\geq 4$ and $k\geq 0$, $$\mathcal{C}_k(n) \cong \mathcal{W}_{k'}(\mathfrak{s} \mathfrak{l}_{2k+n}, f_{\rm{prin}}),\quad k' = -(2 k + n) +  \frac{k + n-1}{k + n},\  -(2 k + n) +  \frac{k + n}{k + n-1}.$$
\end{conj}

This conjecture is equivalent to a much stronger conjecture, namely, the explicit generator of the ideal $K_n$ for all $n\geq 4$.

\begin{conj} \label{conj:minwcoset} 
For all $n\geq 4$, $K_n$ is generated by the polynomial
\begin{equation} \label{ideal:minw} \begin{split} p_n(c,\lambda) & = 3 + 108 \lambda - 12 c \lambda + 1152 \lambda^2 - 384 c \lambda^2 + 3072 \lambda^3 - 3072 c \lambda^3 - 4 n  - 144 \lambda n + 20 c \lambda n 
\\ & - 1536 \lambda^2 n  + 592 c \lambda^2 n - 16 c^2 \lambda^2 n - 4096 \lambda^3 n + 4352 c \lambda^3 n - 256 c^2 \lambda^3 n - 2 n^2 - 28 \lambda n^2 
\\ & - 2 c \lambda n^2 + 100 \lambda^2 n^2  - 152 c \lambda^2 n^2 + 4 c^2 \lambda^2 n^2 + 576 \lambda^3 n^2 - 1920 c \lambda^3 n^2 + 192 c^2 \lambda^3 n^2 
\\ & + 4 n^3 + 94 \lambda n^3 - 8 c \lambda n^3 + 548 \lambda^2 n^3 - 144 c \lambda^2 n^3 + 4 c^2 \lambda^2 n^3 + 1088 \lambda^3 n^3 + 32 c \lambda^3 n^3 
\\ & + 32 c^2 \lambda^3 n^3 - n^4 - 32 \lambda n^4 + 2 c \lambda n^4 - 199 \lambda^2 n^4 + 80 c \lambda^2 n^4 - c^2 \lambda^2 n^4 - 384 \lambda^3 n^4 
\\ & + 144 c \lambda^3 n^4 - 48 c^2 \lambda^3 n^4 + 2 \lambda n^5 + 4 \lambda^2 n^5 - 10 c \lambda^2 n^5 - 16 \lambda^3 n^5 + 8 c \lambda^3 n^5 + 8 c^2 \lambda^3 n^5 \\ & + 3 \lambda^2 n^6 + 12 \lambda^3 n^6 - 12 c \lambda^3 n^6.
\end{split} \end{equation}
\end{conj}
The corresponding variety $V(K_n) \subseteq \mathbb{C}^2$ is a rational curve with parametrization
\begin{equation} \label{eq:paraminw} c =-\frac{(1 + k) (2 k + n-1) (3 k + 2 n)}{(k + n-1) (k + n)}, \qquad \lambda = \frac{ (k + n-1) (k + n)}{(n-2) (2 k + n-2) (4 k + 3 n)}.\end{equation} The proof that Conjectures \ref{conj:kawasetsu} and \ref{conj:minwcoset} are equivalent is a consequence of Theorem \ref{thm:prinw}. It is similar to the proof of Theorem \ref{thm:genpara}, and is left to the reader. In the case $n=4$, we have verified Conjecture \ref{conj:minwcoset} by computer calculation.

\subsection{Cosets of subregular $\cW$-algebras of type $A$} For $n\geq 4$, recall that $\cW^k(\gs\gl_n, f_{\text{subreg}})$ contains a Heisenberg algebra $\cH$. Let 
\begin{equation} \label{coset:subreg} \cC^k(n) = \text{Com}(\cH, \cW^k(\gs\gl_n, f_{\text{subreg}})),\end{equation} which has central charge 
\begin{equation} \label{cc:subreg} c  = -\frac{n ( n^2 - 2 k + n k  - 3 n +1) (n^2 - k + n k - 2 n-1)}{n + k}.\end{equation}

\begin{conj}  \label{conj:subregtype} For all $n\geq 4$, there exists an ideal $K_n \subseteq \mathbb{C}[c,\lambda]$ and a localization $R_n$ of $\mathbb{C}[c,\lambda] / K_n$ such that the following hold.
\begin{enumerate} 
\item The vertex algebra $\cW^{K_n}_{R_n}(c,\lambda)$ has a singular vector in weight $2n+2$ of the form 
$$W^{2n+2} - P(L, W^3, \dots, W^{2n}).$$ 
\item Letting $\cK_n$ be the maximal proper graded ideal of $\cW^{K_n}_{R_n}(c,\lambda)$, we have $$\cW^{K_n}_{R_n}(c,\lambda)/ \cK_n \cong \cC^k(n),$$ where $k$ and $c$ are related by \eqref{cc:subreg}.
\end{enumerate}
\end{conj}
In particular, this conjecture implies that $\cC^k(n)$ is of type $\cW(2,3,\dots, 2n+1)$, which is known in the case $n=4$ \cite{CLIII}.

Let $\cC_k(n)$ denote the simple quotient of $\cC^k(n)$, which coincides with $\text{Com}(\cH, \cW_k(\gs\gl_n, f_{\text{subreg}}))$. The following conjecture is due to \cite{B-H}; see also Conjecture 6.1 of \cite{CLIII}.

\begin{conj} \label{conj:subregphysics} 
For all $n\geq 4$ and $m\geq 3$, we have 
$$\cC_k(n) \cong \cW_{k'}(\gs\gl_m, f_{\rm{prin}}),\qquad k = -n + \frac{m + n}{n-1},\qquad k' = -m + \frac{m +1}{m + n}.$$
\end{conj}

In conjunction with Theorem \ref{thm:prinw}, this conjecture is equivalent to the following explicit description of the ideal $K_n$ in Conjecture \ref{conj:subregtype}, or rather, the variety $V(K_n) \subseteq \mathbb{C}^2$.
\begin{conj} \label{conj:subregcoset} For all $n\geq 4$, $V(K_n)$ has the following rational parametrization with parameter $k$.
\begin{equation} \label{subreg:ideal} \begin{split} c & = -\frac{n ( n^2 - 2 k + n k  - 3 n +1) (n^2 - k + n k - 2 n-1)}{n + k},\\ & \lambda = -\frac{ (n + k-1) (n + k)}{( n^2 - 3 k + n k - 4 n + 2) (n^2 - k + n k  - 2 n-2) (n^2 + k + n k)}.\end{split} \end{equation}
\end{conj}
In the case $n=4$, this holds by Theorem \ref{thm:sl4subregcoset}.

\section{Quotients of $\cW(c,\lambda)$ by maximal ideals and coincidences between algebras of type $\cW(2,3,\dots, N)$} \label{section:coincidences}
So far, we have considered quotients of the form $\cW^I_R(c,\lambda)$ which are one-parameter vertex algebras in the sense that $R$ has Krull dimension $1$. Here, we consider simple quotients of $\cW^I(c,\lambda)$ where $I\subseteq \mathbb{C}[c,\lambda]$ is a {\it maximal} ideal. Such an ideal always has the form $I = (c- c_0, \lambda- \lambda_0)$ for $c_0, \lambda_0\in \mathbb{C}$, and $\cW^I(c,\lambda)$ is a vertex algebra over $\mathbb{C}$. We first need a criterion for when the simple quotients of two such vertex algebras are isomorphic. 
\begin{thm} \label{thm:coincidences} Let $c_0,c_1, \lambda_0, \lambda_1$ be complex numbers and let $$I_0 = (c- c_0, \lambda- \lambda_0),\qquad I_1 = (c - c_1, \lambda - \lambda_1)$$ be the corresponding maximal ideals in $\mathbb{C}[c,\lambda]$. Let $\cW_0$ and $\cW_1$ be the simple quotients of $\cW^{I_0}(c,\lambda)$ and $\cW^{I_1}(c,\lambda)$. Then $\cW_0 \cong \cW_1$ are isomorphic only in the following three cases.
\begin{enumerate}
\item $c_0 = c_1$ and $\lambda_0= \lambda_1$.
\item $c_0 = 0 = c_1$ and no restriction on $\lambda_0, \lambda_1$.
\item $c_0 = -2 = c_1$ and no restriction on $\lambda_0, \lambda_1$.
\end{enumerate} 
\end{thm}

\begin{proof} Clearly the isomorphism holds in case (1). It holds in case (2) as well because the simple quotient is $\mathbb{C}$ for all $\lambda$. As for case (3), let $J \subseteq \mathbb{C}[c,\lambda]$ be the ideal $(c+2)$, and let $\cJ \subseteq \cW^{J}(c,\lambda)$ be the ideal generated by 
\begin{equation} \label{sing:w23second} W^4 - \frac{8}{3} :LL:   + \frac{1}{2} \partial^2 L.\end{equation} 
This is easily seen to be a singular vector in $\cW^{J}(c,\lambda)$, so by Lemma \ref{decoup}, $\cW^{J}(c,\lambda)/ \cJ$ is of type $\cW(2,3)$ with strong generators $L, W^3$. It follows from \eqref{ope:first}-\eqref{ope:fourth} that $L, W^3$ satisfy the OPE relations of the Zamolodchikov $\cW_3$-algebra with $c=-2$. Therefore the simple quotient of $\cW^{J}(c,\lambda)/\cJ$ is isomorphic to the simple Zamolodchikov algebra $\cW_{3,-2}$ for all $\lambda$, so the isomorphism holds in case (3). 

Conversely, suppose there is another case where $\cW_0 \cong \cW_1$. Necessarily $c_0 = c_1\neq 0, -2$, and $\lambda_0 \neq \lambda_1$. Since $c \neq -2$, by \eqref{ope:fourth} the coefficient of $W^3$ in $W^3_{(3)} W^4$ depends nontrivially on $\lambda$, and since $c\neq 0$, $\cW^3$ is not a singular vector. Therefore by \eqref{ope:fourth}, $\cW_0 \cong \cW_1$ implies that $\lambda_0 = \lambda_1$, so we are in case (1). \end{proof}

This result should be contrasted with the phenomenon of {\it triality}, which says that for generic values of $c$, there are three distinct values of $\mu$ which give rise to the same algebra \cite{GGII}. Here $\mu$ and $\lambda$ are related by \eqref{intro:lambdaofmu}.

\begin{cor} \label{cor:intersections} Let $I = (p)$, $J = (q)$ be prime ideals in $\mathbb{C}[c,\lambda]$ which lie in the Shapovalov spectrum of $\cW(c,\lambda)$. Then aside from the coincidences at $c = 0, -2$ given by Theorem \ref{thm:coincidences}, any additional pointwise coincidences between the simple quotients of $\cW^{I}(c,\lambda)$ and $\cW^{J}(c,\lambda)$ must correspond to intersection points of the truncation curves $V(I) \cap V(J)$. \end{cor}

\begin{cor} \label{cor:uniqueness} Suppose that $\cA$ is a simple, one-parameter vertex algebra which is isomorphic to the simple quotient of $\cW^{I}(c,\lambda)$ for some prime ideal $I = (p)\subseteq \mathbb{C}[c,\lambda]$, possibly after localization. Then if $\cA$ is the quotient of $\cW^{J}(c,\lambda)$ for some prime ideal $J$, possibly localized, we must have $I = J$. \end{cor}

\begin{proof} This is immediate from Theorem \ref{thm:coincidences} and Corollary \ref{cor:intersections}, since if $I$ and $J$ are distinct prime ideals in $\mathbb{C}[c,\lambda]$, their truncation curves $V(I)$ and $V(J)$ can intersect at at most finitely many points. The simple quotients of $\cW^{I}(c,\lambda)$ and $\cW^{J}(c,\lambda)$ therefore cannot coincide as one-parameter families. \end{proof}

For the rest of this section, we shall use Corollary \ref{cor:intersections} to establish many more coincidences of the kind given by Theorems \ref{thm:aly}, \ref{thm:acl}, and \ref{thm:cl}, not necessarily among rational or $C_2$-cofinite vertex algebras. Let $I$ be an ideal in the Shapovalov spectrum and let $\cW^I(c,\lambda)$ be the corresponding one-parameter vertex algebra. Let $\cA^k$ be a simple one-parameter vertex algebra which is isomorphic to the simple quotient $\cW^I(c,\lambda)/ \cI$, via 
\begin{equation} \label{isoviaparametrization} L \mapsto \tilde{L}, \qquad W^3 \mapsto \tilde{W}^3, \qquad c\mapsto c(k), \qquad \lambda\mapsto \lambda(k). \end{equation} Here $\{\tilde{L}, \tilde{W}^3\}$ are the standard generators of $\cA^k$, where $\displaystyle \tilde{W}^3_{(5)} \tilde{W}^3 = \frac{c(k)}{3}$, and $k \mapsto (c(k), \lambda(k))$ is a rational parametrization of the curve $V(I)$. Recall that all structure constants in the OPEs $\tilde{W}^i(z) \tilde{W}^j(w)$ are polynomials in $c(k)$ and $\lambda(k)$, where $\tilde{W}^i = \tilde{W}^3_{(1)}\tilde{W}^{i-1}$ for $i \geq 4$. It follows that $\cA^k$ is well defined and the isomorphism \eqref{isoviaparametrization} holds for all values of $k$ except for the poles of $c(k)$ and $\lambda(k)$. This should be contrasted with statements in previous sections that required us to work over a localization $R$ of the ring $\mathbb{C}[c,\lambda] / I$. The reason we localized was to obtain a singular vector of the form $W^{N+1} - P(L, W^3, \dots, W^{N-1})$ in $\cW^I_R(c,\lambda)$, so that the simple quotient $\cW^I_R(c,\lambda)/ \cI$ truncates to an algebra of type $\cW(2,3,\dots N)$. Here we prefer {\it not} to localize; the price we pay is that $\cW^I(c,\lambda)/ \cI$ need not truncate to $\cW(2,3,\dots, N)$ for any $N$, but the benefit is that we do not need to exclude any values of $k$ except for the poles of $c(k)$ and $\lambda(k)$.

There are a few subtle points that need to be mentioned. For generic $k$, the examples $\cA^k$ in this paper are either $\cW^k(\gs\gl_n, f_{\text{prin}})$, or a coset of the form $\text{Com}(\cH, \tilde{\cA}^k)$ or $\text{Com}(V^k(\gg), \tilde{\cA}^k)$, for some vertex algebra $\tilde{\cA}^k$. Even though $\cA^k$ is well defined for all $k$ away from the poles of $c(k)$ and $\lambda(k)$, the specialization of $\cA^{k}$ at a value $k = k_0$, can fail to coincide with the algebra of interest. This subtlety does not occur in the case of $\cW^k(\gs\gl_n, f_{\text{prin}})$ since it is freely generated of type $\cW(2,3,\dots, n)$ for all $k$, but it does occur for the examples arising as cosets. For example, consider $N^k(\gs\gl_2) = \text{Com}(\cH, V^k(\gs\gl_2))$, with $k$ regarded as a formal parameter. For all $k_0\neq 0$, the specialization of $N^k(\gs\gl_2)$ at $k=k_0$ coincides with $ \text{Com}(\cH, V^{k_0}(\gs\gl_2))$, and in particular has trivial weight one subspace. But at $k_0 = 0$, $\text{Com}(\cH, V^0(\gs\gl_2))$ is strictly larger since it contains a weight one field. The behavior of $\cC^k = \text{Com}(V^k(\gg), \tilde{\cA}^k)$ for simple $\gg$ can be more complicated, and was studied in \cite{CLII} in a general setting. The specialization of $\cC^k$ at $k = k_0$ can be a proper subalgebra of $\text{Com}(V^{k_0}(\gg), \tilde{\cA}^{k_0})$, but by Corollary 6.7 of \cite{CLII}, under mild hypotheses that are satisfied in all our examples, this can only occur for rational numbers $k_0 \leq -h^{\vee}$, where $h^{\vee}$ is the dual Coxeter number of $\gg$. By a slight abuse of notation, if $\cA^k$ is a one-parameter vertex algebra that generically coincides with a coset $\text{Com}(\cH, \tilde{\cA}^k)$ or $\text{Com}(V^k(\gg), \tilde{\cA}^k)$, by $\cA^{k_0}$ we will always mean the specialization of $\cA^k$ at the value $k = k_0$, even if is smaller than the actual coset. If $\cA^k$ is a one-parameter quotient of $\cW^I(c, \lambda)$, $\cA^{k_0}$ will then be a quotient of $\cW^I(c, \lambda)$, even if the actual coset at this point is not such a quotient.

The other subtlety is that even if $k_0$ is a pole of $c(k)$ or $\lambda(k)$, the algebra $\cA^k$ might still be well defined at $k = k_0$. For example, in the case $\cA^k = \cW^k(\gs\gl_n, f_{\text{prin}})$ and $I = I_n$ with generator given by \eqref{ideal:wslna}, this occurs at the critical level $k_0 = -n$, where $c(k)$ is undefined, and at $\displaystyle k_0 = -n + \frac{n}{n+2},\  -n +\frac{n+2}{n}$, where $\displaystyle c(k) = - \frac{(n-1)(3n+2)}{n+2}$, but $\lambda(k)$ is undefined. The fact that these three values of $k_0$ are poles of $c(k)$ or $\lambda(k)$ means that at these values, $\cW^{k_0}(\gs\gl_n, f_{\text{prin}})$ cannot be obtained as a quotient of $\cW^{I_n}(c,\lambda)$, even though it is well defined for all $k$. In general, we expect that at points where $\cA^{k}$ cannot be obtained as a quotient of $\cW^I(c,\lambda)$, it can still be realized as a quotient of a {\it suitable limit} of $\cW(c,\lambda)$. We plan to address this in a separate paper.

Suppose that $\cA^k, \cB^{\ell}$ are one-parameter vertex algebras that arise as the simple quotients of $\cW^I(c,\lambda), \cW^J(c,\lambda)$, respectively, and we wish to classify the coincidences between their simple quotients $\cA_{k_0}$ and $\cB_{\ell_0}$ at points $k_0, \ell_0$. By Corollary \ref{cor:intersections}, aside from the cases $c = 0, -2$, the coincidences for which $k_0$ and $\ell_0$ are not poles of $c(k), \lambda(k)$, and $c(\ell), \lambda(\ell)$, respectively, correspond to the intersection points in $V(I) \cap V(J)$. On the other hand, if $k_0$ is a pole of $c(k)$ or $\lambda(k)$, and $\ell_0$ is a pole of $c(\ell)$ or $\lambda(\ell)$, but $\cA^{k_0}, \cB^{\ell_0}$ are still defined, Corollary \ref{cor:intersections} does not apply, and we need different methods to determine if $\cA_{k_0}$ and $\cB_{\ell_0}$ are isomorphic.

\subsection{Coincidences between $\cW_{k}(\gs\gl_n,f_{\rm{prin}})$ and $\cW_{k'}(\gs\gl_m,f_{\rm{prin}})$}
The following result was conjectured by Gaberdiel and Gopakumar in \cite{GGII}.

\begin{thm} For $n,m \geq 2$ and $n \neq m$, we have the following isomorphisms
\begin{equation} \begin{split} & \mathcal{W}_{k}(\mathfrak{s}\mathfrak{l}_n, f_{\rm{prin}}) \cong \mathcal{W}_{k'}(\mathfrak{s}\mathfrak{l}_m, f_{\rm{prin}}),\\ & k = -n + \frac{m + n}{n},\  -n + \frac{n}{m+n},\qquad k' = -m + \frac{m  + n}{m}, \  -m + \frac{m}{m + n}.\end{split} \end{equation} Moreover, aside from the critical levels $k=-n$ and $k'=-m$, and the cases $c = 0, -2$, these are all such coincidences.
\end{thm}

\begin{proof} First, we assume $n,m\geq 3$, and we exclude the cases $$k =  -n + \frac{n}{n+2},\   -n + \frac{n+2}{n},\qquad k' =   -m + \frac{m}{m+2},\   -m + \frac{m+2}{m},$$ since at these values, $\cW_{k}(\gs\gl_n,f_{\text{prin}})$ and $\cW_{k'}(\gs\gl_m,f_{\text{prin}})$ are not quotients of $\cW(c,\lambda)$. For all other noncritical values of $k$ and $k'$, $\cW_{k}(\gs\gl_n,f_{\text{prin}})$ and $\cW_{k'}(\gs\gl_m,f_{\text{prin}})$ are obtained as quotients of $\cW^{I_n}(c,\lambda)$ and $\cW^{I_m}(c,\lambda)$, respectively. By Corollary \ref{cor:intersections}, aside from the coincidences at $c = 0,-2$, all other coincidences $\cW_{k}(\gs\gl_n,f_{\text{prin}}) \cong \cW_{k'}(\gs\gl_m,f_{\text{prin}})$ correspond to intersection points on the truncation curves $V(I_n)$ and $V(I_m)$. These ideals are described explicitly by \eqref{ideal:wslna}, and $V(I_n) \cap V(I_m)$ consists of exactly one point 
$$(c, \lambda) =\bigg( -\frac{(m-1) (n-1) (m + n + m n)}{m + n}, \  -\frac{m + n}{(m-2) (n-2) (2 m + 2 n + m n)}\bigg).$$ By replacing the parameter $c$ with the levels $k,k'$, we see that the above isomorphisms hold, and that these are the only such isomorphisms for $n,m \geq 3$, except possibly for the values of $k, k'$ excluded above. 

Next, assume that $n,m\geq 3$, and let $\displaystyle k =  -n + \frac{n}{n+2}$ or $ \displaystyle k = -n + \frac{n+2}{n}$. Then $\cW_k(\gs\gl_n, f_{\text{prin}})$ is isomorphic to the Virasoro algebra with $\displaystyle c = - \frac{(n-1)(3n+2)}{n+2}$, since the weight $3$ field $W^3 \in \cW^k(\gs\gl_n, f_{\text{prin}})$ is a singular vector, and $\cW^k(\gs\gl_n, f_{\text{prin}})$ is generated by $\{L, W^3\}$ for all noncritical values of $k$, by Proposition 5.2 of \cite{ALY} (see also \cite{FL}). If $\cW_k(\gs\gl_n, f_{\text{prin}})$ were isomorphic to $\cW_{k'}(\gs\gl_m, f_{\text{prin}})$ for some $m\neq n$, the weight $3$ field $W^3\in \cW^{k'}(\gs\gl_m, f_{\text{prin}})$ would also have to be singular. But for $c\neq 0$, $\cW_{k'}(\gs\gl_m, f_{\text{prin}})$ has a singular vector in weight $3$ only for $\displaystyle c = -\frac{(m-1) (3m+2)}{m+2}$, and for $m\neq n$ this cannot coincide with the above central charge. This proves the result for $n,m\geq 3$.

Finally, suppose that $n=2$, so that $\cW^k(\gs\gl_2, f_{\text{prin}})$ is just the Virasoro algebra with $\displaystyle c = -\frac{(1 + 2 k) (4 + 3 k)}{2 + k}$. Clearly $\cW^k(\gs\gl_2, f_{\text{prin}})$ cannot be obtained as a quotient of $\cW(c,\lambda)$ for $c \neq 0$, because the weight $3$ field $W^3$ is normalized so that $\displaystyle W^3_{(5)} W^3 = \frac{c}{3} 1$. If $\cW_k(\gs\gl_2, f_{\text{prin}})\cong \cW_{k'}(\gs\gl_m, f_{\text{prin}})$ for some $m\geq 3$, the weight $3$ field of $\cW^{k'}(\gs\gl_m, f_{\text{prin}})$ would also need to be singular. But for $c\neq 0$, $\cW^{k'}(\gs\gl_m, f_{\text{prin}})$ has a singular vector in weight $3$ only at $\displaystyle c = - \frac{(m-1)(3m+2)}{m+2}$, so we must have $\displaystyle k = -2 + \frac{2}{m+2},\ -2+ \frac{m+2}{2}$ and the result holds for $n=2$ as well.
\end{proof}

\begin{remark}
If $m,n$ are coprime, the levels $\displaystyle k =  -n + \frac{n}{m+n} $ and $\displaystyle k' =  -m + \frac{m}{m + n}$ are {\it boundary admissible} in the sense of \cite{KWV}. In particular, the above $\cW$-algebras are $C_2$-cofinite and rational by \cite{ArII,ArIII}. It $m$ and $n$ are not coprime, $k$ and $k'$ are not admissible, and it is an interesting problem to determine if $\cW_k(\gs\gl_n, f_{\rm{prin}})$ is still $C_2$-cofinite and rational. \end{remark}

\subsection{Coincidences between $N_k(\gs\gl_2)$ and $\cW_{k'}(\gs\gl_n, f_{\rm{prin}})$.}
Recall that as a one-parameter vertex algebra, the parafermion algebra $N^k(\gs\gl_2) = \text{Com}(\cH, V^k(\gs\gl_2))$ is obtained as the simple quotient of $\cW^{I}(c,\lambda)$ via the parametrization \eqref{ratpara:para}, where $I$ is generated by \eqref{curve:para}. Recall that the specialization $N^{k_0}(\gs\gl_2)$ of the one-parameter algebra $N^{k}(\gs\gl_2)$ at $k = k_0$, coincides with the coset $\text{Com}(\cH, V^{k_0}(\gs\gl_2))$ for all $k_0\neq 0$. By abuse of notation, we shall use the same notation  $N^k(\gs\gl_2)$ if $k$ is regarded as a complex number rather than a formal parameter, so that $N^k(\gs\gl_2)$ always denotes the specialization of the one-parameter algebra at $k\in \mathbb{C}$. For all $k \neq 0$, the simple quotient $N_k(\gs\gl_2)$ coincides with $\text{Com}(\cH, L_k(\gs\gl_2))$.

\begin{thm} \label{parawslncoincidences} For all $n\geq 3$, aside from the critical levels $k = -2$ and $k' = -n$, and the cases $c = 0, -2$, all isomorphisms $N_k(\gs\gl_2) \cong \cW_{k'}(\gs\gl_n, f_{\rm{prin}})$, appear in the following three families.
\begin{enumerate}
\item  $\displaystyle k = n, \qquad k' = -n + \frac{2 + n}{1 + n},\   -n + \frac{1+ n}{2 + n}$.

This family is $C_2$-cofinite and rational, and is given by Theorem \ref{thm:aly}.

\item $\displaystyle n\neq 4,\qquad k = -2 + \frac{n-2}{n-1}, \qquad k' = -n + \frac{n-2}{n-1},\  -n + \frac{n-1}{n-2}$. 

Note that the level $k$ is admissible for $\widehat{\gs\gl_2}$ if $n\geq 5$. 

\item $\displaystyle  k = -2 + \frac{2}{1 + n}, \qquad k' = -n + \frac{n-1}{n+1},\    -n + \frac{n+1}{n-1}$. 

Note that if $n$ is even, $k$ is boundary admissible for $\widehat{\gs\gl_2}$.

\end{enumerate}
\end{thm}

\begin{proof} First, we exclude the values $k = 2, \ -\frac{4}{3}$, since it is apparent from \eqref{ratpara:para} that $N^k(\gs\gl_2)$ is not obtained as a quotient of $\cW^{I}(c,\lambda)$ at these points. By Corollary \ref{cor:intersections}, aside from the cases $c = 0,-2$, all other isomorphisms $N_k(\gs\gl_2) \cong \cW_{k'}(\gs\gl_n, f_{\text{prin}})$ correspond to intersection points on the curves $V(I)$ and $V(I_n)$, where $I_n$ is given by \eqref{ideal:wslna}. A calculation shows that there are exactly three such points $(c, \lambda)$, namely,
\begin{equation} \begin{split} 
&  \bigg(\frac{2 (n-1)}{n+ 2},  \ \frac{n+ 1}{(n-2) (3n+ 4)}\bigg), \qquad \bigg( -\frac{2 (2n -1)}{n-2}, \   \frac{n-1}{(n-4) (3n-2)} \bigg),
\\ & \bigg(-(3n+1), \  -\frac{(n-1) (n+1)}{4 (n-2) (2n+1)}\bigg) .\end{split} \end{equation}
By replacing $c$ with the levels $k,k'$, it is immediate that the above isomorphisms all hold.

To show that our list is complete, we need to show that no additional coincidences occur at the excluded points $k = 2$ and $k = -\frac{4}{3}$. First, $N_2(\gs\gl_2)$ is isomorphic to the Virasoro algebra with $c = \frac{1}{2}$ \cite{DLY}, so this case really belongs in the first family with $n = 2$, but is excluded because we are restricting to $n\geq 3$. Second, $N_{-4/3}(\gs\gl_2)$ has central charge $c = -7$ and is isomorphic to the $M(3)$ singlet algebra which is of type $\cW(2,5)$ \cite{AII}. The only case where $\cW_{k'}(\gs\gl_n, f_{\text{prin}})$ has a singular vector in weight $3$ at this central charge is $n = 4$ and $k' = - \frac{10}{3},\ -\frac{5}{2}$, but in this case $\cW_{k'}(\gs\gl_4, f_{\text{prin}})$ is just the Virasoro algebra. \end{proof}

\begin{remark} For the second and third families of coincidences in Theorem \ref{parawslncoincidences}, the decomposition of $L_k(\gs\gl_2)$ as a module over $\cH\otimes N_k(\gs\gl_2)$, is multiplicity-free. This follows from the fact that the simple $\cH$-modules appearing in $L_k(\gs\gl_2)$ that do not lie in $N_k(\gs\gl_2)$ have lowest-weight vectors $:X^n:$ or $:Y^n:$ for $n\geq 1$, and these modules have conformal weight $ \frac{n^2}{4k}$ with respect to the conformal vector $ \frac{1}{4k} :HH:$ in $\cH$. Here $X,Y,H$ denote in the standard generators of $L_k(\gs\gl_2)$. Hence if $k <0$, as explained in Section 3.2 of \cite{ACKL}, it follows that $N_k(\gs\gl_2)$ has infinitely many simple modules, so it cannot be $C_2$-cofinite or rational.
\end{remark}

\subsection{Coincidences between the Bershadsky-Polyakov coset $\cC_{\ell}$ and $\cW_{\ell'}(\gs\gl_n, f_{\rm{prin}})$}
Recall that as one-parameter vertex algebras, $\cC^{\ell} = \text{Com}(\cH, \cW^{\ell - 3/2}(\gs\gl_3, f_{\text{min}}))$ is obtained as the simple quotient of $\cW^{I}(c,\lambda)$ via the parametrization \eqref{ratparam:cell}, where $I$ is generated by \eqref{curve:bp}. As above, the specialization $\cC^{\ell_0}$ of the one-parameter algebra $\cC^{\ell}$ at $\ell = \ell_0$, coincides with the coset $\text{Com}(\cH, \cW^{\ell_0 - 3/2}(\gs\gl_3, f_{\text{min}}))$ for all $\ell_0\neq 0$. As above, we use the same notation $\cC^{\ell}$ if $\ell$ is regarded as a complex number rather than a formal parameter, so that $\cC^{\ell}$ always denotes the specialization of the one-parameter algebra at $\ell \in \mathbb{C}$. For all $\ell \neq 0$, the simple quotient $\cC_{\ell}$ coincides with $\text{Com}(\cH, \cW_{\ell - 3/2}(\gs\gl_3, f_{\text{min}}))$.

\begin{thm} For all $n\geq 3$, aside from the critical levels $\ell = -\frac{3}{2}$ and $\ell' = -n$, and the cases $c = 0, -2$, all isomorphisms $\cC_{\ell} \cong \cW_{\ell'}(\gs\gl_n, f_{\rm{prin}})$, appear in the following three families.
\begin{enumerate}

\item $\displaystyle  \ell = \frac{n}{2},\qquad \ell' = -n + \frac{3 + n}{1 + n},\   -n + \frac{1 + n}{3 + n}$.

For $\ell = 1,2,3,\dots$, this family is $C_2$-cofinite and rational, and is given by Theorem \ref{thm:acl}.

\item $\displaystyle n \geq 4, \qquad n\neq 6, \qquad \ell = -\frac{n}{2 (n-2)},\qquad \ell' = -n + \frac{n-3}{n-2} ,\  -n + \frac{n-2}{n-3}$.

Note that the level $\displaystyle k = \ell - \frac{3}{2} = -3 + \frac{n-3}{n-2}$ is admissible for $\widehat{\gs\gl_3}$, for all $n > 6$. 

\item $\displaystyle  \ell = -\frac{3 n}{2 (n+ 2)},\qquad  \ell' =-n + \frac{n-1}{n+2}, \  -n + \frac{n+2}{n-1}$.

The level $\displaystyle k = \ell - 3/2 = -3 + \frac{3}{2 + n}$ is boundary admissible for $\widehat{\gs\gl_3}$ if $n \equiv 0\ \text{mod}\ 3$, or $n\equiv 2 \ \text{mod}\ 3$.

\end{enumerate}
\end{thm}

\begin{proof} First, we exclude the values $ \ell = 1, \ -\frac{3}{4}$, since it is apparent from \eqref{ratparam:cell} that $\cC^{\ell}$ is not obtained as a quotient of $\cW^{I}(c,\lambda)$ at these points. By Corollary \ref{cor:intersections}, aside from the cases $c = 0,-2$, all other isomorphisms $\cC_{\ell} \cong \cW_{\ell'}(\gs\gl_n, f_{\text{prin}})$ correspond to intersection points on the curves $V(I)$ and $V(I_n)$, where $I_n$ is given by \eqref{ideal:wslna}. There are exactly three such points $(c, \lambda)$, namely
\begin{equation} \begin{split} &  \bigg(-\frac{3 (n-1)^2}{3 + n}, \ \frac{(1 + n) (3 + n)}{4 (n-2) (3 + 2 n)}\bigg),\qquad \bigg( -\frac{6 (n-1)^2}{(n-3) (n-2)},\  \frac{n-3}{(n-6) (3n -4)} \bigg),
\\ & \bigg( -\frac{2 (1 + 2 n)^2}{2 + n},\    -\frac{n-1}{(n-2) (5n + 4)} \bigg).\end{split} \end{equation}
Replacing $c$ with $\ell, \ell'$, we see that the above isomorphisms all hold, and that our list is complete except for possible coincidences at the excluded points $\ell= 1,\  -\frac{3}{4}$.

First, $\cC_1$ is isomorphic to the Virasoro algebra with $c = -\frac{3}{5}$ \cite{Kaw}, but is excluded from the first family because we are restricting to $n\geq 3$. Second, $\cC_{-3/4}$ has central charge $c = -\frac{25}{2}$ and is isomorphic to the $M(4)$ singlet algebra which is of type $\cW(2,7)$ \cite{CRW}. The only case where $\cW^{\ell'}(\gs\gl_n, f_{\text{prin}})$ has a singular vector in weight $3$ at this central charge is $n = 6$ and $ \ell' =-\frac{21}{4}, \  -\frac{14}{3}$, but in this case $\cW_{\ell'}(\gs\gl_6, f_{\text{prin}})$ is just the Virasoro algebra. \end{proof}

\begin{remark} The highest-weight modules for $\cH$ inside $\cW_{\ell} = \cW_{\ell -3/2},(\gs\gl_3, f_{\rm{min}})$ that do not lie in $\cC_{\ell}$ have conformal weights $\displaystyle \frac{3n^2}{4\ell}$ for $n\geq 1$. Hence if $\ell < 0$, these modules all have negative conformal weights. It follows that the second and third families above, for $\displaystyle \ell = -\frac{n}{2 (n-2)}$ and $\displaystyle \ell = -\frac{3 n}{2 (n+ 2)}$, respectively, are not $C_2$-cofinite or rational.
\end{remark}

\subsection{Coincidences between $\text{Com}(\cH, \cW_k(\gs\gl_4, f_{\rm{subreg}}))$ and $\cW_{k'}(\gs\gl_n, f_{\rm{prin}})$.} 
Recall that as one-parameter vertex algebras, $\cC^k = \text{Com}(\cH, \cW^k(\gs\gl_4, f_{\text{subreg}}))$ is obtained as the simple quotient of $\cW^{I}(c,\lambda)$ via the parametrization \eqref{subregn=4}, where $I$ is generated by \eqref{curve:subregn=4}. Recall that the specialization $\cC^{k_0}$ of the one-parameter algebra $\cC^{k}$ at $k = k_0$, coincides with the coset $\text{Com}(\cH, \cW^{k_0}(\gs\gl_4, f_{\text{subreg}}))$ for all $k_0 \neq -\frac{8}{3}$, since this is the level where the Heisenberg algebra degenerates \cite{FS}. We use the same notation $\cC^k$ if $k$ is regarded as a complex number rather than a formal parameter, so that $\cC^{k}$ always denotes the specialization of the one-parameter algebra at $k \in \mathbb{C}$. For all $k \neq -\frac{8}{3}$, the simple quotient $\cC_{k}$ coincides with $\text{Com}(\cH, \cW_k(\gs\gl_4, f_{\text{subreg}}))$.

\begin{thm} For $n\geq 3$, aside from the critical levels $k = -4$ and $k' = -n$, and the cases $c = 0, -2$, all isomorphisms $\cC_k \cong \cW_{k'}(\gs\gl_n, f_{\rm{prin}})$, appear in the following three families.
\begin{enumerate}
\item $\displaystyle k = -4 + \frac{4 + n}{3},\qquad k' = -n + \frac{4 + n}{1 + n}, \   -n + \frac{1 + n}{4 + n}$.

A subset of this family occurs in Theorem \ref{thm:cl}.

\item $\displaystyle n\geq 5,\qquad n \neq 8,\qquad k =  -4 + \frac{n-4}{n-3},\qquad k' = -n +  \frac{n-4}{n-3},\  -n +  \frac{n-3}{n-4}$.

Note that $k$ is admissible for $\widehat{\gs\gl_4}$ if $n > 8$.

\item $\displaystyle k = -4 + \frac{4}{3 + n}, \qquad k' = -n +  \frac{n-1}{n+3},\   -n +  \frac{n+3}{n-1}$.

Note that $k$ is boundary admissible for $\widehat{\gs\gl_4}$ if $n$ is even.

\end{enumerate}
\end{thm}

\begin{proof} We first exclude the values $ k=  -2, \ -\frac{16}{5}$, since we see from the parametrization \eqref{subregn=4} that at these points, $\cC^{k}$ is not obtained as a quotient of $\cW^{I}(c,\lambda)$. By Corollary \ref{cor:intersections}, aside from the cases $c = 0,-2$, all other isomorphisms $\cC_{k} \cong \cW_{k'}(\gs\gl_n, f_{\text{prin}})$ correspond to intersection points on the curves $V(I)$ and $V(I_n)$, where $I_n$ is given by \eqref{ideal:wslna}. There are exactly three such points $(c, \lambda)$, namely
\begin{equation} \begin{split}
& \bigg( -\frac{4 (n-1) (2n -1)}{4 + n}  , \  -\frac{(1 + n) (4 + n)}{(n-2)^2 (8 + 5 n)}    \bigg),
\quad \bigg( -\frac{4 (n-1) (2n -3)}{(n-4) (n-3)}, \  \frac{(n-4) (n-3)}{3 (n-8) (n-2)^2}   \bigg),
\\ & \bigg(  -\frac{(1 + 3 n) (3 + 5 n)}{3 + n} ,\  -\frac{(n-1) (n+ 3)}{12 (n-2) (n+ 1)^2}  \bigg) .
\end{split}\end{equation}
It follows that the above isomorphisms all hold, and that our list is complete except for possible coincidences at the excluded points $k=-2$ and $k= -\frac{16}{5}$. 

First, $\cC_{-2}$ has central charge $c = -2$, so this point is already excluded from our list. Second, $\cC_{-16/5}$ has central charge $c = -\frac{91}{5}$ and is isomorphic to the $M(5)$ singlet algebra which is of type $\cW(2,9)$ \cite{CRW}. The only case where $\cW^{k'}(\gs\gl_n, f_{\text{prin}})$ has a singular vector in weight $3$ at this central charge is $n = 8$ and $ k' = -\frac{36}{5},\ -\frac{27}{4}$, but in this case $\cW_{k'}(\gs\gl_8, f_{\text{prin}})$ is just the Virasoro algebra. \end{proof}

\begin{remark} The highest-weight modules for $\cH$ inside $\cW_k(\gs\gl_4, f_{\rm{subreg}})$ that do not lie in $\cC_k$ have conformal weights $\displaystyle \frac{2n^2}{8 + 3 k}$ for $n\geq 1$. Hence if $ k < -\frac{8}{3}$, these modules all have negative conformal weights. It follows that the second and third families above, for $\displaystyle k =  -4 + \frac{n-4}{n-3}$ and $\displaystyle k =  -4 + \frac{4}{3 + n}$, respectively, are not $C_2$-cofinite or rational.
\end{remark}

\subsection{Coincidences between generalized parafermion algebras and $\cW_{k'}(\gs\gl_n, f_{\rm{prin}})$.}
For $m \geq 2$, recall that as a one-parameter vertex algebra, the generalized parafermion algebra $\cC^k(m) = \text{Com}(V^k(\gg\gl_m), V^k(\gs\gl_{m+1}))$,  is obtained as the simple quotient of $\cW^{K_m}(c,\lambda)$ via the parametrization \eqref{lambdagenpar}. Recall that the specialization $\cC^{k_0}(m)$ of the one-parameter algebra $\cC^{k}(m)$ at $k = k_0$, can be a proper subalgebra of $\text{Com}(V^{k_0}(\gg\gl_m), V^{k_0}(\gs\gl_{m+1}))$, but this can only occur at $k_0 = 0$, or for rational numbers $k_0 \leq -m$. We use the same notation $\cC^k(m)$ if $k$ is regarded as a complex number rather than a formal parameter, so that $\cC^k(m)$ always denotes the specialization of the one-parameter algebra at $k\in \mathbb{C}$ even if it is a proper subalgebra of the coset. For all $k\in \mathbb{C}$, we denote by $\cC_{k}(m)$ the simple quotient of $\cC^k(m)$.

\begin{thm} \label{thm:gparacoincid} For all $n\geq 3$ and $m \geq 2$, aside from the critical levels $k = -m$ and $k = -m-1$ for $\gs\gl_m$ and $\gs\gl_{m+1}$, the critical level $k' = -n$ for $\gs\gl_n$, and the cases $c = 0, -2$, all isomorphisms $\cC_k(m) \cong \cW_{k'}(\gs\gl_n, f_{\rm{prin}})$, appear in the following three families.
\begin{enumerate}
\item $ \displaystyle k = n, \qquad k' = -n + \frac{m + n}{1 + m + n},\   -n  + \frac{1+m + n}{m + n}$.

This family is $C_2$-cofinite and rational, and appears in Theorem 13.1 of \cite{ACLII}.

\item  $\displaystyle n \neq m+1,\qquad k = -m + \frac{m}{1 - n} ,\qquad k' = -n + \frac{n-m-1}{n-1},\  -n + \frac{n-1}{n -m -1}$.

\item $\displaystyle n\neq m, \qquad k = -m + \frac{m - n}{1 + n}, \qquad k' = -n + \frac{n-m}{1 + n},\   -n + \frac{1 + n}{n-m}$.

\end{enumerate}
\end{thm}

\begin{proof} 
We first exclude the values $ k =  2, \ -2m,\  -\frac{1}{3}(2m+2)$, since it follows from the parametrization \eqref{lambdagenpar} that at these points, $\cC^k(m)$ is not obtained as a quotient of $\cW^{K_m}(c,\lambda)$. By Corollary \ref{cor:intersections}, aside from the cases $c = 0,-2$, all remaining isomorphisms $\cC_k(m)\cong \cW_{k'}(\gs\gl_n, f_{\text{prin}})$ correspond to intersection points on the curves $V(K_m)$ and $V(I_n)$, where $K_m$ and $I_n$ are given by \eqref{gpara:ideal} and \eqref{ideal:wslna}, respectively. For each $n \geq 3$ and $m \geq 2$, there are exactly three intersection points $(c, \lambda)$, namely,
\begin{equation} \begin{split}
& \bigg( \frac{m (n-1) (1 + m + 2 n)}{(m + n) (1 + m + n)},\   \frac{(m + n) (1 + m + n)}{(n-2) (2 m + n) (2 + 2 m + 3 n)}      \bigg),
\\ & \bigg(  \frac{(1 + m - n + m n) (n + m n -1)}{1 + m - n)}, \ -\frac{(n-1) (n-1 - m)}{(n-2) (2 + 2 m - 2 n + m n) (2 n + m n -2)}  \bigg),
\\ & \bigg(  \frac{m (n-1) (1 + 2 n + m n)}{m - n},\  -\frac{(1 + n) (n-m)}{(n-2) (2 m - n + m n) (2 + 3 n + m n)}  \bigg).
\end{split}
\end{equation}
 It is immediate that the above isomorphisms all hold, and that our list is complete except for possible coincidences at the excluded points $k =  2, \ -2m,\  -\frac{1}{3}(2m+2)$.
 
At $k = 2$, $\cC_k(m)$ has central charge $\displaystyle c = \frac{m (5 + m)}{(2 + m) (3 + m)}$ and the weight $3$ field is singular. Since $\cW_{k'}(\gs\gl_n, f_{\text{prin}})$ has a singular vector in weight $3$ only for $c = 0$ and $\displaystyle c = -\frac{(n-1) (3n+2)}{n+2}$, it is apparent that there are no integers $n\geq 3$ for which $\cW_{k'}(\gs\gl_n, f_{\text{prin}})$ has a singular vector in weight $3$ at the above central charge. Hence there is no coincidence at this point. Similarly, at $k = -2m$, $\cC_k(m)$ has central charge $\displaystyle c = \frac{(2m+1) (3m-1)}{m-1}$, and again there are no integers $n \geq 3$ for which $\cW_{k'}(\gs\gl_n, f_{\text{prin}})$ has a singular vector in weight $3$ at this central charge. Finally, for $k =  -\frac{1}{3}(2m+2)$, $\cC_k(m)$ has central charge $\displaystyle c = \frac{m (2m+5)}{m-2}$, and by the same argument there is no coincidence at this point. 
\end{proof}

\begin{remark} In the first family of coincidences in Theorem \ref{thm:gparacoincid} we have $k > -m$, and in the third family, we have $k> -m$ and $k\neq 0$ whenever $m>n$. Therefore in these cases, $\cC^k(m)$ coincides with $\text{Com}(V^k(\gg\gl_m), V^k(\gs\gl_{m+1}))$, and $\cC_k(m)$ is its simple quotient. However, in the second family, and in the third family for $m < n$, we have $k < -m$, so it is possible that $\cC^k(m)$ is a proper subalgebra of $\text{Com}(V^k(\gg\gl_m), V^k(\gs\gl_{m+1}))$.
\end{remark}

\subsection{Coincidences between minimal $\cW$-algebra cosets and $\cW_{k'}(\gs\gl_m, f_{\rm{prin}})$.}
For $n\geq 4$, recall that as a one-parameter family, $\cC^k(n) = \text{Com}(V^{k+1}(\gg\gl_{n-2}), \cW^k(\gs\gl_n, f_{\text{min}}))$ is conjecturally obtained as the simple quotient of $\cW^{K_n}(c,\lambda)$ via the parametrization \eqref{eq:paraminw}. Recall that the specialization $\cC^{k_0}(n)$ of the one-parameter algebra $\cC^{k}(n)$ at $k = k_0$, can be a proper subalgebra of $\text{Com}(V^{k_0+1}(\gg\gl_{n-2}), \cW^{k_0}(\gs\gl_n, f_{\text{min}}))$, but this can only occur at $k_0 = 0$, or for rational numbers $k_0 \leq -n+2$. We use the same notation $\cC^k(n)$ if $k$ is regarded as a complex number rather than a formal parameter, so that $\cC^k(n)$ always denotes the specialization of the one-parameter algebra at $k\in \mathbb{C}$. For all $k\in \mathbb{C}$, we denote by $\cC_{k}(n)$ the simple quotient of $\cC^k(n)$.

\begin{conj} \label{conj:minwcoincidences} For $n\geq 4$ and $m\geq 3$, aside from cases $k = -n, \ -n+1$, $k' = -m$, and the cases $c = 0,-2$, all isomorphisms 
$\cC_k(n) \cong \cW_{k'}(\gs\gl_m, f_{\rm{prin}})$, appear in the following three families.

\begin{enumerate}

\item $\displaystyle k = \frac{m-n}{2}, \qquad k' = -m + \frac{m + n}{m + n-2}, \  -m + \frac{m + n-2}{m + n}$.

\item $\displaystyle  n \neq m+2, \qquad k = -\frac{(1 + m) n}{2 + m},\qquad k' = -m + \frac{2 + m}{2 + m - n}, \   -m + \frac{2 + m - n}{2 + m}$.

\item $\displaystyle n\neq m, \qquad k =  -\frac{m n -m - n}{m-2},   \qquad k' = -m + \frac{m-2}{m - n},\  -m + \frac{m - n}{m-2}$.
\end{enumerate}
\end{conj}

We will show that Conjecture \ref{conj:minwcoincidences} follows from Conjecture \ref{conj:minwcoset}. First, we exclude the values $k = -\frac{1}{2}(n-2),\  -\frac{3n}{4}$, since at these points, $\cC^k(n)$ is not obtained as a quotient of $\cW^{K_n}(c,\lambda)$. Assuming Conjecture \ref{conj:minwcoset}, and applying Corollary \ref{cor:intersections}, aside from the cases $c = 0,-2$, all remaining isomorphisms $\cC_k(n) \cong \cW_{k'}(\gs\gl_m, f_{\text{prin}})$ correspond to the intersection points on the curves $V(K_n)$ and $V(I_m)$, where $K_n$ and $I_m$ are given by \eqref{ideal:minw} and \eqref{ideal:wslna}. There are exactly three intersection points $(c,\lambda)$, namely
\begin{equation}
\begin{split} 
& \bigg(   -\frac{(m-1) (2 + m - n) (3 m + n)}{(m + n -2 ) (m + n)},   \  \frac{(m + n -2) (m + n)}{4 (m-2) (n-2) (2 m + n)}  \bigg),
\\ & \bigg(-\frac{(m-1) (2 + m + m n) ( n + m n -m -2)}{(2 + m) (2 + m - n)},  \  -\frac{2 + m - n}{(m-2) (n-2) (4 + 2 m + m n)} \bigg),
\\ & \bigg(  -\frac{(m-1) ( m n -2 - m) (n -3 m + m n)}{(m-2) (m - n)} , \   -\frac{m - n}{(m n -4 ) (2 n -4 m  + m n)}    \bigg).
\end{split}
\end{equation}
 It is immediate that the above isomorphisms all hold, and that our list is complete except for possible coincidences at the excluded points $k = -\frac{1}{2}(n-2),\  -\frac{3n}{4}$.
 
At $k = -\frac{1}{2}(n-2)$, $\cC_k(n)$ has central charge $\displaystyle c = \frac{(n-4) (n+6)}{n (n+2)}$, and the weight $3$ field is singular. Since $\cW_{k'}(\gs\gl_m, f_{\text{prin}})$ has a singular vector in weight $3$ only for $c = 0$ and $\displaystyle c = -\frac{(m-1) (3m+2)}{m+2}$, there are no integers $m\geq 3$ for which $\cW_{k'}(\gs\gl_m, f_{\text{prin}})$ has a singular vector in weight $3$ at the above central charge. Similarly, at $k =-\frac{3n}{4}$, $\cC_k(n)$ has central charge $\displaystyle c = \frac{(n+2) (3n-4)}{2 (n-4)}$, and again there are no integers $m \geq 3$ for which $\cW_{k'}(\gs\gl_m, f_{\text{prin}})$ has a singular vector in weight $3$ at this central charge. This shows that there are no additional coincidences at the excluded points, which completes the proof that Conjecture \ref{conj:minwcoset} implies Conjecture \ref{conj:minwcoincidences}.

\begin{remark} The first family of coincidences in Conjecture \ref{conj:minwcoincidences}, in the case $m = 2k+n$ and $k \in \mathbb{N}$, is the one appearing in Kawasetsu's conjecture. Since Conjecture \ref{conj:minwcoset} holds for $n=4$, Conjecture \ref{conj:minwcoincidences} is a theorem in the case $n = 4$, for all $m$. \end{remark}

\subsection{Coincidences between subregular $\cW$-algebra cosets and $\cW_{k'}(\gs\gl_m, f_{\rm{prin}})$} 
Recall that as one-parameter vertex algebras, $\cC^k(n) = \text{Com}(\cH, \cW^k(\gs\gl_n, f_{\text{subreg}}))$ is conjecturally obtained as the simple quotient of $\cW^{K_n}(c,\lambda)$, where the ideal $K_n$ is determined by the parametrization \eqref{subreg:ideal}. Recall that the specialization $\cC^{k_0}(n)$ of the one-parameter algebra $\cC^{k}(n)$ at $k = k_0$, coincides with the coset $\text{Com}(\cH, \cW^k(\gs\gl_n, f_{\text{subreg}}))$ for all $\displaystyle k \neq -\frac{n(n-2)}{n-1}$, since this is the level where the Heisenberg algebra degenerates \cite{FS}. As usual, we use the same notation $\cC^k(n)$ if $k$ is regarded as a complex number rather than a formal parameter. For all $\displaystyle k \neq -\frac{n(n-2)}{n-1}$, the simple quotient $\cC_{k}(n)$ coincides with $\text{Com}(\cH, \cW_k(\gs\gl_n, f_{\text{subreg}}))$.

We exclude the values $\displaystyle k = -n, \ -n + \frac{n}{n+1},\ -n + \frac{n-2}{n-3},\ -n + \frac{n+2}{n-1}$, since at these points, $\cC^k(n)$ is not obtained as a quotient of $\cW^{K_n}(c,\lambda)$. Assuming Conjecture \ref{conj:subregcoset}, we can apply Corollary \ref{cor:intersections} to classify the coincidences among $\cC_k(n)$ and $\cW_{k'}(\gs\gl_m, f_{\text{prin}})$ for $n\geq 4$ and $m\geq 3$. There are exactly three intersection points $(c,\lambda)$ on the curves $V(K_n)$ and $V(I_m)$, where $K_n$ and $I_m$ are given by \eqref{subreg:ideal} and \eqref{ideal:wslna}, respectively:
\begin{equation}
\begin{split} 
& \bigg(   -\frac{n (m -1) (m n -1 - 2 m )}{m + n} , \   -\frac{(1 + m) (m + n)}{(m-2) (m n -2 - 3 m) (m + 2 n + m n)}  \bigg),
\\ &\bigg(  -\frac{(1 - m + m n) (m + n + m n -1 )}{m + n -1},   \ -\frac{(m-1) (m + n -1)}{(m-2) (2 - 2 m + m n) (2 m + 2 n + m n -2 )} \bigg),
\\ & \bigg( -\frac{n(m -1) (1 + 2 m - n)}{(m - n) (1 + m - n)} , \      \frac{(m - n) (1 + m - n)}{(m-2) (m - 2 n) (2 + 3 m - 2 n)}    \bigg).
\end{split}
\end{equation}
We obtain

\begin{conj} \label{conj:subregwcoincidences} For $n\geq 4$ and $m\geq 3$, aside from the points excluded above, all isomorphisms $\cC_k(n)\cong \cW_{k'}(\gs\gl_m, f_{\rm{prin}})$ for $c \neq 0,-2$, appear in the following three families.

\begin{enumerate}
\item $ \displaystyle k = -n + \frac{m + n}{n-1}, \qquad k' = -m + \frac{m +1}{m + n},\  -m + \frac{m + n}{m+1}$,

which is the family appearing in Conjecture \ref{conj:subregphysics},
\item $\displaystyle k = -n + \frac{n}{m + n -1}, \qquad k' = -m + \frac{m-1}{m + n -1},\  -m + \frac{m + n -1}{m-1}$,
\item $\displaystyle n\neq m, \ m+1,\qquad k = \frac{n-m}{ n -m -1},\qquad k' = -m + \frac{m - n}{1 + m - n},\  -m + \frac{1 + m - n}{m - n}$.
\end{enumerate}
\end{conj}

We briefly consider the excluded points $\displaystyle k = -n + \frac{n}{n+1},\ -n + \frac{n-2}{n-3},\ -n + \frac{n+2}{n-1}$. First, for $\displaystyle k =  -n + \frac{n}{n+1}$, $\cC_k(n)$ has central charge $\displaystyle c =-\frac{(2n-1) (3n+1)}{n+1}$, and is isomorphic to the $M(n+1)$ singlet algebra, which is of type $\cW(2,2n+1)$. This was conjectured in \cite{CRW} and proved recently in \cite{ACGY}. The only case where $\cW_{k'}(\gs\gl_m, f_{\text{prin}})$ for $m\geq 3$ has a singular vector in weight $3$ at this central charge is $m = 2n$ and $\displaystyle k' = -2n + \frac{2n+2}{2n},\ -2n + \frac{2n}{2n+2}$, but $\cW_{k'}(\gs\gl_m, f_{\text{prin}})$ is just the Virasoro algebra. Therefore we do not have a coincidence at this point. Second, for $\displaystyle k =  -n + \frac{n-2}{n-3}$, $\cC_k(n)$ has central charge $\displaystyle c = \frac{n(n-5)}{(n-3) (n-2)}$. There are no integers $m \geq 3$ for which $\cW_{k'}(\gs\gl_m, f_{\text{prin}})$ has this central charge, so there are no coincidences at this point. Finally, for $\displaystyle k = -n + \frac{n+2}{n-1}$, $\cC_k(n)$ has central charge is $\displaystyle c = -\frac{n (2n-5)}{n+2}$. If $n \geq 7$ and $n\equiv 1\  \text{mod}\ 3$, letting $\displaystyle m = \frac{2(n-1)}{3}$ and $\displaystyle k' = -\frac{2(n-1)}{3} +\frac{n-1}{n+2},\  -\frac{2(n-1)}{3} +\frac{n+2}{n-1}$, then $\cW_{k'}(\gs\gl_m, f_{\text{prin}})$ is just the Virasoro algebra with this central charge. It is possible that $\cC_k(n) \cong \cW_{k'}(\gs\gl_m, f_{\text{prin}})$, but we do not know how to determine this.

So far, we have only considered coincidences between $\cW_k(\gs\gl_n, f_{\text{prin}})$ and other families of vertex algebras of type $\cW(2,3,\dots, N)$. Using Corollary \ref{cor:intersections}, we can also classify the coincidences among the families of quotients of $\cW(c,\lambda)$ given by \eqref{curve:para}, \eqref{curve:bp}, \eqref{curve:subregn=4}, \eqref{gpara:ideal} (and conjecturally by \eqref{ideal:minw} and \eqref{subreg:ideal}), by finding the pairwise intersection points of the curves. It is straightforward to verify that {\it all} intersection points between any two curves in any of these families, are rational points. These curves also have rational singular points. We speculate that the pairwise intersections of {\it all} truncation curves of $\cW(c,\lambda)$, as well as the singular points on these curves, are rational.

\section{Deforming the $\cW_{1+\infty}$-algebra}\label{section:oneplus}
The $\cW_{1+\infty,c}$-algebra with central charge $c$ is a module over the centrally extended Lie algebra of regular differential operators on the circle. It has been studied extensively in both the physics and mathematics literature \cite{AI,ASM,AFMO,FKRW,KP,KRI,KRII,Wa}. In this section, we show that $\cW_{1+\infty,c}$ admits a one-parameter deformation. 

Let $\cD$ be the Lie algebra of regular differential operators on the circle, with coordinate $t$. It has a basis $$J^l_k = -t^{l+k} (\partial_t)^l, \qquad  k\in \mathbb{Z},\qquad l\in \mathbb{Z}_{\geq 0},$$ where $\partial_t = \frac{d}{dt}$. There is a $2$-cocycle on $\cD$ given by \begin{equation}\label{cocycle} \Psi\big(f(t) (\partial_t)^m, g(t) (\partial_t)^n\big) = \frac{m! n!}{(m+n+1)!} \text{Res}_{t=0} f^{(n+1)}(t) g^{(m)}(t) dt,\end{equation} and a corresponding central extension $\hat{\cD} = \cD \oplus \mathbb{C} \kappa$ \cite{KP}. We have a $\mathbb{Z}$-grading by weight 
$$\hat{\cD} = \bigoplus_{j\in\mathbb{Z}} \hat{\cD}_j,\qquad \text{wt} (J^l_k) = k,\qquad \text{wt} (\kappa) = 0,$$ and a triangular decomposition $$\hat{\cD} = \hat{\cD}_+\oplus\hat{\cD}_0\oplus \hat{\cD}_-,\qquad \hat{\cD}_{\pm} = \bigoplus_{j\in \pm \mathbb{N}} \hat{\cD}_j,\qquad \hat{\cD}_0 = \cD_0\oplus \mathbb{C}\kappa.$$  For $c\in\mathbb{C}$ and $\lambda\in \cD_0^*$, define the Verma module 
$$\cM_c(\hat{\cD},\lambda) = U(\hat{\cD})\otimes_{U(\hat{\cD}_0\oplus \hat{\cD}_+)} \C_{\lambda},$$ where $\C_{\lambda}$ is the one-dimensional $\hat{\cD}_0\oplus \hat{\cD}_+$-module on which $\kappa$ acts by multiplication by $c$ and $h\in\hat{\cD}_0$ acts by multiplication by $\lambda(h)$, and $\hat{\cD}_+$ acts by zero. Let $\cP \subseteq \cD$ be the Lie subalgebra of differential operators which extend to all of $\mathbb{C}$, which has a basis $\{J^l_k|\ l\geq 0,\ l+k\geq 0\}$. Since $\Psi$ vanishes on $\cP$, $\cP$ may be regarded as a subalgebra of $\hat{\cD}$, and $\hat{\cD}_0\oplus \hat{\cD}_+\subseteq \hat{\cP}$, where $\hat{\cP} = \cP\oplus \mathbb{C}\kappa$. The induced $\hat{\cD}$-module $$\cM_c = U(\hat{\cD})\otimes_{U(\hat{\cP})} \C_0$$ is a quotient of $\cM_c(\hat{\cD},0)$, and is known as the {\it vacuum $\hat{\cD}$-module of central charge $c$}. It has a vertex algebra structure with generators \begin{equation} \label{winfgen} J^l(z) = \sum_{k\in\mathbb{Z}} J^l_k z^{-k-l-1},\qquad l\geq 0,\end{equation} of weight $l+1$. Then $\{J^l_k, \kappa\}$ represent $\hat{\cD}$ on $\cM_c$, and we write \eqref{winfgen} in the form
$$J^l(z) = \sum_{k\in\mathbb{Z}} J^l(k) z^{-k-1}, \qquad J^l(k) = J^l_{k-l}.$$ In fact, $\cM_c$ is freely generated by $\{J^l(z)|\ l\geq 0\}$, and these fields close linearly under OPE. The vertex algebra $\cW_{1+\infty,c}$ is defined to be the quotient of $\cM_c$ by its maximal proper graded ideal. It is simple as a vertex algebra over $\mathbb{C}[c]$. The cocycle \eqref{cocycle} is normalized so that $\cM_c$ has a nontrivial ideal if and only if $c\in \mathbb{Z}$, and $\cM_c \cong \cW_{1+\infty,c}$ for all $c \notin \mathbb{Z}$.

 Let $\cH$ be the Heisenberg algebra with generator $J$ satisfying $J(z) J(w) \sim (z-w)^{-2}$, and define \begin{equation} \cV(c,\lambda) = \cH \otimes \cW(c,\lambda),\end{equation} which is defined over the ring $\mathbb{C}[c,\lambda]$ and is freely generated type $\cW(1,2,3,\dots)$. Note that $\cH$ has Virasoro element $L^{\cH} = \frac{1}{2} :JJ:$ of central charge $1$, so $\cV(c,\lambda)$ has Virasoro element
 $$L^{\cV} = L^{\cH}\ + L,$$ of central charge $c+1$. Given an ideal $I \subseteq \mathbb{C}[c,\lambda]$, we have a vertex algebra ideal $I \cdot \cV(c,\lambda)$. The quotient
 $$\cV^I(c,\lambda) = \cV(c,\lambda) / I \cdot \cV(c,\lambda)$$ is defined over $\mathbb{C}[c,\lambda] / I$ and is freely generated of type $\cW(1,2,3,\dots)$. If $R = D^{-1} \mathbb{C}[c,\lambda] / I$ is a localization along some multiplicative set $D \subseteq \mathbb{C}[c,\lambda] / I$, we have the localization 
 $$\cV^I_R(c,\lambda) = R \otimes_{\mathbb{C}[c,\lambda] / I} \cV^I(c,\lambda).$$

\begin{thm} \label{thm:winfreal} Let $I\subseteq \mathbb{C}[c,\lambda]$ be the ideal generated by 
\begin{equation} \label{ideal:winf} 4\lambda (c-1) - 1,\end{equation}
$D$ the multiplicative set generated by $(c-1)$, and $$R = D^{-1} \mathbb{C}[c,\lambda] / I \cong D^{-1} \mathbb{C}[c]$$ the localization along $D$. Then as vertex algebras over $R$, we have 
$$R \otimes_{\mathbb{C}[c]}\cM_{c+1} \cong \cH \otimes \cW^I_R(c,\lambda) \cong \cV^I_R(c,\lambda).$$ 
In particular, we may regard $\cV_R(c,\lambda)$ as a one-parameter deformation of $\cM_{c+1}$. Here $R$ is regarded as a localization of $\mathbb{C}[c,\lambda]$ instead of $\mathbb{C}[c,\lambda] / I$, and $$\cV_R(c,\lambda) = R \otimes_{\mathbb{C}[c,\lambda]} \cV(c,\lambda).$$  \end{thm}

\begin{proof} Since the zero mode $J_0$ acts trivially on $\cM_{c+1}$, we clearly have $$\cM_{c+1} \cong \cH \otimes \cC_c,\qquad \cC_c = \text{Com}(\cH, \cM_{c+1}).$$ Note that $\cC_c$ has central charge $c$. It is well known to be freely generated of type $\cW(2,3,\dots)$ and is generated by the Virasoro field $L$ and a weight $3$ primary field $W^3$. By Theorem \ref{refinedsimplequotient}, $\cC_c$ can be realized in the form $\cW^I_R(c,\lambda)$ for some $I$ and $R$. It is then easy to find the explicit form of $I$ and $R$ by setting $W^4 = W^3_{(1)} W^3$, and computing the coefficient of $W^3$ in the fourth order pole $W^3_{(3)} W^4$.
\end{proof}

\subsection{Deformations of $\cW_{1+\infty,n}$ when $n$ is a positive integer}

The case $n=1$ is not interesting since $\cW_{1+\infty,1} \cong \cH$, so assume that $n\geq 2$. As shown in \cite{FKRW}, $\cW_{1+\infty,n}$ has a free field realization as the $GL_n$-orbifold of the rank $n$ $bc$-system with odd generators $b^i, c^i$, $i=1,\dots, n$ and OPE relations $$b^i(z) c^j(w) \sim \delta_{i,j} (z-w)^{-1}.$$
Note that $\cE(n)$ has the charge grading $$\cE(n) = \bigoplus_{i \in \mathbb{Z}}\cE^i(n),$$ where $\cE^i(n)$ is the eigenspace of eigenvalue $i$ of the zero mode $J_0$ of $J = -\sum_{i=1}^n :b^i c^i:$. Clearly $J$ generates a Heisenberg algebra $\cH$ and it is well known that $\text{Com}(\cH, \cE(n)) \cong L_{1}(\gs\gl_n)$. Then $\cE^0(n) \cong \cH \otimes L_1(\gs\gl_n)$ and $$\cW_{1+\infty,n} \cong\cE(n)^{GL_n} \cong \cE^0(n)^{GL_n} \cong \cH \otimes L_1(\gs\gl_n)^{SL_n}.$$ Therefore $\cC_{n-1} \cong L_1(\gs\gl_n)^{SL_n}$, which is known to be isomorphic to $\cW^{1-n}(\gs\gl_n, f_{\text{prin}})$ \cite{FKRW}.

Recall the coset
$$\cC^k(\gs\gl_n) = \text{Com}(V^{k+1}(\gs\gl_n), V^{k}(\gs\gl_n) \otimes L_1(\gs\gl_n)).$$ By Theorem 6.10 and Example 7.13 of \cite{CLII}, $\cC^k(\gs\gl_n)$ is a deformation of $L_1(\gs\gl_n)^{SL_n}$ in the sense that $$\lim_{k\ra \infty} \cC^k(\gs\gl_n) \cong L_1(\gs\gl_n)^{SL_n} \cong \cC_{n-1}.$$ Since $\cC^k(\gs\gl_n) \cong \cW_{k'}(\gs\gl_n, f_{\text{prin}})$ by Theorem \ref{ACLIImain}, this corresponds to the statement $$\cC_{n-1} \cong \lim_{k' \ra 1-n} \cW^{k'}(\gs\gl_n, f_{\text{prin}}).$$

\subsection{Deformations of $\cW_{1+\infty,-n}$ when $n\geq 1$}
Next, we consider $\cW_{1+\infty,-n}$ for $n\geq 1$, which is more complicated than the positive integer case. By \cite{KRII}, it has a similar free field realization as the $GL_n$-orbifold of the $\beta\gamma$-system $\cS(n)$, which has even generators $\beta^i, \gamma^i$, $i=1,\dots, n$, satisfying
$$\beta^i(z) \gamma^j(w) \sim \delta_{i,j}(z-w)^{-1}.$$
We have the charge grading $$\cS(n) = \bigoplus_{i\in\mathbb{Z}} \cS^i(n),$$ where $\cS^i(n)$ is the eigenspace of eigenvalue $i$ under $J_0$ where $J = \sum_{i=1}^n :\beta^i \gamma^i:$. The cases $n=1$ and $n=2$ are exceptional: $\cW_{1+\infty, -1} \cong \cS(1)^{GL_1} = \cS(1)^0 \cong \cH \otimes \cW_{3,-2}$, where $\cW_{3,-2}$ is the simple Zamolodchikov algebra with $c = -2$ \cite{Wa}, and we will treat the case $n=2$ separately below. The cases $n\geq 3$ can be studied uniformly because of the isomorphism 
$$\cS^0(n) \cong \cH \otimes L_{-1}(\gs\gl_n),\qquad n\geq 3,$$ which is due to  Adamovi\'c and Per\v{s}e \cite{AP}. We then have 
$$\cW_{1+\infty,-n} \cong \cS(n)^{GL_n} \cong \cS^0(n)^{GL_n} \cong \cH \otimes L_{-1}(\gs\gl_n)^{SL_n},$$ and $\cC_{-n-1} \cong L_{-1}(\gs\gl_n)^{SL_n}$, for all $n\geq 3$. Consider the coset
\begin{equation} \label{coset:winfneg} \cC^k(n) =  \text{Com}(V^k(\gs\gl_n), V^{k+1}(\gs\gl_n) \otimes L_{-1}(\gs \gl_n)),\end{equation} which has central charge 
\begin{equation} \label{cc:winfneg} c= -\frac{(1 + k) (1 + n) (k + 2 n)}{(k + n) (1 + k + n)}.\end{equation} By Theorem 6.10 of \cite{CLII}, we have $\displaystyle \lim_{k\ra \infty} \cC^k(n) \cong L_{-1}(\gs\gl_n)^{SL_n} \cong \cC_{-n-1}$.

\begin{thm} For all $n\geq 3$, there exists an ideal $K_n \subseteq \mathbb{C}[c,\lambda]$ and a localization $R_n$ of $\mathbb{C}[c,\lambda] / K_n$ such that
\begin{enumerate} 
\item The vertex algebra $\cW^{K_n}_{R_n}(c,\lambda)$ has a singular vector in weight  $(n+1)^2$ of the form $$W^{(n+1)^2} - P(L, W^3, \dots, W^{n^2+2n-1}).$$ 
\item Letting $\cK_n$ be the maximal proper graded ideal of $\cW^{K_n}_{R_n}(c,\lambda)$, we have $$\cW^{K_n}_{R_n}(c,\lambda)/ \cK_n \cong \mathcal{C}^k(n),$$ where $k$ and $c$ are related by \eqref{cc:winfneg}.
\end{enumerate}
\end{thm}

\begin{proof}
In terms of its realization in $\cS(n)^{GL_n}$, the Virasoro field $L$ and the weight $3$ primary field $W^3$ of $\cC_{-n-1}$ appear in Equations 4.42 and 4.43 of \cite{LI}. It is easy to check that $L,W^3$ generate $\cC_{-n-1}$, and by Weyl's second fundamental theorem of invariant theory for $GL_n$, there are no normally ordered relations among $L, W^i$ and their derivatives of weight less than $(n+1)^2$. Here $W^i = W^3_{(1)} W^{i-1}$ for $i\geq 4$, as usual. It follows that for generic values of $k$, the coset $\cC^k(n)$ is also generated by the corresponding fields $L,W^3$, and there are no relations among $L, W^i$ and their derivatives of weight less than $(n+1)^2$. By the same argument as Corollary \ref{cor:genpara}, $\cC^k(n)$ is simple, so by Theorem \ref{refinedsimplequotient}, $\cC^k(n)$ can be realized as the simple quotient of $\cW^{K_n}_{R_n}(c,\lambda)$ for some $K_n$ and $R_n$. 

Finally, $\cW_{1+\infty,-n}$ is of type $\cW(1,2,\dots, n^2+2n)$ by Theorem 4.16 of \cite{LI}, which implies that $\cC_{-n-1}$ is of type $\cW(2,3,\dots, n^2+2n)$. By Lemma 3.2 and Theorem 6.10 of \cite{CLII}, $\cC^k(n)$ is also of type $\cW(2,3,\dots, n^2+2n)$, a fact which was originally conjectured in \cite{B-H}. Since the weight $(n+1)^2$ field decouples, the singular vector in $\cW^{K_n}_{R_n}(c,\lambda)$ must occur in weight $(n+1)^2$ and have the desired form.
\end{proof}

In the physics literature \cite{B-H,HII}, the coset \eqref{coset:winfneg} has been regarded as the {\it definition} of the principal $\cW$-algebra of $\gs\gl_{-n}$. This translates to the following description of $K_n$.

\begin{conj} \label{conj:winf} For $n\geq 3$, $K_n$ is generated by
\begin{equation} \label{conj:winfideal} p_n(c,\lambda) = \lambda  (n+2) (3 n^2  + n -2+ c (2-n)) -  (1+n) (1-n) ,\end{equation}
which is obtained from \eqref{ideal:wslna} by replacing $n$ with $-n$. Equivalently, the variety $V(K_n)$ has the rational parametrization $$c= -\frac{(1 + k) (1 + n) (k + 2 n)}{(k + n) (1 + k + n)},\qquad \lambda = -\frac{(k + n) (1 + k + n)}{(2 + n) (2 + 2 k + n) (2 k + 3 n)}.$$
\end{conj}

The case $n=2$ must be treated differently from $n\geq 3$ because $\text{Com}(\cH, \cS(2))\neq L_{-1}(\gs\gl_2)$. Instead, $\text{Com}(\cH, \cS(2))$ is an extension of $L_{-1}(\gs\gl_2)$, and is isomorphic to the simple rectangular $\cW$-algebra $\cW_{-5/2}(\gs\gl_4, f_{\text{rect}})$; see \cite{CKLR}, Proposition 5.4. Here the nilpotent element $f_{\text{rect}}$ corresponds to the embedding $\gs\gl_2 \hookrightarrow \gs\gl_4$ such that $\gs\gl_4$ decomposes as a sum four copies of the adjoint representation of $\gs\gl_2$ and three copies of the trivial representation. Then $\cW^k(\gs\gl_4, f_{\text{rect}})$ is of type $\cW(1, 1, 1, 2, 2, 2, 2)$ and the affine subalgebra is $V^{k+3/2}(\gs\gl_2)$. At level $k = -5/2$, there is a singular vector in weight $2$, and $L_{-1}(\gs\gl_2)$ is conformally embedded in the simple quotient $\cW_{-5/2}(\gs\gl_4, f_{\text{rect}})$, which is of type $\cW(1,1,1,2,2,2)$.

Since $\cS^0(2) \cong \cH \otimes \cW_{-5/2}(\gs\gl_4, f_{\text{rect}})$, we obtain
$$\cW_{1+\infty,-2} \cong \cS(2)^{GL_2} \cong \cS^0(2)^{GL_2} \cong \cH \otimes \cW_{-5/2}(\gs\gl_4, f_{\text{rect}})^{SL_2},$$ so that $\cC_{-3} \cong \cW_{-5/2}(\gs\gl_4, f_{\text{rect}})^{SL_2}$. This is the orbifold limit of the following coset
\begin{equation} \label{coset:winf2} \cC^k(2) = \text{Com}(V^k(\gs\gl_2),  V^{k+1}(\gs\gl_2) \otimes \cW_{-5/2}(\gs\gl_4, f_{\text{rect}})),\end{equation} that is, 
$\lim_{k\ra \infty} \cC^k(2) \cong \cW_{-5/2}(\gs\gl_4, f_{\text{rect}})^{SL_2}$. Since $\cC_{-3}$ is of type $\cW(2,3,4,5,6,7,8)$, so is $\cC^k(2)$.

Note that the formula \eqref{conj:winfideal} makes sense for $n = 2$, and we obtain $\lambda = -1/16$. It can be verified by computer that $(\lambda + 1/16)$ lies in the Shapovalov spectrum of level $9$. Let $I = (\lambda + 1/16)$ and let $R$ be the localization of $\mathbb{C}[c,\lambda] / I \cong \mathbb{C}[c]$ along the multiplicative set generated by $c$. In weight $9$, $\cW^I_R(c,\lambda)$ turns out to have a unique singular vector of the form $W^9 - P(L, W^3, \dots, W^7)$, so the simple quotient of $\cW^I_R(c,\lambda)$ is of type $\cW(2,3,\dots, 8)$. 

\begin{thm} Let $I = (\lambda +1/16)$ and let $R$ be the localization of $\mathbb{C}[c,\lambda] / I \cong \mathbb{C}[c]$ along the multiplicative set generated by $c$. The coset $\cC^k(2)$ given by \eqref{coset:winf2} is isomorphic to the simple quotient of $\cW^I_R(c,\lambda)$, where $c$ and $k$ are related by $\displaystyle c = -\frac{3 (1 + k) (4 + k)}{(2 + k) (3 + k)}$. \end{thm}

\begin{proof} Using the isomorphism $\cS^0(2) \cong \cH \otimes \cW_{-5/2}(\gs\gl_4, f_{\text{rect}})$, we have an embedding $\cC^k(2) \rightarrow \text{Com}(V^k(\gs\gl_2), V^{k+1}(\gs\gl_2) \otimes \cS^0(2))$, which is determined by specifying the weight $3$ primary field $W^3$, normalized so that $\displaystyle W^3_{(5)} W^3 = \frac{c}{3} 1 = -\frac{(1 + k) (4 + k)}{(2 + k) (3 + k)} 1$. Choosing standard generators $X,Y,H$ for $V^{k+1}(\gs\gl_2)$ and generators $\beta_i, \gamma_i$, $i=1,2$ for $\cS(2)$ with $\beta_i(z) \gamma_j(w) \sim \delta_{i,j} (z-w)^{-1}$, we have the following explicit formula.

\begin{equation} \begin{split} W^3 & =\frac{1+k}{12 \sqrt{(2 + k) (3 + k)}}
\bigg( :(\beta_1)^3 (\gamma_1)^3: + :(\beta_2)^3 (\gamma_2)^3: + 
 3 :(\beta_1)^2 \beta_2 (\gamma_1)^2 \gamma_2:  
\\ & + 6  :(\beta_1)^2 (\partial \gamma_1) \gamma_1:    
+ 3 : \beta_1 (\beta_2)^2  \gamma_1 (\gamma_2)^2: 
+ 6 :\beta_1 \beta_2 \gamma_2 (\partial \gamma_2):  
+ 6 :\beta_1 \beta_2 (\partial \gamma_1) \gamma_2: 
\\ & - 6 :\beta_1 (\partial \beta_2) \gamma_1 \gamma_2:  
+  6 :(\beta_2)^2 (\partial \gamma_2) \gamma_2: 
 - 6 :(\partial \beta_1) \beta_1 (\gamma_1)^2: 
 - 6  :(\partial \beta_1) \beta_2 \gamma_1 \gamma_2: 
\\ & - 6 :(\partial \beta_2) \beta_2 (\gamma_2)^2: +  3 :\beta_2 (\partial^2 \gamma_2):  + 3 : \beta_1 (\partial^2 \gamma_1):  - 12 :(\partial \beta_1)(\partial \gamma_1):  
\\ & - 12 :(\partial \beta_2)(\partial \gamma_2):  + 3 :(\partial^2 \beta_1) \gamma_1: + 3 :(\partial^2 \beta_2) \gamma_2: 
 \\ & - \frac{3}{1 + k} \bigg(X \otimes \big( 2 :(\beta_1)^2 \gamma_1 \gamma_2: + 2 :\beta_1 \beta_2 (\gamma_2)^2:  + 4 :\beta_1 (\partial \gamma_2): - 4 :(\partial \beta_1) \gamma_2: \big)\bigg)      
\\ &  - \frac{3}{1 + k} \bigg(Y \otimes \big( 2 :\beta_1 \beta_2 (\gamma_1)^2: + 2 :(\beta_2)^2 \gamma_1 \gamma_2: + 4 :\beta_2 (\partial \gamma_1):  - 4 :(\partial \beta_2) \gamma_1: \big)\bigg)
 \\ & - \frac{3}{1 + k} \bigg( H \otimes \big(   :(\beta_1)^2 (\gamma_1)^2: + 2 :\beta_1 (\partial \gamma_1):  -  :(\beta_2)^2 (\gamma_2)^2:  - 2 :\beta_2 \partial \gamma_2:  
 \\ &- 2 :(\partial \beta_1) \gamma_1: + 2 :(\partial \beta_2) \gamma_2: \big)\bigg)\bigg).
\end{split}
\end{equation}
Setting $W^4 = W^3_{(1)} W^3$, and comparing the coefficient of $W^3$ in the fourth order pole $W^3_{(3)} W^4$ to \eqref{ope:fourth}, we find that $\lambda = -1/16$. Since $\cC^k(2)$ is easily seen to satisfy the hypotheses of Theorem \ref{refinedsimplequotient}, this proves the claim. \end{proof}

\subsection{Coincidences}
Recall that the specialization $\cC^{k_0}(2)$ of the one-parameter algebra $\cC^{k}(2)$ at $k = k_0$, can be a proper subalgebra of the coset $\text{Com}(V^{k_0}(\gs\gl_2),  V^{k_0+1}(\gs\gl_2) \otimes \cW_{-5/2}(\gs\gl_4, f_{\text{rect}}))$, but this can only occur for rational numbers $k_0 \leq -2$. Similarly, for $n\geq 3$, the specialization $\cC^{k_0}(n)$ can be a proper subalgebra of the coset $\text{Com}(V^{k_0}(\gs\gl_n), V^{k_0+1}(\gs\gl_n) \otimes L_{-1}(\gs \gl_n))$, but this can only happen for rational numbers $k_0 \leq -n$. For $n\geq 2$, we use the same notation $\cC^k(n)$ if $k$ is regarded as a complex number rather than a formal parameter, so that $\cC^k(n)$ always denotes the specialization of the one-parameter algebra at $k\in \mathbb{C}$ even if it is a proper subalgebra of the coset. For all $k\in \mathbb{C}$, we denote by $\cC_{k}(n)$ the simple quotient of $\cC^k(n)$.

\begin{thm} \label{thm:wsl(-2)coinc} For $m \geq 3$, aside from the cases $k = -2,\ -3$, and the cases $c = 0, -2$, all isomorphisms $\cC_k(2) \cong \cW_{k'}(\gs\gl_m, f_{\rm{prin}})$ appear on the following list.
$$ k = -2 - \frac{m-2}{m} ,\  -2 - \frac{2}{m}, \qquad k' = -m + \frac{m}{m-2},\ -m + \frac{m-2}{m}.$$
\end{thm}

\begin{proof} By Corollary \ref{cor:intersections}, aside from the cases $c = 0,-2$, all remaining isomorphisms $\cC_k(2) \cong \cW_{k'}(\gs\gl_m, f_{\text{prin}})$ correspond to intersection points of the curves $V(I)$ and $V(I_m)$, where $I = (\lambda +1/16)$ and $I_m$ is given by \eqref{ideal:wslna}. For each $m \geq 3$, there is exactly one intersection point 
$$(c,\lambda) = \bigg(  -\frac{3 (n-1) (n+2)}{n-2}, \ -\frac{1}{16} \bigg).$$ 
It is immediate that the above isomorphisms all hold, and that our list is complete. \end{proof}

\begin{remark} We have $k < -2$ at the points appearing in Theorem \ref{thm:wsl(-2)coinc}, so at these points, $\cC^k(2)$ can be a proper subalgebra of the coset $\text{Com}(V^{k}(\gs\gl_2),  V^{k+1}(\gs\gl_2) \otimes \cW_{-5/2}(\gs\gl_4, f_{\rm{rect}}))$.\end{remark}

\begin{conj}  \label{conj:wsl(-n)coinc} For $m,n \geq 3$ and $m \neq n$, aside from the cases $k = -n, -n-1$, $k' = -m$, and $c = 0 -2$, all isomorphisms $\cC_k(n) \cong \cW_{k'}(\gs\gl_m, f_{\rm{prin}})$ appear on the following list.
$$k = -n - \frac{m-n}{m},\ -n - \frac{n}{m}, \qquad k' = -m + \frac{m}{m-n},\ -m + \frac{m-n}{m}.$$
\end{conj}

We prove that this conjecture follows from Conjecture \ref{conj:winf}. First, we exclude the points $\displaystyle k = -\frac{1}{2}(n+2),\ -\frac{3n}{2}$, since $\cC^k(n)$ is not obtained as a quotient of $\cW^{K_n}(c,\lambda)$ at these points. Assuming Conjecture \ref{conj:winf}, and applying  Corollary \ref{cor:intersections},  aside from the cases $c = 0,-2$, all remaining isomorphisms $\cC_k(n) \cong \cW_{k'}(\gs\gl_m, f_{\text{prin}})$ correspond to intersection points on the curves \eqref{conj:winfideal} and \eqref{ideal:wslna}. There is exactly one such intersection point
$$(c,\lambda) = \bigg( -\frac{(m-1) (n+1) (n -m + m n)}{m - n},\  -\frac{m - n}{(m-2) (n+ 2) (2 n -2 m + m n)} \bigg).$$ 
It follows that the above isomorphisms all hold, and that our list is complete except for possible coincidences at the excluded points $\displaystyle k = -\frac{1}{2}(n+2),\ -\frac{3n}{2}$.

At both excluded points, $\cC_k(n)$ has central charge $\displaystyle c = \frac{(n+ 1) (3n -2)}{n-2}$, and has a singular vector in weight $3$. Since $\cW_{k'}(\gs\gl_m, f_{\text{prin}})$ has a singular vector in weight $3$ only at $c = 0$ and $\displaystyle c = -\frac{(m-1)(3m+2)}{m+2}$, there are no integers $m\geq 3$ for which $\cW_{k'}(\gs\gl_m, f_{\text{prin}})$ has a singular vector in weight $3$ at the above central charge. Therefore there are no additional coincidences at the excluded points. 

\begin{remark} At the points appearing in Conjecture \ref{conj:wsl(-n)coinc}, we have $k > -n$ whenever $m<n$, so in these cases $\cC^k(n)$ coincides with the coset $\text{Com}(V^{k}(\gs\gl_n), V^{k+1}(\gs\gl_n) \otimes L_{-1}(\gs \gl_n))$. However, if $m > n$, $\cC^k(n)$ can be a proper subalgebra of $\text{Com}(V^{k}(\gs\gl_n), V^{k+1}(\gs\gl_n) \otimes L_{-1}(\gs \gl_n))$. \end{remark}

\appendix
\section{}

In this Appendix we give the explicit OPE relations of the form $W^i(z) W^j(w)$ in $\cW(c,\lambda)$ for $2 \leq i \leq j$ and $i+j = 8$ and $i+j = 9$. Starting from  \eqref{ope:first}-\eqref{ope:fourth}, these are determined uniquely by imposing the Jacobi relations of type $(W^i, W^j, W^k)$ for $i+j+k \leq 11$.

\begin{equation} \label{ope:w2w6} \begin{split} L(z) W^6(w) & \sim -13 c \big(-55 + 16 \lambda (2 + c)\big) (z-w)^{-8} + \big(2100 - 768 \lambda (2 + c)\big) L(w)(z-w)^{-6}  \\ & + \big( 770 - 224 \lambda (2 + c) \big) \partial L(w)(z-w)^{-5} \\ & + 
\bigg( \big(660 - 80 \lambda (13 + 5 c)\big) W^4 + 640 \lambda :LL: 
+   \big( 50 + 40 \lambda (-1 + c) \big) \partial^2 L  \bigg)(w)(z-w)^{-4}
\\ & +\bigg(  \big( 195 - 12 \lambda (17 + 7 c) \big) \partial W^4 + 192 \lambda :(\partial L)L: 
\\ & +  \frac{1}{6} \big(-65 + 4 \lambda (31 + 17 c)\big) \partial^3 L \bigg)(w)(z-w)^{-3}
\\ & +6 W^6 (w)(z-w)^{-2} + \partial W^6 (w)(z-w)^{-1}.\end{split} \end{equation}

\begin{equation} \label{ope:w3w5} \begin{split} W^3(z) W^5(w) & \sim  -c \big(-55 + 16 \lambda (2 + c)\big) (z-w)^{-8} -\frac{4}{3}\big(-175 + 64 \lambda (2 + c)\big) L(w)(z-w)^{-6}  \\ & + \big(110 - 32 \lambda (2 + c)\big) \partial L(w)(z-w)^{-5} \\& + 
\bigg( \big(95 - 16 \lambda (11 + 4 c)\big) W^4 + 128 \lambda :LL: 
+   \big(10 + 8 \lambda (-1 + c)\big) \partial^2 L  \bigg)(w)(z-w)^{-4}
\\ & +\bigg(  64 \lambda :(\partial L)L: + \big(\frac{75}{2} - 4 \lambda (13 + 5 c)\big) \partial W^4 
\\ & + \frac{1}{12} \big(-25 + 8 \lambda (9 + 5 c)\big) \partial^3 L \bigg)(w)(z-w)^{-3}
\\ & +W^6 (w)(z-w)^{-2}
\\ & +\bigg( \frac{1}{3} \partial W^6 + \frac{32 \lambda}{3} :L \partial W^4:  -\frac{64 \lambda}{3} :(\partial L) W^4:  - \frac{16 \lambda}{3} :(\partial^3 L)L: 
\\ & + \big(-\frac{5}{4} + \frac{2}{3} \lambda (1 + c)\big) \partial^3 W^4
+\big(\frac{5}{72} - \frac{1}{45} \lambda (13 + 5 c)\big) \partial^5 L \bigg)(w)(z-w)^{-1}.\end{split} \end{equation}

\begin{equation} \label{ope:w4w4} \begin{split} W^4(z) W^4(w) & \sim  -\frac{1}{3} c \big(-139 + 16 \lambda (2 + c)\big) (z-w)^{-8}  -\frac{4}{3} \big(-125 + 32 \lambda (2 + c)\big) L(w)(z-w)^{-6}  \\ & + \big(\frac{250}{3} - \frac{64}{3} \lambda (2 + c)\big) \partial L(w)(z-w)^{-5} \\& + 
\bigg(   \big(72 - 48 \lambda (3 + c)\big)  W^4 + 128 \lambda  :LL: 
+  \big(10 + 8 \lambda (-1 + c)\big) \partial^2 L  \bigg)(w)(z-w)^{-4}
\\ & +\bigg(  128 \lambda :(\partial L)L: + \big( 36 - 24 \lambda (3 + c)\big) \partial W^4 
\\ & + \frac{1}{18} \big(-35 + 8 \lambda (23 + 13 c)\big)
\partial^3 L \bigg)(w)(z-w)^{-3}
\\ & +\bigg(\frac{4}{5}W^6 + \frac{64 \lambda}{5} :L W^4:  -\frac{288 \lambda}{5}:W^3 W^3: + 32 \lambda :(\partial^2 L)L: + 16 \lambda :(\partial L) \partial L: 
\\ & + \frac{1}{15} \big(35 - 4 \lambda (19 + 11 c)\big) \partial^2 W^4 + \frac{1}{90} \big(-5 + 4 \lambda (7 + 23 c)\big) \partial^4 L \bigg) (w)(z-w)^{-2}
\\ & +\bigg( -\frac{2}{5} \partial W^6 -\frac{32 \lambda}{5} :L \partial W^4: +\frac{288 \lambda}{5} :(\partial W^3) W^3:  -\frac{32 \lambda}{5} :(\partial L) W^4:  
\\ & - \frac{16 \lambda}{3} :(\partial^3 L)L: 
 + \big( \frac{11}{6} - \frac{16 \lambda}{15} -  \frac{8 \lambda c}{15} \big) \partial^3 W^4
+\big( -\frac{1}{4} + \frac{8 \lambda}{25} \big) \partial^5 L \bigg)(w)(z-w)^{-1}.\end{split} \end{equation}

\begin{equation} \label{ope:w2w7} \begin{split} L(z) W^7(w) & \sim  18 \big(4725 - 4784 \lambda (2 + c) + 256 \lambda^2 (26 + 23 c + 5 c^2)\big)W^3(w) (z-w)^{-6} 
\\ & + 14 \big(2225 - 1920 \lambda (2 + c) + 64 \lambda^2 (34 + 31 c + 7 c^2)\big) \partial W^3(w)(z-w)^{-5}
\\ & + \bigg( -5 \big(-357 + 8 \lambda (97 + 31 c)\big) W^5  -640 \lambda \big(-35 + 8 \lambda (2 + c)\big) :L W^3:  
\\ & + \frac{5}{2} \big(805 - 8 \lambda (19 + 27 c) + 128 \lambda^2 (6 + 5 c + c^2)\big) \partial^2 W^3 \bigg)(w)(z-w)^{-4}
\\ &+ \bigg( -\frac{3}{5} \big(-875 + 32 \lambda (39 + 14 c)\big) \partial W^5  -\frac{64}{5} \lambda \big(-425 + 4 \lambda (79 + 29 c)\big) :L \partial W^3: 
\\ & + \frac{288}{5} \lambda \big(5 + 4 \lambda (13 + 3 c)\big)  :(\partial L) W^3:  
\\ & + \big( -\frac{875}{2} + 152 \lambda (5 + 3 c) - \frac{32}{5} \lambda^2 (-23 + 15 c + 8 c^2) \big) \partial^3 W^3 \bigg)(w)(z-w)^{-3}
\\ &+ 7 W^7(w)(z-w)^{-2} + \partial W^7(w)(z-w)^{-1}.\end{split} \end{equation}

\begin{equation} \label{ope:w3w6} \begin{split} W^3(z) W^6(w) & \sim   2 \big(4375 - 4656 \lambda (2 + c) + 256 \lambda^2 (26 + 23 c + 5 c^2)\big)W^3(w) (z-w)^{-6} 
\\ & + 4 \big(975 - 920 \lambda (2 + c) + 32 \lambda^2 (34 + 31 c + 7 c^2)\big)  \partial W^3(w)(z-w)^{-5}
\\ & + \bigg( \big(225 - 8 \lambda (71 + 21 c)\big) W^5  -128 \lambda \big(-29 + 8 \lambda (2 + c)\big) :L W^3:  
\\ & + \big(\frac{665}{2} - 4 \lambda (53 + 41 c) + 64 \lambda^2 (6 + 5 c + c^2)\big) \partial^2 W^3 \bigg)(w)(z-w)^{-4}
\\ &+ \bigg( \big( 84 - \frac{4}{5} \lambda (193 + 63 c) \big) \partial W^5  -\frac{32}{15} \lambda \big(-505 + 4 \lambda (107 + 37 c)\big) :L \partial W^3: 
\\ & -\frac{48}{5} \lambda \big(-55 + 4 \lambda (-9 + c)\big)  :(\partial L) W^3:  
\\ & + \big( -70 + \lambda (\frac{490}{3} + 82 c) - \frac{16}{15} \lambda^2 (-29 + 20 c + 9 c^2)\big) \partial^3 W^3 \bigg)(w)(z-w)^{-3}
\\ &+  W^7(w)(z-w)^{-2} + \bigg(\frac{2}{7}\partial W^7 + \frac{496 \lambda}{35} :L \partial W^5: -\frac{248 \lambda}{7} :(\partial L) W^5: 
\\ & + \frac{192 \lambda}{7} : W^3 \partial W^4:  -\frac{256 \lambda}{7} :(\partial W^3) W^4: + \frac{1536 \lambda^2}{35} :(\partial L) L W^3: -\frac{1024 \lambda^2}{35} :LL \partial W^3: 
 \\ & + \frac{8}{35} \lambda \big(-455 + 4 \lambda (135 + 41 c)\big) :(\partial^3 L) W^3:  -\frac{192}{35} \lambda \big(5 + 2 \lambda (-3 + c)\big) :(\partial^2 L) \partial W^3: 
 \\ & + \frac{12}{35} \lambda \big(95 + 8 \lambda (-3 + c)\big) :(\partial L) \partial^2 W^3:  +\frac{8}{105} \lambda \big(-455 + 8 \lambda (-25 + 7 c)\big) :L \partial^3 W: 
 \\ & + \big(-2 + \frac{2}{35} \lambda (17 + 21 c)\big)  \partial^3 W^5 
 \\ & + \frac{1}{105 } \big(175 - \lambda (149 + 205 c) + 24 \lambda^2 (11 + c^2)\big) \partial^5 W^3 \bigg)(w)(z-w)^{-1}.\end{split} \end{equation}

\begin{equation} \label{ope:w4w5} \begin{split} W^4(z) W^5(w) & \sim  \big( 4950 - 4928 \lambda (2 + c) + 256 \lambda^2 (26 + 23 c + 5 c^2) \big)W^3(w) (z-w)^{-6} 
\\ & + \frac{2}{3} \big(3625 - 3600 \lambda (2 + c) + 128 \lambda^2 (34 + 31 c + 7 c^2)\big)  \partial W^3(w)(z-w)^{-5}
\\ & + \bigg( \big(  140 - 8 \lambda (49 + 13 c)  \big) W^5  -128 \lambda \big(-23 + 8 \lambda (2 + c)\big) :L W^3:  
\\ & + \big( \frac{525}{2} - 4 \lambda (87 + 55 c) + 64 \lambda^2 (6 + 5 c + c^2)   \big)  \partial^2 W^3 \bigg)(w)(z-w)^{-4}
\\ & + \bigg( \big(64 - \frac{16}{5} \lambda (51 + 14 c)\big) \partial W^5  -\frac{32}{15} \lambda \big(-485 + 16 \lambda (34 + 11 c)\big) :L \partial W^3: 
\\ & -\frac{48}{5} \lambda \big(-145 + 16 \lambda (2 + 3 c)\big) :(\partial L) W^3:  + \frac{1}{30} \big(-1575 + 40 \lambda (127 + 43 c)
\\ & - 256 \lambda^2 (-4 + 3 c + c^2) \big) \partial^3 W^3 \bigg)(w)(z-w)^{-3} + \bigg( \frac{2}{3} W^7   +\frac{64 \lambda}{3} :LW^5:  -128 \lambda :W^3 W^4: 
\\ & -\frac{32}{5} \lambda \big(-95 + 2 \lambda (65 + 19 c)\big) :(\partial^2 L)W^3:  -\frac{32}{15} \lambda \big(-125 + 2 \lambda (65 + 19 c)\big) :(\partial L) \partial W^3: 
\\ & -\frac{32}{15} \lambda \big(35 + 4 \lambda (-25 + c)\big) :L\partial^2 W^3: + \big( \frac{5}{2} - \frac{8}{5} \lambda (5 + 2 c) \big) \partial^2 W^5 
\\ &+ \frac{1}{180} \big(-175 + 80 \lambda (33 + c) - 64 \lambda^2 (65 - 6 c + c^2)\big) \partial^4 W^3 \bigg)(w)(z-w)^{-2} 
\\ &+ \bigg(\frac{2}{7}\partial W^7 + \frac{384 \lambda}{35} :L \partial W^5:  + \frac{32 \lambda}{7} :(\partial L) W^5: + \frac{192 \lambda}{7} : W^3 \partial W^4: 
\\ & -\frac{1152 \lambda}{7} :(\partial W^3) W^4: -\frac{9216 \lambda^2}{35} :(\partial L) L W^3: + \frac{6144 \lambda^2}{35} :LL \partial W^3: 
\\ & -\frac{8}{105} \lambda \big(-2345 + 8 \lambda (389 + 145 c)\big) :(\partial^3 L) W^3: 
\\ & -\frac{32}{35} \lambda \big(-145 + 8 \lambda (13 + 5 c)\big) :(\partial^2 L) \partial W^3: 
+ \frac{8}{35} \lambda \big(-295 + 8 \lambda (13 + 5 c)\big) :(\partial L) \partial^2 W^3: 
\\ & + \frac{16}{35} \lambda (-245 + 8 \lambda (11 + 7 c)\big) :L \partial^3 W: 
+ \big( -\frac{17}{6} + \frac{4}{105} \lambda (43 + 35 c) \big)  \partial^3 W^5 
\\ & + \frac{1}{420} \big(1925 - 16 \lambda (-87 + 130 c) + 64 \lambda^2 (-29 + 5 c^2)\big) \partial^5 W^3 \bigg)(w)(z-w)^{-1}.\end{split} \end{equation}

\end{document}